\documentclass[a4paper,fleqn]{cas-sc}

\usepackage[utf8]{inputenc}
\usepackage[english]{babel}
\usepackage[authoryear,longnamesfirst]{natbib}
\usepackage{graphicx}
 \graphicspath{{./figures/}}
\usepackage[justification=centering]{caption}
\usepackage{subcaption}
\usepackage{mathtools}
\usepackage{amsmath}
\usepackage{amsthm,amsfonts,amstext,amssymb}
\usepackage{bbm}
\usepackage{enumitem}   
\usepackage{indentfirst}
\usepackage[font=small,skip=0pt]{caption}
\usepackage{longtable}
\usepackage{float}
\restylefloat{table}
\usepackage{xcolor}
\usepackage{xr}
\usepackage{tikz}
\usetikzlibrary{shapes,arrows,positioning,calc}
\usepackage{standalone}
\usepackage[normalem]{ulem}
\usepackage{algorithm}
\usepackage{algpseudocode}
\usepackage{soul,xcolor}
\setstcolor{red}
\usepackage{tabularx}

\newtheorem{theorem}{Theorem}

\newtheorem{proposition}{Proposition}
\theoremstyle{definition}
\newtheorem{definition}{Definition}
\theoremstyle{remark}
\newtheorem{remark}{Remark}

\newtheorem{example}{Example}

\algrenewcommand\algorithmicrequire{\textbf{Input:}}
\newcommand{\mpcommentout}[1]{}
\newcommand{\mc}{\mathcal}
\newcommand{\kscomment}[1]{{\color{red} #1}}

\newcommand{\mpcomment}[1]{{\color{black} #1}}
\newcommand{\mpmargin}[1]{\marginpar{\color{blue}\footnotesize [MP]:
#1}}
\newcommand{\ra}{\rightarrow}

\newcommand{\eps}{\varepsilon}
\newcommand{\R}{\mathbb{R}}
\newcommand{\Z}{\mathbb{Z}}
\newcommand{\N}{\mathbb{N}}
\newcommand{\E}[1]{\mathbb{E}\left[ #1 \right]}

\DeclareMathOperator*{\argmax}{arg\,max}
\DeclareMathOperator*{\argmin}{arg\,min}
\newcommand{\vehiclelength}{L}
\newcommand{\timeheadway}{h}
\newcommand{\standstilldist}{S_0}
\newcommand{\speedlim}{V_f}

\newcommand{\maxaccel}{a_{\max}}
\newcommand{\minaccel}{a_{\min}}
\newcommand{\maxjerk}{J_{\max}}
\newcommand{\Speedmode}{\text{speed tracking}}
\newcommand{\Safetymode}{\text{safety}}

\newcommand{\stabconst}[1]{c_{#1}}
\newcommand{\StateofDySys}{X}
\newcommand{\stabexpconst}[1]{\alpha_{#1}}

\newcommand{\Perimeter}{P}
\newcommand{\numvehicles}{n}

\newcommand{\numramps}{m}
\newcommand{\rampindic}{\mathcal{M}}

\newcommand{\numcells}{n_c}
\newcommand{\numaccslots}[1]{n_{#1}}
\newcommand{\addlslot}[1]{n_{#1}}

\newcommand{\timestep}{\tau}

\newcommand{\routingmatrix}{R}

\newcommand{\Arrivalrate}[1]{\lambda_{#1}}
\newcommand{\Numarrivals}[1]{A_{#1}}
\newcommand{\Cumularrivals}[2]{\bar{A}_{#1}(#2)}
\newcommand{\Numdepartures}[1]{D_{#1}}
\newcommand{\Cumuldepartures}[1]{D_{#1}}

\newcommand{\Queuelength}[1]{Q_{#1}}
\newcommand{\Nodedegree}[1]{N_{#1}}

\newcommand{\NewNodedegree}[1]{\Tilde{N}_{#1}}
\newcommand{\Avgload}[1]{\rho_{#1}}
\newcommand{\Nodedest}[1]{M_{#1}}
\newcommand{\Lyap}[1]{V(#1)}
\newcommand{\Cyclelength}{T_{\text{cyc}}}
\newcommand{\StateofMC}{Y}
\newcommand{\StateofMCExp}{Z}
\newcommand{\LyapValue}{l}
\newcommand{\StateSpace}{\mathcal{Y}}
\newcommand{\MaxNumArrival}[3]{\tilde{A}_{#1}^{#2}(#3)}
\newcommand{\Releasedist}{K}
\newcommand{\Period}{\Delta}
\newcommand{\DesDecreaseconst}{\gamma_{1}}
\newcommand{\ReleasetimeInc}[1]{\theta_{#1}}
\newcommand{\ReleasetimeDec}[1]{\alpha_{#1}}
\newcommand{\ReleasetimeDecConst}[1]{\gamma_{#1}}
\newcommand{\Releasetime}[1]{g_{#1}}
\newcommand{\ReleasetimeAdj}[1]{\beta_{#1}}
\newcommand{\Interarrivals}[1]{\psi_{#1}}
\newcommand{\Tempty}{T_{\text{free}}}
\newcommand{\updateperiod}{T_{\text{per}}}

\newcommand{\Tmax}{T_{\text{max}}}

\newcommand{\GenRM}{\text{DRR}}
\newcommand{\ConRM}{\text{DSG}}

\newcommand{\FCQRM}{\text{FCQ}}
\newcommand{\DisDRRRM}{\text{DisDRR}}

\ExplSyntaxOn
\keys_set:nn {stm/mktitle}{nologo}
\ExplSyntaxOff

\begin{document}
\let\WriteBookmarks\relax
\let\printorcid\relax
\def\floatpagepagefraction{1}
\def\textpagefraction{.001}
\shorttitle{}
\shortauthors{M.Pooladsanj et~al.}
%\begin{frontmatter}

\title[mode = title]{Ramp Metering to Maximize Freeway Throughput under Vehicle Safety Constraints}                      

\author[1]{Milad Pooladsanj}
\cortext[cor1]{Corresponding author}
\cormark[1]
\ead{pooladsa@usc.edu}
\address[1]{Department of Electrical Engineering, University of Southern California, Los Angeles, CA 90007 USA}

\author[2]{Ketan Savla}

\ead{ksavla@usc.edu}
\address[2]{Department of Civil and Environmental Engineering, University of Southern California, Los Angeles, CA 90089 USA}

\author[1]{Petros A. Ioannou}
\ead{ioannou@usc.edu}

\nonumnote{K. Savla has financial interest in Xtelligent, Inc. This work was supported in part by NSF CMMI 1636377 and METRANS 19-17.}

\begin{abstract}
Ramp Metering (RM) is one of the most effective control techniques to alleviate freeway congestion. We consider RM at the microscopic level subject to vehicle following safety constraints for a freeway with arbitrary number of on- and off-ramps. The arrival times of vehicles to the on-ramps, as well as their destinations are modeled by exogenous stochastic processes. Once a vehicle is released from an on-ramp, it accelerates towards the free flow speed if it is not obstructed by another vehicle; once it gets close to another vehicle, it adopts a safe gap vehicle following behavior. The vehicle exits the freeway once it reaches its destination off-ramp. We design traffic-responsive RM policies that maximize the \emph{throughput}. For a given routing matrix, the throughput of a RM policy is characterized by the set of on-ramp arrival rates for which the queue sizes at all the on-ramps remain bounded in expectation. The proposed RM policies operate under vehicle following safety constraints, where new vehicles are released only if \mpcomment{there is sufficient gap between vehicles on the mainline at the moment of release}. Furthermore, the proposed policies work in synchronous \emph{cycles} during which an on-ramp does not release more vehicles than the number of vehicles waiting in its queue at the start of the cycle. All the proposed policies are \emph{reactive}, meaning that they only require real-time traffic measurements without the need for demand prediction. \mpcomment{However, they differ in how they use the traffic measurements. In particular, we provide three different mechanisms under which each on-ramp either}: (i) pauses release for a time interval before the next cycle begins, or (ii) adjusts the release rate during a cycle, or (iii) adopts a conservative safe gap criterion for release during a cycle. The throughput of the proposed policies is characterized by studying stochastic stability of the induced Markov chains, and is proven to be maximized when the merging speed at all the on-ramps equals the free flow speed. Simulations are provided to illustrate the performance of our policies and compare with a well-known RM policy from the literature \mpcomment{which relies on local traffic measurements}.
\end{abstract}

%\item Research highlights item 1
%\item Research highlights item 2
%\item Research highlights item 3


\maketitle

\section{Introduction}\label{Section: Introduction}
\par
Freeway congestion is caused by high demand competing to use the limited supply of the freeway system.
%%%%%%%%%%%%%%%%%% Commented 01/18/22 %%%%%%%%%%%%%%%%%%%%%%%%%
\mpcommentout{
Merging on freeways is known as one of the main sources of initiating (or further intensifying) the freeway congestion (\citet{rios2016automated}).
}
%%%%%%%%%%%%%%%%%% Commented 01/18/22 %%%%%%%%%%%%%%%%%%%%%%%%%
One of the most effective tools to combat this congestion is Ramp Metering (RM), which involves controlling the inflow of traffic to the freeway in order to balance the supply and demand and ultimately improve some measure of performance (\citet{papageorgiou2002freeway, rios2016automated, lee2006quantifying}). 
%%%%%%%%%%%%%%%%%% Commented 01/18/22 %%%%%%%%%%%%%%%%%%%%%%%%%
\mpcommentout{
The main objective of an RM strategy is to improve the mainstream flow (possibly at the scale of a freeway network) while preventing the ramp queues from building up.}
%%%%%%%%%%%%%%%%%% Commented 01/18/22 %%%%%%%%%%%%%%%%%%%%%%%%%
\mpcomment{Typically, macroscopic traffic flow models are used to design RM policies. However, these models lack the necessary resolution to ensure the existence of safe merging gaps for vehicles entering from on-ramps.} An alternative is the microscopic-level approach which is limited to heuristics or simulations beyond an isolated on-ramp (\citet{rios2016automated,sun2006set}). \mpcomment{The objective of this paper is to systematically design RM policies at the microscopic level and analyze their system-level performance.}
\par
There is an overwhelming body of work on designing RM policies using macroscopic traffic flow models. We review them here only briefly; interested readers are referred to \citet{papageorgiou2002freeway} for a comprehensive review. RM policies can be generally classified as fixed-time or traffic-responsive (\citet{papageorgiou2002freeway}). Fixed-time policies such as \citet{wattleworth1965peak} are fine-tuned offline and operate based on historical traffic data. Due to the uncertainty in the traffic demand and the absence of real-time measurements, these policies would either lead to congestion or under-utilization of the capacity of the freeway (\citet{papamichail2010coordinated}). Traffic-responsive policies, on the other hand, use real-time measurements. These policies can be further sub-classified as local or coordinated depending on whether the on-ramps make use of the measurements obtained from their vicinity (local) or other regions of the freeway (coordinated) (\citet{papageorgiou2002freeway, papageorgiou2003review}). A well-known example of a local policy is ALINEA (\citet{papageorgiou1991alinea}) which has been shown, both analytically and in practice, to yield a (locally) good performance. A caveat in employing local policies is that there is no guarantee that they can improve the overall performance of the freeway while providing a fair access to the freeway from different on-ramps (\citet{papamichail2010coordinated}). This motivates the study of coordinated policies such as the ones considered in \citet{papageorgiou1990modelling, papamichail2010coordinated, gomes2006optimal}.
\par
Despite the wide use of macroscopic traffic flow models in RM design, this approach has some limitations. \mpcomment{One of its major limitations is that it does not capture the specific safety requirements for vehicles entering the freeway from the on-ramps. As a result, it does not guarantee the existence of sufficient gaps between vehicles on the mainline to safely accommodate the merging vehicles. Therefore, the merging vehicles may sometimes be forced to merge even if the gap is not safe, or they may not be able to merge at all.} To address this limitation, an alternative is the microscopic-level approach which allows to include the safety requirements in the design of RM to optimize the freeway performance (\citet{rios2016automated}). In addition, it allows to naturally incorporate vehicle automation, which can compensate for human errors, or Vehicle-to-Vehicle (V2V) and Vehicle-to-Infrastructure (V2I) communication, which can provide accurate traffic measurements, in the RM design (\citet{Sugiyama2008TrafficJW, stern2018dissipation, zheng2018smoothing, pooladsanj2020vehicle, pooladsanj2021vehicle, rios2016automated, li2013survey, lioris2017platoons}). \mpcommentout{Indeed, another approach is to design RM at the macroscopic-level and use vehicle cooperation to guarantee vehicle following safety during merging. However, such decoupling can degrade the freeway performance.}
\par
The microscopic-level RM problem involves determining the order in which vehicles should cross the merging point of the mainline and the on-ramp, i.e., determining their so-called \emph{merging sequence}. Previous works have mostly focused on isolated on-ramps to determine the merging sequence, employing various techniques such as minimizing the average time taken by the vehicles to cross the merging point (\citet{raravi2007merge}) or simple heuristic rules such as the first-in first-out scheme (\citet{rios2016automated}). However, there has been limited analysis of the broader impact of these designs on the overall freeway performance. For example, do they optimize certain system-level performance? Indeed it is possible that a \emph{greedy} policy, i.e., a policy where each on-ramp acts independently, limits entry from the downstream on-ramps and thereby creating long queues or congestion. 
\par
Motivated by the aforementioned gaps and increasing vehicle connectivity and automation, the primary objective of this paper is to systematically design traffic-responsive RM policies at the microscopic level and analyze their system-level performance. The key performance metric to evaluate a RM policy in this paper is its \emph{throughput}. For a given routing matrix, the throughput of a RM policy is characterized by the set of on-ramp arrival rates for which the queue sizes at all the on-ramps remain bounded in expectation. Roughly, the throughput determines the highest traffic demand that a RM policy can handle without creating long queues at the on-ramps. For example, in case of an isolated on-ramp, the throughput of any policy cannot exceed the on-ramp's \emph{flow capacity}. The throughput is tightly related to, but different than, the notion of total travel time, which is a popular metric in the literature (\citet{gomes2006optimal, hegyi2005model}). In particular, if the demand exceeds a policy's throughput, then the expected total travel time of an arriving vehicle, from the moment it joins an on-ramp queue until it leaves the freeway, will increase over time, i.e., the later the vehicle arrives, the more total travel time it will experience on average. On the other hand, if the demand does not exceed the throughput, then the expected total travel time is stabilized. In other words, the throughput of a policy determines the demand threshold at which the expected total travel time transitions from being stabilized to being increasing over time. Therefore, it is desirable to design a policy that achieves the \emph{maximum} possible throughput among all policies; that is, the set of demands for which the policy can stabilize the expected total travel time is at least as large as any other policy.
\par
\par
In this paper, we design traffic-responsive RM policies that can achieve the maximum possible throughput subject to vehicle following safety constraints. \mpcomment{The main idea is to design policies in which each on-ramp keeps a balance in using the safe merging gaps on the mainline without under-utilizing or over-utilizing them.} More specifically, all the proposed RM policies operate under vehicle following safety constraints, where the on-ramps release new vehicles only if \mpcomment{there is sufficient gap between vehicles on the mainline at the moment of release}. Furthermore, the proposed policies work in synchronous \emph{cycles} during which an on-ramp does not release more vehicles than the number of vehicles waiting in its queue, i.e., its queue size, at the start of the cycle. The proposed policies use \mpcomment{real-time traffic measurements obtained by V2I communication, but do not require prior knowledge of the on-ramp arrival rates or the routing matrix, i.e., they are reactive. However, they differ in how they use the traffic measurements. In particular, we provide three different mechanisms under which each on-ramp either:} (i) pauses release for a time-interval before the next cycle starts, or (ii) adjusts the time between successive releases during a cycle, or (iii) adopts a conservative safe gap criterion for release during a cycle. The throughput of these policies is characterized by studying stochastic stability of the induced Markov chains, and is proven to be maximized when the merging speed at all the on-ramps equals the free flow speed. When the merging speeds are less than the free flow speed, the safety considerations may lead to a drop in the traffic flow. However, this effect, which is similar to the capacity drop phenomenon (\citet{hall1991freeway}), is shown to be less significant compared to a well-known macroscopic-level RM policy \mpcomment{that relies on local traffic measurements obtained by roadside sensors. The purpose of this comparison is to show how much V2I communication can improve performance compared with
a state-of-the-art RM policy which relies only on roadside sensors.}
\par
In summary, the three main contributions of this paper are as follows: 
\begin{itemize}
\item Designing traffic-responsive RM policies at the microscopic level to understand the interplay of safety, connectivity, automation protocols, and the throughput. The proposed policies are reactive, meaning that they only require real-time traffic measurements without requiring any prior knowledge of the demand (thus, they can adapt to time-varying demand). They obtain the traffic measurements by V2I communication.
%but do not require vehicle autonomy, making them suitable for mixed-autonomy scenarios. 

\item \mpcomment{Designing RM policies that operate under vehicle following safety constraints, where the on-ramps release new vehicles only if there is sufficient gap between vehicles on the mainline. Each on-ramp keeps a balance in using the safe merging gaps without under-utilizing or over-utilizing them. By systematically including safety in the RM design, the risk of collisions at merging bottlenecks is reduced without compromising the throughput.} 
%Moreover,
  
\item \mpcomment{Providing (i) an outer-estimate for the throughput of any RM policy for freeways with arbitrary number of on- and off-ramps, and (ii) an inner-estimate for the throughput of the proposed microscopic-level RM policies. The outer-estimate serves as the network equivalent of the flow capacity of an isolated on-ramp. In other words, it establishes an upper-limit on the achievable throughput of any RM policy (microscopic or macroscopic). Comparing the outer- and inner-estimates shows that the proposed policies are able to maximize throughput. That is, their throughput is at least as good as any other policy, including those with prior knowledge of the demand.} 
\mpcommentout{Throughput measures the highest traffic demand that a RM policy can handle without creating long on-ramp queues. Equivalently, it determines the demand threshold at which the expected total travel time transitions from being stabilized to being increasing over time.}
%\item Design of policies that  yet maximize the freeway throughput. In other words, the proposed policies only require real-time traffic measurements, and do not need any information about the arrival rates or the routing matrix (thus, they can adapt to time-varying demand); yet, their throughput is at least as good as any other policy, including the policies that know the demand in advance. 

\end{itemize}

The rest of the paper is organized as follows: in Section~\ref{sec:notations}, we gather all the necessary notations used in the paper. In Section~\ref{Section: Problem Formulation}, we state the problem formulation, vehicle-level rules, demand model, and a summary of the RM policies studied in this paper. The design and performance analysis of the proposed RM policies takes place in Section~\ref{sec:large-merge-speed}. We simulate the proposed policies in Section \ref{Section: Simulation} and conclude the paper in Section \ref{Section: Conclusion}. 
%%%%Mar15%%%%Congestion caused by vehicles has costed out ... \par
%%%%Mar15%%%%There is a heavy literature on the capabilities of CAVs in order to resolve congestion ... \par
%%%%Mar15%%%%There are two important constraints when dealing with CAVs: safety, and limited physical transportation infrastructure. \par
%%%%Mar15%%%%The majority of research in academia and industry are not directly concerned with studying the impact of CAVs on the mobility \citet{lioris2017platoons}. \par
%%%%Mar15%%%%Ramp metering is a promising tool for freeway congestion control \citet{papageorgiou2002freeway}.

\section{Notations}\label{sec:notations}
For convenience, we gather the main notations used in the paper in Table~\ref{table:notation}.
%\begin{table}[H]
%\begin{center}
\begin{longtable}{p{.3\textwidth}  p{.65\textwidth}}
 \hline
Notation & Description \\
 \hline
 \rule{0pt}{3ex}
 \emph{Mathematical notations} & \\
 $\N$, $\N_{0}$, $\R$ & Set of positive integers ($\N$), non-negative integers ($\N_{0}$), and real numbers ($\R$) \\
 $[k]$ ($k \in \N$) & The set $\{1, \ldots, k\}$ \\
& \\
 \emph{Network} & \\
 $\Perimeter$ & Length of the freeway \\
 $\numramps$ & Number of on- and off-ramps \\
 $n$ & Number of vehicles on the mainline and acceleration lanes \\
 %$\Queuelength{i}$ & Vector of destination off-ramps for the vehicles waiting at on-ramp $i$, arranged in the order of their arrival \\
 $\Queuelength{i}$, $\Queuelength{}$
 %, $\overline{\Queuelength{}}$
 & Number of vehicles at on-ramp $i$, i.e., the queue size, ($\Queuelength{i}$), and vector of queue sizes ($\Queuelength{}$)%, and the long-run expected queue size ($\overline{\Queuelength{}}$)
 \\
%$\numcells$, $\numaccslots{i}$, $\numaccslots{a}$, $\numaccslots{m}$ & Number of mainline slots ($\numcells$), number of $i$-th acceleration lane slots ($\numaccslots{i}$), total number of acceleration lane slots ($\numaccslots{a}$), and total number of slots in all the merging areas ( $\numaccslots{m}$) \\ 
%TTT$_n$ & Average total travel time of the first $n$ completed trips \\
%TTC & Time-to-collision \\
& \\
 \emph{Vehicles} & \\
$\vehiclelength$, $\timeheadway$, $\standstilldist$, $\speedlim$ & Vehicle length ($\vehiclelength$), safe time headway constant ($\timeheadway$), standstill gap ($\standstilldist$), and free flow speed ($\speedlim$) \\
$\minaccel$%, $\maxaccel$, $\maxjerk$ 
& Minimum deceleration %($\minaccel$), maximum acceleration ($\maxaccel$), and maximum jerk ($\maxjerk$) 
\\
$\timestep$ & Minimum safe time headway between the front bumpers of two vehicles at the free flow speed  \\
$\timestep_i$ & Minimum time headway between the front bumpers of two vehicles at the free flow speed such that a vehicle from on-ramp $i$ can safely merge between them \\
 $p_e$, $v_e$, $a_e$, $y_e$, $S_e$  & Ego vehicle's location ($p_e$), speed ($v_e$), acceleration ($a_e$), distance to the leading vehicle ($y_e$), and safety distance required to avoid collision ($S_e$) \\ 
 $\hat{v}_e$, $\hat{y}_e$, $\hat{S}_e$ & Ego vehicle's predicted speed ($\hat{v}_e$), predicted gap with respect to its virtual leading vehicle ($\hat{y}_e$), and predicted safety distance between the two vehicles ($\hat{S}_e$)\\
 %$\delta_e$, $\hat{\delta}_e$ & Ego vehicle's maximum error ($\delta_e$) and maximum predicted error ($\hat{\delta}_e$) in the relative spacing \\
 $I_e$ & A binary variable which is equal to one if the ego vehicle is in the $\Safetymode$ mode, and zero otherwise \\
 $x_e$, $\StateofDySys$ & State of the ego vehicle ($x_e$), and state of all the vehicles ($\StateofDySys$) \\
 $t_m|t$ & Predicted time of crossing the merging point calculated at time $t$ \\
$v_l$, $\hat{v}_{\hat{l}}$ & Speed of the leading vehicle ($v_l$) and predicted speed of the virtual leading vehicle ($\hat{v}_{\hat{l}}$) %of the ego vehicle 
\\
%$\hat{y}_{\hat{f}}$ & Predicted gap between the ego vehicle and its virtual following vehicle\\
% $\Tempty$ & Time to reach the free flow state if no on-ramp releases a vehicle \\
  & \\
  \emph{Demand} & \\
$\Arrivalrate{i}$, $\Arrivalrate{}$ & Arrival rate at on-ramp $i$ ($\Arrivalrate{i}$) and vector of arrival rates ($\Arrivalrate{}$) \\
$\routingmatrix_{ij}$, $\routingmatrix$ & fraction of arrivals at on-ramp $i$ that want to exit from off-ramp $j$ ($\routingmatrix_{ij}$) and routing matrix ($\routingmatrix$) \\
$\Tilde{\routingmatrix}_{ij}$, $\Tilde{\routingmatrix}$ & Fraction of arrivals at on-ramp $i$ that need to cross link $j$ in order to reach their destination ($\Tilde{\routingmatrix}_{ij}$) and cumulative routing matrix ($\Tilde{\routingmatrix}$)\\
$\Avgload{j}$, $\Avgload{}$ & Total load induced on link $j$ by all the on-ramps ($\Avgload{j}$) and maximum load ($\Avgload{}$) \\
  & \\
\emph{Ramp metering} & \\
$U_{\pi, \routingmatrix}$ & Under-saturation region of a RM policy $\pi$ given the routing matrix $\routingmatrix$ \\
%$C$ & V2I communication cost of a RM policy \\
%$c_1$, $c_2$ & Contribution of queue size to the communication cost under the Renewal policy ($c_1$) and the $\GenRM$, $\DisDRRRM$, $\ConRM$, and $\FCQRM$ policies ($c_2$) \\ 
$\Cyclelength$ & Constant cycle length coefficient \\
%$\Releasetime{i}(t)$ (also $\Releasetime{}(t)$) & Minimum time gap between successive releases required by on-ramp $i$ at time $t$ \\ 
%$\StateofDySys_f$, $\StateofDySys_f^{i}$ & Distance to the free flow state ($\StateofDySys_f$) and contribution of the vehicles between the $i$-th and $(i+1)$-th on-ramps to $\StateofDySys_f$ ($\StateofDySys_f^{i}$) \\
%$\updateperiod$, $\Tmax$, $\DesDecreaseconst$, $\ReleasetimeDecConst{2}$, $\ReleasetimeInc{}^{\circ}$, $\ReleasetimeInc{i}^{\circ}$, $\ReleasetimeAdj{}$ & Design constants in the $\GenRM$ and $\DisDRRRM$ policies \\
%$K_i$ ($i=1,2,3$) & Dynamic safe gaps in the $\ConRM$ policy \\
$\Cumuldepartures{\pi,p}(t)$ & Cumulative number of vehicles that has crossed point $p$ on the mainline up to time $t$ under the RM policy $\pi$ \\
%$r(t)$, $o(t)$, $\hat{o}$ & On-ramp outflow ($r(t)$) and mainline occupancy ($o(t)$) at time $t$, and desired occupancy ($\hat{o}$) \\
%$K_r$ & Design constant in the ALINEA policy \\
 \hline
 
 \caption{\sf The main notations used in the paper.}\label{table:notation}
\end{longtable}
%\end{center}
%\end{table}

%%%%%%%%%%%%%%%%%%%%%%%%%%%%%%%%%%%%%%%%%%%%%%%%%%%%%%%%%%%%%%%%%%%%%%%%%%%%%
%%%%%%%%%%%%%%%%%%%%%%%%%%%%%%%%%%%%%%%%%%%%%%%%%%%%%%%%%%%%%%%%%%%%%%%%%%%%%
%%%%%%%%%%%%%%%%%%%%%%%%%%%%%%%%% Section %%%%%%%%%%%%%%%%%%%%%%%%%%%%%%%%%%%
%%%%%%%%%%%%%%%%%%%%%%%%%%%%%%%%%%%%%%%%%%%%%%%%%%%%%%%%%%%%%%%%%%%%%%%%%%%%%
%%%%%%%%%%%%%%%%%%%%%%%%%%%%%%%%%%%%%%%%%%%%%%%%%%%%%%%%%%%%%%%%%%%%%%%%%%%%%
\section{Problem Formulation}\label{Section: Problem Formulation}
In this paper, we consider a single freeway lane of length $\Perimeter$ with multiple single-lane on- and off- ramps, where we assume that \mpcomment{there is a ramp meter at every on-ramp}. We distinguish between two settings: the first setting is a ring road configuration such as the one in Figure~\ref{fig: Motivating Example}. The main reason for choosing a ring road is that it can capture the creation and dissipation of stop-and-go waves (\citet{Sugiyama2008TrafficJW, stern2018dissipation}). The second setting is the more natural straight road configuration such as the one in Figure~\ref{Fig: line network}. Note that in Figure~\ref{Fig: line network}, the traffic light at the upstream entry point indicates that we have control over any previous on-ramp. We describe the problem formulation and the main results for the ring road configuration. The extension of our results to the straight road configuration is discussed in Section~\ref{Section: Straight Line}.
\par
The number of on- and off-ramps is $\numramps$; they are placed alternately and are numbered in an increasing order along the direction of travel, such that, for all $i \in [\numramps]$, off-ramp $i$ comes after on-ramp $i$ \footnote{By this numbering scheme, we do not mean to imply that the location of an on-ramp must be close to the next off-ramp as this is not usually the case in practice; see Figure~\ref{Fig: model}. In fact, our setup is quite flexible and can also deal with cases where the number of on- and off-ramps are not the same. For simplicity of notation, we have used the same number of on- and off-ramps in this paper.}. The section of the mainline between the $i$-th on- and off- ramps is referred to as \emph{link} $i$. Vehicles arrive at the on-ramps from outside the network and join the on-ramp queues. We assume a point queue model for vehicles waiting at the on-ramps, with the queue on an on-ramp co-located with its ramp meter. The on-ramp vehicles are released into the mainline by the ramp meters installed at each on-ramp. Upon release from the ramp, each vehicle follows the standard speed and safety rules until it reaches its destination off-ramp at which point it exits the network without creating any obstruction for the upstream vehicles. Each vehicle is equipped with a V2I communication system, which is used to transmit its state, e.g., speed, to the on-ramp control units, and receive information about the state of the nearby vehicles in merging areas. The exact information communicated will be specified in later sections.
\par
\mpcomment{The objective is to design RM policies that operate under vehicle following safety constraints and analyze their performance.} The performance of a RM policy is evaluated in terms of its \emph{throughput} defined as follows: let $\Arrivalrate{i}$ be the \emph{arrival rate} to on-ramp $i$, $i \in [\numramps]$, and $\Arrivalrate{} := [\Arrivalrate{i}]$ be the vector of arrival rates. Let $\routingmatrix = [\routingmatrix_{ij}]$ be the \textit{routing matrix}, where $\routingmatrix_{ij}$ specifies the fraction of arrivals to on-ramp $i$ that want to exit from off-ramp $j$. For a given routing matrix $\routingmatrix$, the \emph{under-saturation region} of a RM policy is defined as the set of $\Arrivalrate{}$'s for which the queue sizes at all the on-ramps remain bounded in expectation \footnote{Note that $\Arrivalrate{}$ and $\routingmatrix$ are macroscopic quantities. In order to specify vehicle arrivals and their destination at the microscopic level, a more detailed (probabilistic) demand model is required; see Section~\ref{sec:demand}. The expected queue size is defined with respect to the probabilistic demand model.}. The boundary of the under-saturation region is called the throughput. We are interested in finding RM policies that ``maximize" the throughput for any given $\routingmatrix$. We will formalize this in Section~\ref{sec:demand}.
\par
The remainder of this section is organized as follows: in Section~\ref{sec:vehicle-control}, we discuss the \mpcomment{vehicle following safety constraints}. We specify the demand model and formalize the notion of throughput in Section~\ref{sec:demand}. We summarize the RM policies considered in this paper in Section~\ref{Section: RM rules}.
%%%%%%%%%%%%%%%%% Commented (06/07/22) %%%%%%%%%%%%%%%%%%%
\mpcommentout{
Vehicles arrive to the network according to exogenous processes and form queues at the on-ramps. The rates at which vehicles are released from the queues are determined by ramp meters installed at the on-ramps. Once a vehicle is released, the vehicle controller takes charge of the vehicle trajectory until it exits from an off-ramp. The objective is to design suitable ramp metering controllers.
}
%%%%%%%%%%%%%%%%% Commented (06/07/22) %%%%%%%%%%%%%%%%%%%
\begin{figure}[htb!]
\centering
    \begin{subfigure}{.35\textwidth}
        \centering
        \includegraphics[width=\textwidth]{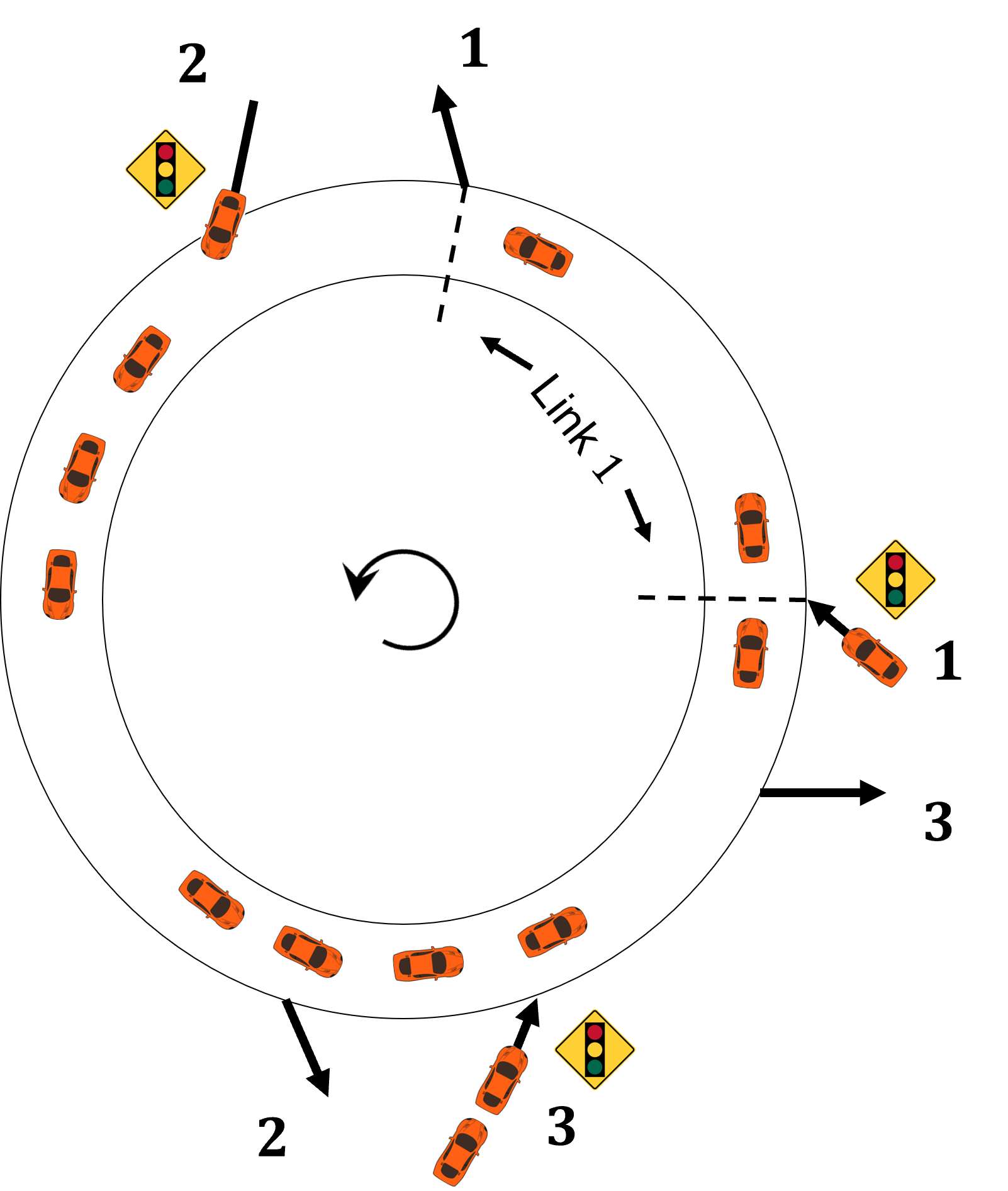}
        \caption{} \label{fig: Motivating Example}
    \end{subfigure}
    \begin{subfigure}{0.9\textwidth}
        \centering
        \includegraphics[width=\textwidth]{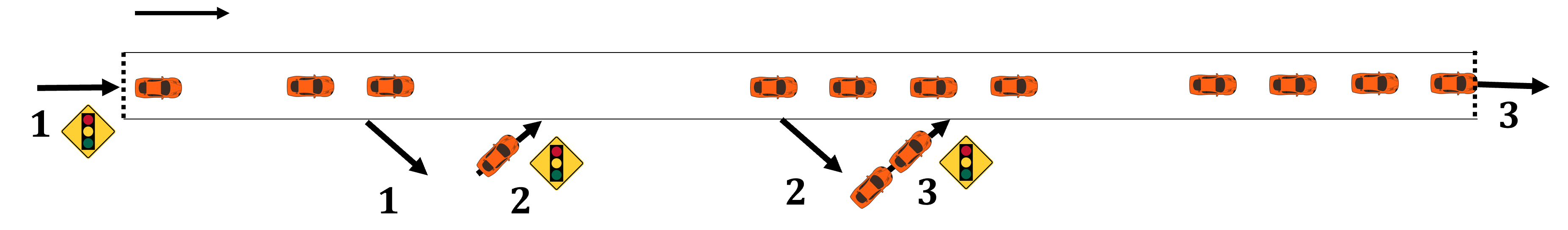}
        \caption{} \label{Fig: line network}
    \end{subfigure}
    \vspace{0.2 cm}
    \caption{\centering\sf Example of a (a) ring road configuration, (b) straight road configuration where the traffic light at the upstream entry point indicates that we have control over any previous on-ramp.}
    \label{Fig: model}
\end{figure}
%%%%%%%%%%%%%%%%% Commented (V2)(06/01/22) %%%%%%%%%%%%%%%%%%%
\mpcommentout{
Consider a single-lane ring road with multiple on-ramps and off-ramps. Vehicles in this network are homogeneous, obey the safety, speed, and comfort rules, and are admitted to the ring road by the ramp meters installed at the on-ramps. The objective is to design safe Ramp Metering (RM) policies that keep the network under-saturated. The ``under-saturated" term means that the long-run average queue sizes are bounded for all on-ramps. This measure of performance ensures that the queues would not grow indefinitely and it is tightly related to the average time spent by a driver in the network. 
\par
We seek to answer the following questions about this network: first, is it always possible to keep the network under-saturated for all demand profiles by means of, e.g., coordination between ramps? If not, what can be said about the network's ``saturation limit", i.e., the region beyond which the network cannot be kept under-saturated? And finally, if the demand is within the saturation limit, what type of RM policies, e.g., local or coordinated, are able to keep the network under-saturated? We formalize and study these questions in Section \ref{Section: Ramp metering design}. In the remainder of this section, we discuss the vehicle-level properties (Section \ref{Section: Vehicle Model and Controller}), demand model (Section \ref{Section: Arrival-Destination processes}), and control objectives (Section \ref{Section: Control Objectives}).
}
%%%%%%%%%%%%%%%%% Commented (V2)(06/01/22) %%%%%%%%%%%%%%%%%%%
\subsection{\mpcomment{Vehicle Following Safety Constraints}}\label{sec:vehicle-control}
We consider vehicles of length $\vehiclelength$ that have the same acceleration and braking capabilities, and are equipped with V2I communication systems. We use the term \emph{ego vehicle} to refer to a specific vehicle under consideration, and denote it by $e$. Consider a vehicle following scenario and let $v_e$ (resp. $v_l$) be the speed of the ego vehicle (resp. its leading vehicle), and $S_e$ be the \emph{safety distance} between the two vehicles required to avoid collision. We assume that $S_e$ satisfies
\begin{equation}\label{eq:safety-distance}
    S_e = \timeheadway v_e + \standstilldist + \frac{v^2_{e} - v^2_l}{2|\minaccel|},
\end{equation}
which is calculated based on an emergency stopping scenario with details given in \citet{Ioannou.Chien.1993}. Here, $\timeheadway > 0$ is known as the safe time headway constant, $\standstilldist > 0$ is an additional constant distance, and $\minaccel < 0$ is the minimum possible deceleration of the leading vehicle. For simplicity, we assume a third-order vehicle dynamics throughout the paper. We consider two general modes of operation for each vehicle: the \emph{$\Speedmode$} mode and the \emph{$\Safetymode$} mode. The main objective in the $\Speedmode$ mode is to adjust the speed to the \emph{free flow speed} $\speedlim$ when the ego vehicle is far from any leading vehicle; see Appendix~\ref{Section: (Appx) Cruise control dynamics}. The main objective in the $\Safetymode$ mode is to avoid collision when the ego vehicle gets close to a leading vehicle by maintaining \mpcomment{the safety distance} given by \eqref{eq:safety-distance}. We let $x_e = (p_e, v_e, a_e, I_e)$ be the state of the ego vehicle, where $p_e$ is its location, $v_e$ is the speed, $a_e$ is the acceleration, and $I_e$ is a binary variable which is equal to one if the ego vehicle is in the $\Safetymode$ mode, and zero otherwise. 
\par
\mpcomment{Consider an ego vehicle that starts from rest at a ramp meter and remains in the $\Speedmode$ mode. We define the \emph{acceleration lane} of an on-ramp as the section starting from its ramp meter up to the point on the mainline where the ego vehicle reaches the speed $\speedlim$.} Note that the acceleration lane may overlap with the mainline if the length of the on-ramp, from the ramp meter to the merging point, is not sufficiently long; see Figure~\ref{fig: Merge scenario}. However, if the on-ramp is sufficiently long, then the acceleration lane ends at the merging point where the ego vehicle enters the freeway. We assume that all the vehicles entering the freeway from an on-ramp merge at a fixed merging point. We define the \emph{merging speed} of an on-ramp as the speed of the ego vehicle at the merging point, assuming that it remains in the $\Speedmode$ mode between release and reaching the merging point. Note that the merging speed at an on-ramp is at most $\speedlim$.
%%%%%%%%%%%%%%%%%%% Commented (06/24/22) %%%%%%%%%%%%%%%%%%%%%
\mpcommentout{
\begin{itemize}
\item[(S2)] The ego vehicle merges into the mainline only if, at the moment of merging, its distance to the leading vehicle and the following vehicle, denoted $y_p$ and $y_f$ respectively, are no less than the safety distance. That is, only if $y_e \geq S_e$ and $y_f \geq S_f$.

\end{itemize}
}
%%%%%%%%%%%%%%%%%%% Commented (06/24/22) %%%%%%%%%%%%%%%%%%%%%

\par
%%%%%%%%%%%%%%%%%%% Commented (06/07/22) %%%%%%%%%%%%%%%%%%%%%
\mpcommentout{
satisfy the following rule at the moment of releasing the ego vehicle:
\mpmargin{(M) is consistent with (M2) in ks version.}
\begin{itemize}
    \item [(M)] . Specifically, $\hat{d}_p(0, t_m) > \hat{S}_p(0, t_m)$ and $\hat{d}_f(0, t_m) > \hat{S}_f(0, t_m)$.
\end{itemize}
}
%%%%%%%%%%%%%%%%%%% Commented (06/07/22) %%%%%%%%%%%%%%%%%%%%%
\par
%%%%%%%%%%%%%%%%%%%%%%%%%%%%%%%%%%%%%%%%%%%%%%%%%%%%%%%
%%%%%%%%%%%%%%%%%%%%% Figure %%%%%%%%%%%%%%%%%%%%%%%%%%
%%%%%%%%%%%%%%%%%%%%%%%%%%%%%%%%%%%%%%%%%%%%%%%%%%%%%%%
\begin{figure}[t]
    \centering
    \includegraphics[width=0.8\textwidth]{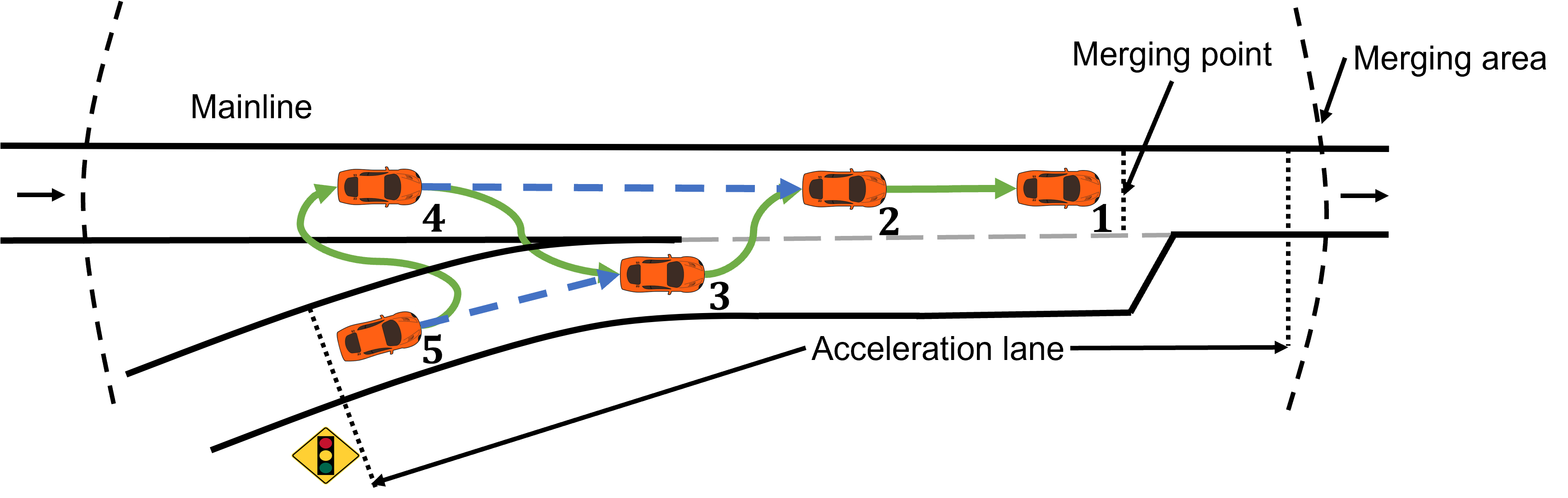}
    
    \vspace{0.2 cm}
    
    \caption{\centering\sf A merging scenario. For each vehicle, the dotted arrow indicates its leading vehicle while the solid arrow indicates its virtual leading vehicle. Note that vehicle $1$ is both the leading and the virtual leading vehicle of vehicle $2$.}
    \label{fig: Merge scenario}
\end{figure}
%%%%%%%%%%%%%%%%%%%%%%%%%%%%%%%%%%%%%%%%%%%%%%%%%%%%%%%
%%%%%%%%%%%%%%%%%%%%% Figure %%%%%%%%%%%%%%%%%%%%%%%%%%
%%%%%%%%%%%%%%%%%%%%%%%%%%%%%%%%%%%%%%%%%%%%%%%%%%%%%%%
%%%%%%%%%%%%%%%%%%%%%%%%%%%%%%%%%%%%%%%%%%%%%%%%%%%%%%%
%%%%%%%%%%%%%%%%%%%%% Figure %%%%%%%%%%%%%%%%%%%%%%%%%%
%%%%%%%%%%%%%%%%%%%%%%%%%%%%%%%%%%%%%%%%%%%%%%%%%%%%%%%
\begin{figure}[t]
    \centering
    \includegraphics[width=0.75\textwidth]{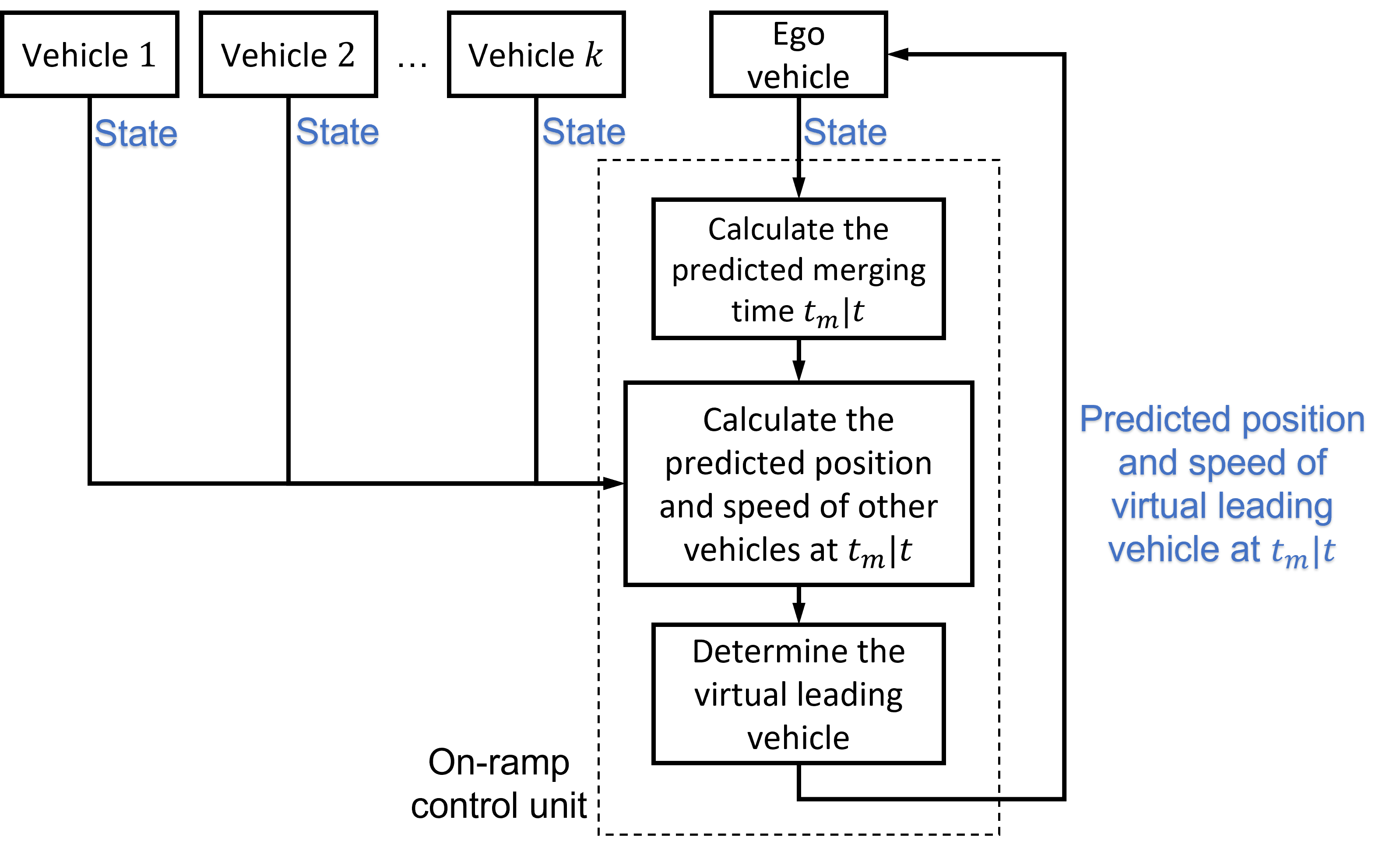}
    
    \vspace{0.2 cm}
    
    \caption{\centering\sf Schematic of the protocol for assigning the ego vehicle's virtual leading vehicle. Here, $k$ is the number of vehicles, other than the ego vehicle, in the merging area.}
    \label{fig:virtual-protocol}
\end{figure}
%%%%%%%%%%%%%%%%%%%%%%%%%%%%%%%%%%%%%%%%%%%%%%%%%%%%%%%
%%%%%%%%%%%%%%%%%%%%% Figure %%%%%%%%%%%%%%%%%%%%%%%%%%
%%%%%%%%%%%%%%%%%%%%%%%%%%%%%%%%%%%%%%%%%%%%%%%%%%%%%%%
Consider a merging scenario as shown in Figure~\ref{fig: Merge scenario}. An ego vehicle entering the merging area is assigned a \emph{virtual} leading vehicle. The virtual leading vehicle, which is the vehicle that is predicted to be in front of the ego vehicle once the ego vehicle has crossed the merging point, is determined (and frequently updated) using a protocol similar to \citet{ntousakis2016optimal}. We will now describe this protocol which is shown in Figure~\ref{fig:virtual-protocol}. At time $t$, the ego vehicle communicates its state $x_e$ to the on-ramp control unit via V2I communication. The on-ramp control unit predicts the time $t_m|t$ that the ego vehicle will cross the merging point according to the following prediction rule: if in the $\Speedmode$ mode, it assumes that the ego vehicle will remain in this mode in the future; if in the $\Safetymode$ mode, it assumes that the ego vehicle will follow a constant speed trajectory in the future. Other vehicles in the merging area also communicate their state to the on-ramp control unit. By using the same prediction rule, the on-ramp control unit predicts the position of each of these vehicles at time $t_m|t$, which will determine the ego vehicle's virtual leading vehicle. Once the virtual leading vehicle has been determined, the on-ramp control unit communicates its predicted position and speed at time $t_m|t$ back to the ego vehicle. This information is frequently updated until the ego vehicle has crossed the merging point.
%%%%%%%%%%%%%%%%% Commented (Rev1)(04/07/23) %%%%%%%%%%%%%%%%%
\mpcommentout{
It then communicates $t_m|t$ and its predicted speed at time $t_m|t$ to all the other vehicles in the merging area via V2V communication. Other vehicles use the same prediction rule to predict their own position and speed at time $t_m|t$ and communicate them back to the ego vehicle. This information is used to determine the ego vehicle's virtual leading vehicle, and is frequently updated until the ego vehicle has crossed the merging point.
}
%%%%%%%%%%%%%%%%% Commented (Rev1)(04/07/23) %%%%%%%%%%%%%%%%%
\par
The above protocol describes how a virtual leading vehicle is assigned. The next set of assumptions prescribe the safety rules all the vehicles follow:
\begin{itemize}
\item[(VC1)] the ego vehicle maintains the constant speed $\speedlim$ at time $t$ if both of the following conditions are satisfied:
\begin{itemize}
    \item[(a)] \mpcomment{\textbf{merging scenario:}} its predicted distance with respect to its virtual leading vehicle at the moment of merging is safe, i.e.,  
    \begin{equation*}
        \hat{y}_e(t_m|t) \geq \hat{S}_e(t_m|t) := \timeheadway \hat{v}_e(t_m|t) + \standstilldist + \frac{\hat{v}_e^2(t_m|t) - \hat{v}_{\hat{l}}^2(t_m|t)}{2|\minaccel|},
    \end{equation*}
    where $\hat{y}_e(t_m|t)$ is the predicted distance between the two vehicles at the moment of merging based on the information available at time $t \leq t_m$. Similarly, $\hat{v}_e$ and $\hat{v}_{\hat{l}}$ are the predicted speeds of the ego vehicle and its virtual leading vehicle, respectively. These predictions are obtained by the protocol explained above.
    \item[(b)] \mpcomment{\textbf{vehicle following scenario:}} it is at a safe distance with respect to its leading vehicle, i.e., $y_e(t) \geq S_e(t)$, where $y_e$ is the distance between the two vehicles. Note that if the leading vehicle is also at the constant speed $\speedlim$, then $y_e \geq S_e = \timeheadway \speedlim + \standstilldist$. This distance is equivalent to a time headway of at least $\timestep := \timeheadway + (\standstilldist + \vehiclelength)/\speedlim$ between the front bumpers of the two vehicles. This rule is widely adopted by human drivers as well as standard adaptive cruise control systems (\citet{Ioannou.Chien.1993}).
\end{itemize}

\item[(VC2)] the ego vehicle is initialized to be in the $\Speedmode$ mode upon release from the on-ramp. It changes mode at time $t$ if $\hat{y}_e(t_m|t) < \hat{S}_e(t_m|t)$ in a merging scenario or $y_e(t) < S_e(t)$ in a vehicle following scenario. Moreover, if the ego vehicle is released from the on-ramp at least $\timestep$ seconds after its leading vehicle, then it will not change mode because of its leading vehicle while the leading vehicle is in the $\Speedmode$ mode. Note that the ego vehicle may still change mode because of a virtual leading vehicle.
%\begin{itemize}
%    \item[(a)] $y_e < S_e$.  
%    \item[(b)] it is at the on-ramp and predicts to violate the safety distance with respect to its projected leading or projected following vehicles.
%    \item[(c)] it is at the on-ramp and its projected following vehicle is in the $\Safetymode$ mode.
\end{itemize}
\par
Note that in order to keep the results fairly general, we intentionally have not specified the total number of submodes within the $\Safetymode$ mode, the exact control logic within each submode, or the exact logic for switching back to the $\Speedmode$ mode. Such details will be introduced only if and when needed for performance analysis of RM policies in the paper. Also, note that the control logic in the $\Safetymode$ mode is allowed to be different for different vehicles when they are not moving at the speed $\speedlim$. However, the following assumption rules out any impractical control logic such as unnecessary braking:
%\item [(VC3)] if the ego vehicle is moving at a constant speed, then this speed is no more than $\speedlim$. If, in addition, the ego vehicle is in the $\Safetymode$ mode, it keeps moving at this speed while $y_e \geq S_e$.
%\item[(VC3)] there exists $\Tempty>0$ such that, starting from any initial condition, all the vehicles on the mainline and acceleration lanes reach their destination off-ramp by time $\Tempty$, if no additional vehicle is released during this time.
\begin{itemize}
\item[(VC3)] there exists $\Tempty$ such that for any initial condition, vehicles reach the \emph{free flow state} after at most $\Tempty$ seconds if no other vehicle is released from the on-ramps. The free flow state refers to a state where: (i) if a vehicle is in the $\Safetymode$ mode, then it moves at the constant speed $\speedlim$ and will maintain this speed until it exits the network, and (ii) if a vehicle is in the $\Speedmode$ mode, then it will remain in this mode until it exits the network.
\end{itemize}
%%%%%%%%%%%%%%%%%%%%%%%%%%%%%%%%%%%%%%%%%%%%%%%%%%%%%%%
%%%%%%%%%%%%%%%%%%%%% Example %%%%%%%%%%%%%%%%%%%%%%%%%
%%%%%%%%%%%%%%%%%%%%%%%%%%%%%%%%%%%%%%%%%%%%%%%%%%%%%%%
\begin{example}\label{ex:prediction-operation}
We present a numerical example to better understand some of the previous assumptions. Let $\timeheadway = 1.5 ~[\text{sec}]$, $\standstilldist = 4~ [\text{m}]$, $\vehiclelength = 4.5 ~[\text{m}]$, and $\speedlim = 15 ~[\text{m/sec}]$. Suppose that all the vehicles in Figure~\ref{fig: Merge scenario} are moving at the constant speed $\speedlim$ at time $t$, and recall that the safety distance at the free flow speed is $\timeheadway\speedlim+\standstilldist=26.5~[\text{m}]$. Let vehicles $1$-$5$ be $5$, $36$, $70$, $101$, and $132$ meters upstream of the merging point, respectively. Vehicle $3$ (the ego vehicle)'s predicted time of crossing the merging point is $t_m|t=70/15=4.67~[\text{sec}]$. The predicted positions of vehicles $1$ and $2$ at time $t_m|t$ are $65$ and $34$ meters downstream, and the predicted positions of vehicles $4$ and $5$ at time $t_m|t$ are $31$ and $61$ meters upstream of the merging point, respectively. Thus, the ego vehicle's virtual leading vehicle at time $t_m|t$ is determined to be vehicle $2$. Moreover, since $\hat{y}_e(t_m|t) = 34-\vehiclelength = 29.5 > 26.5 =\hat{S}_e(t_m|t)=\timeheadway\hat{v}_e(t_m|t)+\standstilldist$, the ego vehicle maintains the constant speed $\speedlim$ by (VC1)(a). Similar calculations show that all the other vehicles also maintain their speeds at $\speedlim$.
\end{example}
%%%%%%%%%%%%%%%%%%%%%%%%%%%%%%%%%%%%%%%%%%%%%%%%%%%%%%%
%\begin{remark}
%initial number of vehicles on the mainline and acceleration lanes does not such that the distance (resp. projected distance) between ego vehicle and its leading (resp. projected leading) vehicle is no less than the safety distance. This will be referred to as a \emph{safe} initial condition. It is easy to see that the number of vehicles on the mainline associated with a safe initial condition is no more than $\Perimeter/(\standstilldist + \vehiclelength)$. \mpcomment{It is also natural to assume that all vehicles merge into the mainline at positive speeds}. This and (S2) then imply that the number of vehicles on the mainline will not exceed $\Perimeter/(\standstilldist + \vehiclelength)$ in the future.
%\end{remark}

%\input{Safety_MP}

%\input{Vehicle_Rules_Mp}

%%%%%%%%%%%%%%%%%%%%%%%%%%%%%%%%%%%%%%%%%%%%%%%%%%%%%%%%%%%%%%%%%%%%%%%%%%%%%
%%%%%%%%%%%%%%%%%%%%%%%%%%%%%%%%%%%%%%%%%%%%%%%%%%%%%%%%%%%%%%%%%%%%%%%%%%%%%
%%%%%%%%%%%%%%%%%%%%%%%%%%%%%%% Sub-Section %%%%%%%%%%%%%%%%%%%%%%%%%%%%%%%%%
%%%%%%%%%%%%%%%%%%%%%%%%%%%%%%%%%%%%%%%%%%%%%%%%%%%%%%%%%%%%%%%%%%%%%%%%%%%%%
%%%%%%%%%%%%%%%%%%%%%%%%%%%%%%%%%%%%%%%%%%%%%%%%%%%%%%%%%%%%%%%%%%%%%%%%%%%%%
\subsection{Demand Model and Throughput}\label{sec:demand}
It will be convenient for performance analysis later on to adopt a discrete time setting. Let the duration of each time step be $\timestep = \timeheadway + (\standstilldist + \vehiclelength)/\speedlim$, representing the minimum time headway between the front bumpers of two vehicles that are moving at the speed $\speedlim$ \mpcomment{and are at a safe distance}; see (VC1)(b) in Section~\ref{sec:vehicle-control}. We assume that vehicles arrive to on-ramp $i \in [\numramps]$ according to an i.i.d. Bernoulli process with parameter $\Arrivalrate{i} \in [0,1]$ independent of the other on-ramps. That is, in any given time step, the probability that a vehicle arrives at the $i$-th on-ramp is $\Arrivalrate{i}$ independent of everything else. Note that $\Arrivalrate{i}$ specifies the arrival rate to on-ramp $i$ in terms of the number of vehicles per $\timestep$ seconds. Let $\Arrivalrate{} := [\Arrivalrate{i}]$ be the vector of arrival rates. The destination off-ramp for individual arriving vehicles is assumed to be i.i.d. and given by the routing matrix $\routingmatrix = [\routingmatrix_{ij}]$, where $0 \leq \routingmatrix_{ij} \leq 1$ is the probability that a vehicle arriving to on-ramp $i$ wants to exit from off-ramp $j$. Note that $\routingmatrix_{ij}$ specifies the long-run fraction of arrivals at on-ramp $i$ that want to exit from off-ramp $j$. Naturally, for every on-ramp $i$ we have $\sum_{j}\routingmatrix_{ij} = 1$.  Finally, we let $\tilde{\routingmatrix} = [\tilde{\routingmatrix}_{ij}]$ be the \emph{cumulative} routing matrix, where $\tilde{\routingmatrix}_{ij}$ is the fraction of arrivals at on-ramp $i$ that need to use link $j$ in order to reach their destination. Therefore, $\Arrivalrate{i}\tilde{\routingmatrix}_{ij}$ is the rate of arrivals at on-ramp $i$ that need to use link $j$ in order to reach their destination, i.e., the \emph{load} induced on link $j$ by on-ramp $i$. Let $\Avgload{j} := \sum_{i}\Arrivalrate{i}\tilde{\routingmatrix}_{ij}$ be the total load induced on link $j$ by all the on-ramps, and let $\Avgload{} := \max_{j \in [\numramps]}\Avgload{j}$ be the \textit{maximum load}.
\mpcommentout{
%%%%%%%%%%%%%%%%%%%%%%%%%%%%%%%%%%%%%%%%%%%%%%%%%%%%%%%
%%%%%%%%%%%%%%%%%%%%% Remark %%%%%%%%%%%%%%%%%%%%%%%%%%
%%%%%%%%%%%%%%%%%%%%%%%%%%%%%%%%%%%%%%%%%%%%%%%%%%%%%%%
\begin{remark}
Since the vehicles were assumed to be equipped with V2V and V2I communication systems, they may communicate their destination to other vehicles or to the on-ramps. However, we do not require them to do so in our RM policies.
\end{remark}
%%%%%%%%%%%%%%%%%%%%%%%%%%%%%%%%%%%%%%%%%%%%%%%%%%%%%%%
%%%%%%%%%%%%%%%%%%%%% Remark %%%%%%%%%%%%%%%%%%%%%%%%%%
%%%%%%%%%%%%%%%%%%%%%%%%%%%%%%%%%%%%%%%%%%%%%%%%%%%%%%%
}
%%%%%%%%%%%%%%%%%%%%%%%%%%%%%%%%%%%%%%%%%%%%%%%%%%%%%%%
%%%%%%%%%%%%%%%%%%%%% Remark %%%%%%%%%%%%%%%%%%%%%%%%%%
%%%%%%%%%%%%%%%%%%%%%%%%%%%%%%%%%%%%%%%%%%%%%%%%%%%%%%%
\begin{remark}\label{remark:demand-generalized}
The current demand model, i.e., Bernoulli arrivals and Bernoulli routing, is chosen to simplify the technical details in the proofs. We believe that our results are far more general and hold for more practical demand models used in the literature, e.g., see \cite{jin2009departure} for an example of arrival models. 
\end{remark}
%%%%%%%%%%%%%%%%%%%%%%%%%%%%%%%%%%%%%%%%%%%%%%%%%%%%%%%
%%%%%%%%%%%%%%%%%%%%% Remark %%%%%%%%%%%%%%%%%%%%%%%%%%
%%%%%%%%%%%%%%%%%%%%%%%%%%%%%%%%%%%%%%%%%%%%%%%%%%%%%%%
%%%%%%%%%%%%%%%%%%%%%%%%%%%%%%%%%%%%%%%%%%%%%%%%%%%%%%%
%%%%%%%%%%%%%%%%%%%%% Example %%%%%%%%%%%%%%%%%%%%%%%%%
%%%%%%%%%%%%%%%%%%%%%%%%%%%%%%%%%%%%%%%%%%%%%%%%%%%%%%%
\begin{example}\label{Ex: Cumulative routing matrix}
Let the number of on- and off-ramps be $\numramps = 3$, and let the routing matrix be
\begin{equation*}
    \routingmatrix = \begin{pmatrix}
    \routingmatrix_{11} & \routingmatrix_{12} & \routingmatrix_{13} \\
    \routingmatrix_{21} & \routingmatrix_{22} & \routingmatrix_{23} \\
    \routingmatrix_{31} & \routingmatrix_{32} & \routingmatrix_{33}
    \end{pmatrix}.
\end{equation*}
\par 
Then, the cumulative routing matrix is
\begin{equation*}
    \tilde{\routingmatrix} = \begin{pmatrix}
    1 & 1 - \routingmatrix_{11} & 1 - (\routingmatrix_{11} + \routingmatrix_{12}) \\
    1 - (\routingmatrix_{22} + \routingmatrix_{23}) & 1 & 1 - \routingmatrix_{22} \\
    1 - \routingmatrix_{33} & 1- (\routingmatrix_{31} + \routingmatrix_{33}) & 1
    \end{pmatrix} = \begin{pmatrix}
    1 &  \routingmatrix_{12} + \routingmatrix_{13} & \routingmatrix_{13} \\
    \routingmatrix_{21} & 1 & \routingmatrix_{21} + \routingmatrix_{23} \\
    \routingmatrix_{31} + \routingmatrix_{32} & \routingmatrix_{32} & 1
    \end{pmatrix}.
\end{equation*}
\end{example}
%%%%%%%%%%%%%%%%%%%%%%%%%%%%%%%%%%%%%%%%%%%%%%%%%%%%%%%
%%%%%%%%%%%%%%%%%%%%% Example %%%%%%%%%%%%%%%%%%%%%%%%%
%%%%%%%%%%%%%%%%%%%%%%%%%%%%%%%%%%%%%%%%%%%%%%%%%%%%%%%
\par
We now formalize the notion of ``throughput" which is the key performance metric in this paper. Let $\Queuelength{i}(t)$ be the queue size at on-ramp $i \in [\numramps]$, and $\Queuelength{}(t) = [\Queuelength{i}(t)]$ be the vector of queue sizes at time $t$. 
%%%%%%%%%%%%%%%%% Commented (V3)(11/10/22) %%%%%%%%%%%%%%%%%%%
\mpcommentout{The freeway is said to be under-saturated under a given demand $(\Arrivalrate{},\routingmatrix)$ and a RM policy if $\limsup_{t \to \infty} \E{\Queuelength{i}(t)} < \infty$ for all $i \in [\numramps]$; otherwise, it is called saturated. The boundary of the under-saturation region is called the throughput of the RM policy.
}
%%%%%%%%%%%%%%%%% Commented (V3)(11/10/22) %%%%%%%%%%%%%%%%%%%
For a given routing matrix $\routingmatrix$, the under-saturation region of a RM policy $\pi$ is defined as follows:
\begin{equation*}
    U_{\pi,\routingmatrix} = \{\Arrivalrate{}: \limsup_{t \to \infty} \E{\Queuelength{i}(t)} < \infty~~\forall i \in [\numramps]~ \text{under policy}~\pi\}.
\end{equation*}
\par
This is the set of $\Arrivalrate{}$'s for which the queue sizes at all the on-ramps remain bounded in expectation. The boundary of this set is called the throughput of the policy $\pi$. We are interested in finding a RM policy $\pi$ such that for every $\routingmatrix$, $U_{\pi',\routingmatrix} \subseteq U_{\pi,\routingmatrix}$ for all policies $\pi'$, including those that have \mpcomment{prior knowledge} of $\Arrivalrate{}$ and $\routingmatrix$. In other words, for every $\routingmatrix$, if the freeway remains under-saturated under some policy $\pi'$, then it also remains under-saturated under the policy $\pi$. In that case, we say that policy $\pi$ maximizes the throughput. \mpcomment{One of our main contributions is designing policies that are reactive, i.e., they do not require $\Arrivalrate{}$ and $\routingmatrix$, but maximize throughput for all practical purposes.}
%%%%%%%%%%%%%%%%%%%%%%%%%%%%%%%%%%%%%%%%%%%%%%%%%%%%%%%
%%%%%%%%%%%%%%%%%%%%%% Remark %%%%%%%%%%%%%%%%%%%%%%%%%
%%%%%%%%%%%%%%%%%%%%%%%%%%%%%%%%%%%%%%%%%%%%%%%%%%%%%%%
\begin{remark}
A rigorous definition of the throughput should also include its dependence on the initial condition of the vehicles and the initial queue sizes. We have removed this dependence for simplicity since the throughput of our policies do not depend on the initial condition. 
\end{remark}
%%%%%%%%%%%%%%%%%%%%%%%%%%%%%%%%%%%%%%%%%%%%%%%%%%%%%%%
%%%%%%%%%%%%%%%%%%%%%% Remark %%%%%%%%%%%%%%%%%%%%%%%%%
%%%%%%%%%%%%%%%%%%%%%%%%%%%%%%%%%%%%%%%%%%%%%%%%%%%%%%%
%%%%%%%%%%%%%%%%%%%%%%%%%%%%%%%%%%%%%%%%%%%%%%%%%%%%%%%
%%%%%%%%%%%%%%%%%%%%% Figure %%%%%%%%%%%%%%%%%%%%%%%%%%
%%%%%%%%%%%%%%%%%%%%%%%%%%%%%%%%%%%%%%%%%%%%%%%%%%%%%%%
\begin{figure}[t!]
        \centering
        \includegraphics[width=0.4\textwidth]{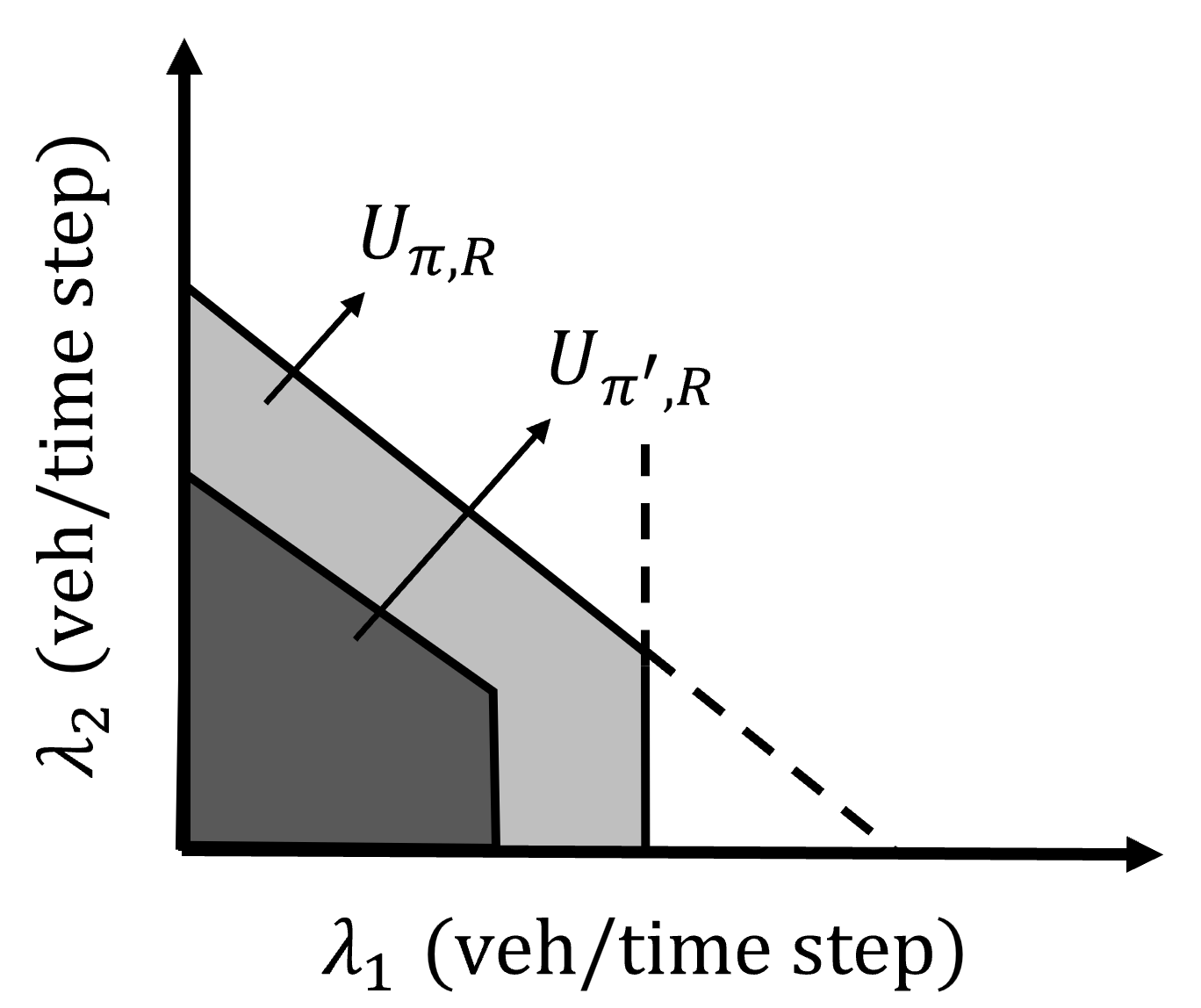}
        \vspace{0.2 cm}
    \caption{\centering\sf An illustration of the under-saturation region of some policy $\pi'$ (dark grey area) and a policy $\pi$ that maximizes the throughput (dark + light grey areas).}\label{Fig: Illustrative example for throughput}
\end{figure}
%%%%%%%%%%%%%%%%%%%%%%%%%%%%%%%%%%%%%%%%%%%%%%%%%%%%%%%
%%%%%%%%%%%%%%%%%%%%% Figure %%%%%%%%%%%%%%%%%%%%%%%%%%
%%%%%%%%%%%%%%%%%%%%%%%%%%%%%%%%%%%%%%%%%%%%%%%%%%%%%%%

%%%%%%%%%%%%%%%%%%%%%%%%%%%%%%%%%%%%%%%%%%%%%%%%%%%%%%%
%%%%%%%%%%%%%%%%%%%%% Example %%%%%%%%%%%%%%%%%%%%%%%%%
%%%%%%%%%%%%%%%%%%%%%%%%%%%%%%%%%%%%%%%%%%%%%%%%%%%%%%%
\begin{example}\label{ex:throughput}
Let $\numramps=3$ and consider a given $\routingmatrix$. Suppose that a policy $\pi$ is able to maximize the throughput; that is, for any other policy $\pi'$, we have $U_{\pi', \routingmatrix} \subseteq U_{\pi, \routingmatrix}$. An illustration of $U_{\pi, \routingmatrix}$ and $U_{\pi', \routingmatrix}$ is shown in Figure~\ref{Fig: Illustrative example for throughput} for a fixed $\Arrivalrate{3}$. From the figure, one can see that if $(\Arrivalrate{1},\Arrivalrate{2}) \in U_{\pi', \routingmatrix}$, then $(\Arrivalrate{1},\Arrivalrate{2}) \in U_{\pi, \routingmatrix}$, i.e., if the freeway remains under-saturated under the policy $\pi'$, then it also remains under-saturated under the policy $\pi$. We will provide a numerical characterization of $U_{\pi, \routingmatrix}$ in Example~\ref{ex:throughput-numerical}.
\end{example}

\subsection{\mpcomment{Slot}}\label{Section: RM rules}
%%%%%%%%%%%%%%%%%%%%%%%%%%%%%%%%%%%%%%%%%%%%%%%%%%%%%%%
%%%%%%%%%%%%%%%%%%%%% Figure %%%%%%%%%%%%%%%%%%%%%%%%%%
%%%%%%%%%%%%%%%%%%%%%%%%%%%%%%%%%%%%%%%%%%%%%%%%%%%%%%%
\begin{figure}[t]
    \centering
    \includegraphics[width=0.8\textwidth]{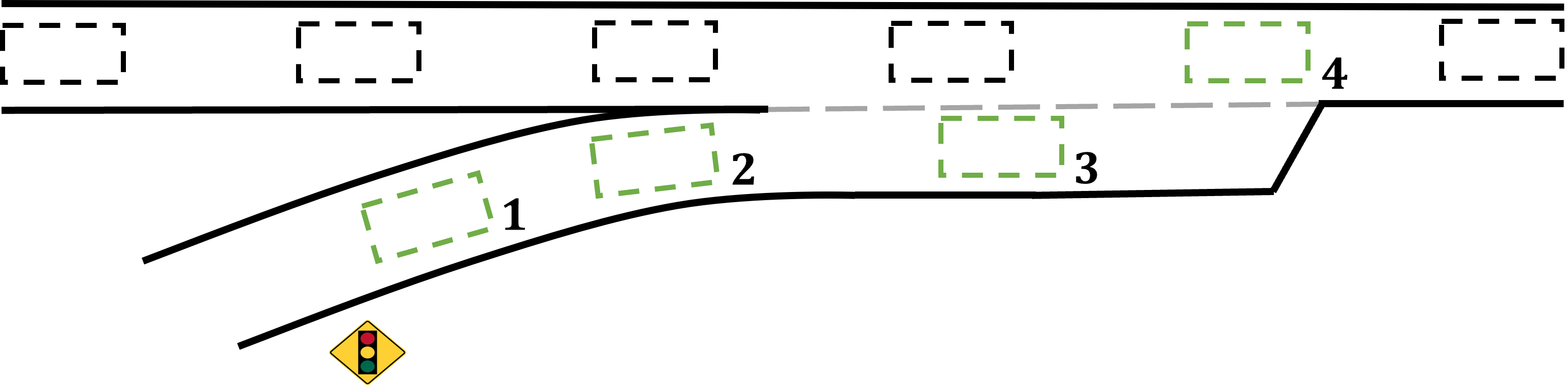}
    
    \vspace{0.2 cm}
    
    \caption{\centering\sf \mpcomment{A configuration of slots for the mainline and acceleration lane of an on-ramp. There are four slots for the acceleration lane (colored in green) with slot $4$ coinciding with a mainline slot.}}
    \label{fig:slots}
\end{figure}
%%%%%%%%%%%%%%%%%%%%%%%%%%%%%%%%%%%%%%%%%%%%%%%%%%%%%%%
%%%%%%%%%%%%%%%%%%%%% Figure %%%%%%%%%%%%%%%%%%%%%%%%%%
%%%%%%%%%%%%%%%%%%%%%%%%%%%%%%%%%%%%%%%%%%%%%%%%%%%%%%%
To conveniently track vehicle locations in discrete time, we introduce the notion of \emph{slot}. A slot corresponds to a specific point on either the mainline or acceleration lanes at a given time. We first define the mainline slots. Let $\numcells$ denote the maximum number of distinct points that can be placed on the mainline, such that the distance between adjacent points is $\timeheadway \speedlim + \standstilldist + L$. This distance is governed by the safety distance $S_e$, as explained in Section~\ref{sec:vehicle-control}. In other words, $\numcells$ is the maximum number of \mpcomment{vehicles that can move at the speed $\speedlim$ on the mainline without violating the vehicle following safety constraints}. Consider a configuration of these $\numcells$ points at time $t = 0$. Each point represents a slot on the mainline that moves at the free flow speed. Without loss of generality, we assume that the length of the mainline $\Perimeter$ is such that each slot replaces the next slot at the end of each time step. 
\par
Next, we define the acceleration lane slots. \mpcomment{Consider the $i$-th on-ramp and suppose that an ego vehicle starts from rest at its ramp meter and remains in the $\Speedmode$ mode. At the end of each time step, the ego vehicle's location represents a slot for the acceleration lane of on-ramp $i$. These slots continue until the ego vehicle exits the acceleration lane.} For example, if the ego vehicle exits the acceleration after $3.5\timestep$ seconds, there are four slots corresponding to its location at times $\timestep$, $2\timestep$, $3\timestep$, and $4\timestep$; see Figure~\ref{fig:slots}. Let $\numaccslots{i}$ be the number of slots for the acceleration lane of on-ramp $i$, and $\numaccslots{a} = \sum_{i}\numaccslots{i}$. Note that by definition, the last slot of every acceleration lane is on the mainline. Without loss of generality, we consider a configuration of slots at $t=0$, such that the last slot of every acceleration lane coincides with a mainline slot. The details to justify the no loss in generality are given as follows: for a given configuration of mainline slots at $t = 0$, there exists $t_i \in [0,\timestep)$ such that the last acceleration lane slot of on-ramp $i$ coincides with a mainline slot at time $t_i$. Thereafter, the last acceleration lane slot coincides with a mainline slot after every $\timestep$ seconds, i.e., at times $k\timestep + t_i$ for all $k \in \N_{0}$. The times $k\timestep + t_i$, $k \in \N_{0}$, are the release times of on-ramp $i$ in the proposed RM policies. Therefore, the assumption that all the last acceleration lane slots initially coincide with a mainline slot (which corresponds to $t_i = 0$ for all $i \in [\numramps]$) only means a shifted sequence of release times, which justifies the no loss in generality.
\par
Consider an initial condition of the vehicles, where the vehicles are in the free flow state and the location of each vehicle coincides with a slot for all times in the future. For this initial condition and under the proposed RM policies, the following sequence of events occurs during each time step: (i) the mainline slots rotate one position in the direction of travel and replace the next slot. Similarly, the acceleration lane slots of each on-ramp replace the next slot, with the last slot replacing the first slot. The numbering of the slots is reset, with the new first mainline slot after on-ramp $1$ numbered $1$, and the rest of the mainline slots numbered in an increasing order; the acceleration lane slots are numbered similarly; (ii) vehicles that reach their destination off-ramp exit the network without creating any obstruction for the upstream vehicles; (iii) if permitted by the RM policy, a new vehicle is released. For the given initial condition and under the proposed RM policies, the location of the newly released vehicle will coincide with a slot for all times in the future \footnote{Without loss of generality, we let the first vehicle in the queue of each on-ramp be at rest before being released. This assumption naturally applies when a queue has formed at the on-ramp. If, on the other hand, there is no queue, we can modify the proposed policies so that if the location an arriving vehicle does not coincide with a slot, then it will not be released.}.
\par

\section{Ramp Metering and Performance Analysis}
\label{sec:large-merge-speed}
In this section, we present traffic-responsive RM policies that operate under vehicle following safety constraints and analyze their performance. An inner-estimate to the under-saturation region of each policy is provided in Sections~\ref{sec:renewal}-\ref{sec:greedy}, and is then compared to an outer-estimate in Section~\ref{sec:necessary}. In Section~\ref{Section: Straight Line}, we discuss the extension of our results to the straight road configuration.
\par
For easier navigation, we briefly review the proposed policies here. The policies in Sections~\ref{sec:renewal}, \ref{sec:release-rate}, and \ref{Section: Impact of Autonomy} are coordinated, while the one in Section~\ref{sec:distributed} is a distributed version of the coordinated policy in Section~\ref{sec:release-rate}, and the one in Section~\ref{sec:greedy} is a local policy. All of the proposed policies operate under vehicle following safety constraints, where the on-ramps release new vehicles only if \mpcomment{there is sufficient gap between vehicles on the mainline at the moment of release}. Moreover, All of the policies are reactive, meaning that they only require real-time traffic measurements without requiring any prior knowledge of the arrival rate $\Arrivalrate{}$ or the routing matrix $\routingmatrix$. \mpcomment{They obtain the real-time traffic measurements by V2I communication.}
\par
The proposed RM policies work in synchronous \emph{cycles} during which an on-ramp does not release more vehicles than the number of vehicles waiting in its queue, i.e., its queue size, at the start of the cycle. The synchronization of cycles is done in real time in Section~\ref{sec:renewal}, whereas it is done once offline in Sections~\ref{sec:release-rate}-\ref{sec:greedy}. The policies differ in how they use the traffic measurements: (i) the policy in Section~\ref{sec:renewal} pauses release until a new cycle starts, (ii) the policies in Sections~\ref{sec:release-rate}-\ref{sec:distributed} adjust the time between successive releases during a cycle, and (iii) the policy in Section~\ref{Section: Impact of Autonomy} adopts a conservative dynamic safe gap criterion for release during a cycle. However, the actions of the three policies in Sections~\ref{sec:release-rate}-\ref{Section: Impact of Autonomy} at the free flow state are equivalent to that of the local policy in Section~\ref{sec:greedy}.
\par
The traffic measurements required by the proposed policies will be a combination of $\Queuelength{}$ and the state (or part of the state) of all the vehicles $\StateofDySys := (x_e)_{e \in [n]}$. Here, $n$ is the number of vehicles on the mainline and acceleration lanes, and $x_e$ is the state of the ego vehicle as defined in Section~\ref{sec:vehicle-control}. These measurements are communicated to the on-ramps by V2I communication. Table \ref{table:ramp-metering-summary} provides a summary of the communication cost of the RM policies considered in this paper. The notion of communication cost is explained in Appendix~\ref{Section: (Appx) comm cost}. In Table \ref{table:ramp-metering-summary}, $\numaccslots{m}$ is the total number of slots in all the merging areas, which is at most $\numcells + \numaccslots{a}$, $\updateperiod$ is a design update period, and $c_{i}$, $i = 1, 2$, is the contribution of the queue size to the communication cost. 
\begin{table}[htb!]
\centering
\begin{tabularx}{\textwidth}{l l l l}
  \hline
  
  Section & RM Policy & \hspace{2cm} & Worst-case Communication Cost \\ 
   & & \hspace{2cm} & (number of transmissions per $\timestep$ seconds) \\
  \hline
 \ref{sec:renewal} & Renewal & & $\numcells + \numaccslots{a} + c_{1}$\\
 \ref{sec:release-rate} & Dynamic Release Rate ($\GenRM$) & & $\numramps(\numcells + \numaccslots{a})/\updateperiod + \numaccslots{m} + c_{2}$ \\ 
 \ref{sec:distributed} & Distributed DRR ($\DisDRRRM$) &  & $(\numcells + \numaccslots{a})/\updateperiod + \numaccslots{m} + c_{2}$ \\
 \ref{Section: Impact of Autonomy} & Dynamic Space Gap ($\ConRM$) &  & $(\numcells + \numaccslots{a})\numramps + c_{2}$\\
 \ref{sec:greedy} & Greedy & & $\numaccslots{m}$\\
 \hline
\end{tabularx}
\caption{\sf Summary of the RM policies studied in this paper.}\label{table:ramp-metering-summary}
\end{table}

\subsection{Renewal Policy}\label{sec:renewal}
The first policy, called the \emph{Renewal} policy, is inspired by the queuing theory literature in the context of communication networks, e.g., see \citet{georgiadis1995scheduling, armony2003queueing}. \mpcomment{In this policy, an on-ramp pauses vehicle release once it has released all the vehicles waiting at the start of a cycle until all other on-ramps have done the same.} 
%%%%%%%%%%%%%%%%% Commented (Rev1)(02/20/23) %%%%%%%%%%%%%%%%%
\mpcommentout{
, and these vehicles exit the freeway, i.e., until the mainline and acceleration lanes are empty}
%%%%%%%%%%%%%%%%%%%%%%%%%%%%%%%%%%%%%%%%%%%%%%%%%%%%%%%%%%%%%%

%%%%%%%%%%%%%%%%%%%%%%%%%%%%%%%%%%%%%%%%%%%%%%%%%%%%%%%
%%%%%%%%%%%%%%%%%%%%% Definition %%%%%%%%%%%%%%%%%%%%%%
%%%%%%%%%%%%%%%%%%%%%%%%%%%%%%%%%%%%%%%%%%%%%%%%%%%%%%%
\begin{definition}\label{Def: Q-RM policy}
\textbf{(Renewal ramp metering policy)}
%%%%%%%%%%%%%%%%% Commented (Rev1)(02/20/23) %%%%%%%%%%%%%%%%%
\mpcommentout{No vehicle is released until all the initial vehicles exit the freeway, i.e., until the mainline and acceleration lanes are empty, say at time $t_1$.}
%%%%%%%%%%%%%%%%%%%%%%%%%%%%%%%%%%%%%%%%%%%%%%%%%%%%%%%%%%%%%%
No vehicle is released until all the initial vehicles reach the free flow state, as defined in (VC3), at time $t_1$. Thereafter, the policy works in cycles of \emph{variable} length. At the start of the $k$-th cycle at time $t_k$, each on-ramp allocates itself a ``quota" equal to the number of vehicles at that on-ramp at $t_k$. At time $t$ during the cycle, an on-ramp releases the ego vehicle if the following conditions are satisfied: 
%\mpmargin{If the distance between a successive off- and on-ramp is small, the projection method considers vehicles that have not reached the off-ramp}
\begin{itemize}
\item[(M1)] $t = k\timestep$ for some $k \in \N_{0}$. 
\item [(M2)] the on-ramp has not reached its quota. 
\item[(M3)] $y_e(t) \geq S_e(t)$, i.e., \mpcomment{the ego is at a safe distance with respect to its leading vehicle (cf. (VC1)(b))}.
    %its distance to the leading vehicle is no less than the distance to the first acceleration lane slot.
\item[(M4)] it predicts that the ego vehicle will be at a safe distance with respect to its virtual leading and following vehicles between merging and exiting the acceleration lane (cf. (VC1)(a)).
\end{itemize}
\par
Once an on-ramp reaches its quota, \mpcomment{it pauses vehicle release until all other on-ramps reach their quotas, at which point the next cycle starts.}
\end{definition}
%%%%%%%%%%%%%%%%%%%%%%%%%%%%%%%%%%%%%%%%%%%%%%%%%%%%%%%
%%%%%%%%%%%%%%%%%%%%% Definition %%%%%%%%%%%%%%%%%%%%%%
%%%%%%%%%%%%%%%%%%%%%%%%%%%%%%%%%%%%%%%%%%%%%%%%%%%%%%%
%%%%%%%%%%%%%%%%%%%%%%%%%%%%%%%%%%%%%%%%%%%%%%%%%%%%%%%
%%%%%%%%%%%%%%%%%%%%%% Remark %%%%%%%%%%%%%%%%%%%%%%%%%
%%%%%%%%%%%%%%%%%%%%%%%%%%%%%%%%%%%%%%%%%%%%%%%%%%%%%%%
\begin{remark}
A simpler form of this policy, called the \emph{Quota} policy, is analyzed in \citet{georgiadis1995scheduling}. \mpcomment{ In order to apply the Quota policy directly to the current transportation setting, additional analysis is required that considers the vehicle dynamics. This analysis is done in the next theorem.}
\end{remark}
%%%%%%%%%%%%%%%%%%%%%%%%%%%%%%%%%%%%%%%%%%%%%%%%%%%%%%%
%%%%%%%%%%%%%%%%%%%%%% Remark %%%%%%%%%%%%%%%%%%%%%%%%%
%%%%%%%%%%%%%%%%%%%%%%%%%%%%%%%%%%%%%%%%%%%%%%%%%%%%%%%
We introduce an additional notation for the next result. Consider a situation where on-ramp $i$ releases the ego vehicle under (M1)-(M4), and its virtual leading and following vehicles are moving at the speed $\speedlim$, both occupying mainline slots. We let $\timestep_i$ be the minimum time headway between the front bumpers of the virtual leading and following vehicles, ensuring that they maintain a safe distance from the ego vehicle after merging. Note that \mpcomment{since both the virtual leading and following vehicles are assumed to occupy mainline slots}, $\timestep_i$ is an exact multiple of $\timestep$ and $\timestep_i \geq 2\timestep$, where the equality holds if and only if the merging speed at on-ramp $i$ is $\speedlim$.
\mpcommentout{
We let $\timestep_i$ be the minimum time headway between the front bumpers of the two vehicles at the moment of merging, such that if the ego vehicle remains in the $\Speedmode$ mode, its projected following vehicle does not violate the safety distance in the future.  
}
%%%%%%%%%%%%%%%%%%%%%%%%%%%%%%%%%%%%%%%%%%%%%%%%%%%%%%%%%%%%
%%%%%%%%%%%%%%%%%%%%%%%% Proposition %%%%%%%%%%%%%%%%%%%%%%%
%%%%%%%%%%%%%%%%%%%%%%%%%%%%%%%%%%%%%%%%%%%%%%%%%%%%%%%%%%%%
%%%%%%%%%%%%%%%%%%%%%%%%%%%%%%%%%%%%%%%%%%%%%%%%%%%%%%%%%%%%
\begin{theorem}\label{Prop: stability of QRM policy for low merging speed}
For any initial condition, the Renewal policy keeps the freeway under-saturated if
%\begin{equation*}
 $(\frac{\timestep_i}{\timestep}-1) \Avgload{i} - (\frac{\timestep_i}{\timestep}-2)\Arrivalrate{i} < 1$ for all $i \in [\numramps]$.
%\end{equation*}
\end{theorem}
\begin{proof}
See Appendix \ref{Section: (Appx) Proof of Q-RM prop}.
\end{proof}
%%%%%%%%%%%%%%%%%%%%%%%%%%%%%%%%%%%%%%%%%%%%%%%%%%%%%%%%%%%%
%%%%%%%%%%%%%%%%%%%%%%%% Proposition %%%%%%%%%%%%%%%%%%%%%%%
%%%%%%%%%%%%%%%%%%%%%%%%%%%%%%%%%%%%%%%%%%%%%%%%%%%%%%%%%%%%
%%%%%%%%%%%%%%%%%%%%%%%%%%%%%%%%%%%%%%%%%%%%%%%%%%%%%%%%%%%%
\par\textbf{V2I communication requirements}: the Renewal policy requires the vector queue sizes $\Queuelength{}=[\Queuelength{i}(t)]$ and the state of all the vehicles $\StateofDySys$. Its worst-case communication cost is calculated as follows: at each time step during a cycle, any vehicle that is on the mainline or an acceleration lane must communicate its state to all on-ramps. This information is used both for \mpcomment{safety distance} evaluation (cf. (M3)-(M4)) and verifying that the vehicles are in the free flow state. After a finite time, the number of vehicles that communicate their state is no more than $\numcells + \numaccslots{a}$. 
%%%%%%%%%%%%%%%%% Commented (Rev1)(02/20/23) %%%%%%%%%%%%%%%%%
\mpcommentout{
at each time step during a cycle, any vehicle that is on the mainline or an acceleration lane must communicate its state to all on-ramps. After a finite time, the number of these vehicles is no more than $\numcells + \numaccslots{a}$.}
%%%%%%%%%%%%%%%%%%%%%%%%%%%%%%%%%%%%%%%%%%%%%%%%%%%%%%%%%%%%
Furthermore, at the start of every cycle, all the vehicles in each on-ramp queue must communicate their presence in the queue to that on-ramp. The contribution of the on-ramp queue to the communication cost is
\begin{equation*}
    c_{1} := \limsup_{K \ra \infty}\frac{1}{K}\sum_{t_k \leq K\timestep}\sum_{i\in [\numramps]}\Queuelength{i}(t_k),
\end{equation*}
where $t_k$ is the start of the $k$-th cycle in the Renewal policy. Hence, the communication cost $C$ is upper-bounded by $\numcells + \numaccslots{a} + c_{1}$ transmissions per $\timestep$ seconds.
%%%%%%%%%%%%%%%%%%%%%%%%%%%%%%%%%%%%%%%%%%%%%%%%%%%%%%%
%%%%%%%%%%%%%%%%%%%%% Figure %%%%%%%%%%%%%%%%%%%%%%%%%%
%%%%%%%%%%%%%%%%%%%%%%%%%%%%%%%%%%%%%%%%%%%%%%%%%%%%%%%
\begin{figure}[t!]
        \centering
        \includegraphics[width=0.4\textwidth]{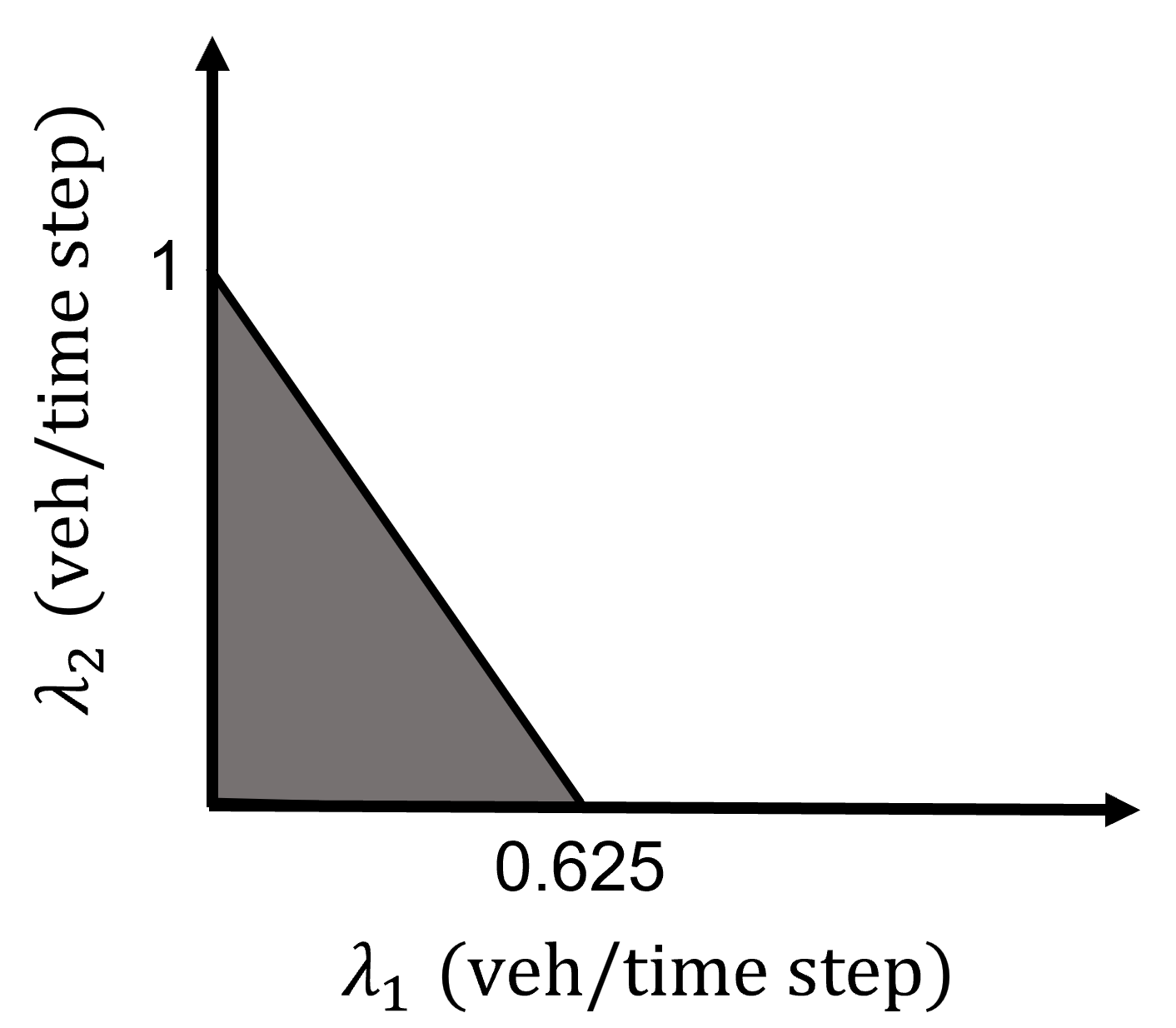}
        \vspace{0.2 cm}
    \caption{\centering\sf An inner-estimate of the under-saturation region of the Renewal policy (grey area) from Example~\ref{Ex: RRM example}.}\label{Fig: Numerical example for RRM}
\end{figure}
%%%%%%%%%%%%%%%%%%%%%%%%%%%%%%%%%%%%%%%%%%%%%%%%%%%%%%%
%%%%%%%%%%%%%%%%%%%%% Figure %%%%%%%%%%%%%%%%%%%%%%%%%%
%%%%%%%%%%%%%%%%%%%%%%%%%%%%%%%%%%%%%%%%%%%%%%%%%%%%%%%
%%%%%%%%%%%%%%%%%%%%%%%%%%%%%%%%%%%%%%%%%%%%%%%%%%%%%%%
%%%%%%%%%%%%%%%%%%%%%% Remark %%%%%%%%%%%%%%%%%%%%%%%%%
%%%%%%%%%%%%%%%%%%%%%%%%%%%%%%%%%%%%%%%%%%%%%%%%%%%%%%%
\begin{remark}
Under the constant time headway safety distance rule in Section~\ref{sec:vehicle-control}, the flow capacity of the mainline is $1$ vehicle per $\timestep$ seconds. Theorem~\ref{Prop: stability of QRM policy for low merging speed} provides an inner-estimate of the under-saturation region in terms of the induced loads $\Avgload{i}$, arrival rates $\Arrivalrate{i}$, and the mainline flow capacity. In particular, if for every $i \in [\numramps]$, $(\timestep_i/\timestep-1)\Avgload{i} - (\timestep_i/\timestep-2)\Arrivalrate{i}$ is less than the flow capacity, then the Renewal policy keeps the freeway under-saturated.
\end{remark}
%%%%%%%%%%%%%%%%%%%%%%%%%%%%%%%%%%%%%%%%%%%%%%%%%%%%%%%
%%%%%%%%%%%%%%%%%%%%%% Remark %%%%%%%%%%%%%%%%%%%%%%%%%
%%%%%%%%%%%%%%%%%%%%%%%%%%%%%%%%%%%%%%%%%%%%%%%%%%%%%%%
%%%%%%%%%%%%%%%%%%%%%%%%%%%%%%%%%%%%%%%%%%%%%%%%%%%%%%%
%%%%%%%%%%%%%%%%%%%%% Example %%%%%%%%%%%%%%%%%%%%%%%%%
%%%%%%%%%%%%%%%%%%%%%%%%%%%%%%%%%%%%%%%%%%%%%%%%%%%%%%%
\begin{example}\label{Ex: RRM example}
Let $\numramps=3$ and suppose that $\timestep_1=\timestep_3=2\timestep$, $\timestep_2=3\timestep$, i.e., the merging speed at on-ramps $1$ and $3$ are $\speedlim$ and is lower at on-ramp $2$. Let
\begin{equation*}
    \routingmatrix = \begin{pmatrix}
    0.2 & 0.7 & 0.1 \\
    0 & 0.8 & 0.2 \\
    0.5 & 0 & 0.5
    \end{pmatrix},~ \Arrivalrate{3}=0.5~[\text{veh}/\text{time step}].
\end{equation*}
\par
Then, the inner-estimate of the under-saturation region given by Theorem~\ref{Prop: stability of QRM policy for low merging speed} is
\begin{equation*}
\begin{aligned}
    \{(\Arrivalrate{1},\Arrivalrate{2}): \Avgload{1} < 1,~ 2\Avgload{2}-\Arrivalrate{2}<1,~ \Avgload{3}<1\} &= \{(\Arrivalrate{1},\Arrivalrate{2}): \Arrivalrate{1}+0.5\Arrivalrate{2} < 1,~ 1.6\Arrivalrate{1}+\Arrivalrate{2}<1,~ 0.1\Arrivalrate{1}+0.2\Arrivalrate{2}<1\} \\
    & =\{(\Arrivalrate{1},\Arrivalrate{2}):  1.6\Arrivalrate{1}+\Arrivalrate{2}<1\},
    \end{aligned}
\end{equation*}
which is illustrated in Figure~\ref{Fig: Numerical example for RRM}.
\end{example}
%%%%%%%%%%%%%%%%%%%%%%%%%%%%%%%%%%%%%%%%%%%%%%%%%%%%%%%
%%%%%%%%%%%%%%%%%%%%% Example %%%%%%%%%%%%%%%%%%%%%%%%%
%%%%%%%%%%%%%%%%%%%%%%%%%%%%%%%%%%%%%%%%%%%%%%%%%%%%%%%
%%%%%%%%%%%%%%%%%%%%%%%%%%%%%%%%%%%%%%%%%%%%%%%%%%%%%%%
%%%%%%%%%%%%%%%%%%%%%% Remark %%%%%%%%%%%%%%%%%%%%%%%%%
%%%%%%%%%%%%%%%%%%%%%%%%%%%%%%%%%%%%%%%%%%%%%%%%%%%%%%%
%\begin{remark}
\mpcomment{Recall that one of our contributions is understanding the interplay of safety and throughput.} This interplay is captured using the parameter $\timestep_i$, $i \in [\numramps]$. In particular, as the merging speed at an on-ramp decreases, the required safety distance $S_e$ in \eqref{eq:safety-distance} increases, \mpcomment{which would increase $\timestep_i$.} This puts a limit on the rate at which the on-ramp can release new vehicles under the vehicle following safety constraint, which in turn decreases the throughput. 
%\end{remark}
%%%%%%%%%%%%%%%%%%%%%%%%%%%%%%%%%%%%%%%%%%%%%%%%%%%%%%%
%%%%%%%%%%%%%%%%%%%%%% Remark %%%%%%%%%%%%%%%%%%%%%%%%%
%%%%%%%%%%%%%%%%%%%%%%%%%%%%%%%%%%%%%%%%%%%%%%%%%%%%%%%

%%%%%%%%%%%%%%%%%%%%%%%%%%%%%%%%%%%%%%%%%%%%%%%%%%%%%%%
%%%%%%%%%%%%%%%%%%%%%% Remark %%%%%%%%%%%%%%%%%%%%%%%%%
%%%%%%%%%%%%%%%%%%%%%%%%%%%%%%%%%%%%%%%%%%%%%%%%%%%%%%%
\begin{remark}
The Renewal policy has a variable cycle length that depends on queue sizes at the start of the cycle, with larger queue sizes generally leading to longer cycles. The impact of a variable cycle length on performance can be both positive and negative. On one hand, it can increase the total travel time because, any arrival during a cycle is delayed until the next cycle starts. To avoid this issue, one may enforce a \emph{fixed} cycle length (independent of the queue sizes). However, this could have a negative impact on throughput. In particular, a variable cycle length increases the chance of vehicles being released in platoons, rather than individually. Such a platoon release can be more efficient in using the space on the freeway since it can increase the release rate of the on-ramps, thus improving the throughput. For example, let the merging speed at on-ramp $i$ be less than $\speedlim$ such that $\tau_i=3\timestep$ seconds, and suppose that two vehicles are waiting in on-ramp $i$'s queue. Each vehicle requires a time headway of at least $\tau_i$ seconds between the mainline vehicles to safely merge between them. Therefore, the individual release requires a time headway of $2\tau_i = 6\timestep$ seconds in total. However, the platoon release only requires a time headway of $\tau_i+\timestep = 4\timestep$ seconds. As we will show, this phenomenon can result in a better throughput under the Renewal policy compared to the other policies in the paper which use a fixed cycle length.
%This is illustrated by the following example: consider a scenario where two vehicles that are moving at the speed $\speedlim$ approach the $i$-th on-ramp with the merging speed less than $\speedlim$. If the two vehicles are in platoon formation,  then the required time until all the three vehicles cross the merging point is $\timestep_i + \timestep$. However, when the two vehicles are not in platoon formation, this time is larger. 
\end{remark}
%%%%%%%%%%%%%%%%%%%%%%%%%%%%%%%%%%%%%%%%%%%%%%%%%%%%%%%
%%%%%%%%%%%%%%%%%%%%%% Remark %%%%%%%%%%%%%%%%%%%%%%%%%
%%%%%%%%%%%%%%%%%%%%%%%%%%%%%%%%%%%%%%%%%%%%%%%%%%%%%%%

%%%%%%%%%%%%%%%%%%%%%%%%%%%%%%%%%%%%%%%%%%%%%%%%%%%%%%%
%%%%%%%%%%%%%%%%%%%%%% Remark %%%%%%%%%%%%%%%%%%%%%%%%%
%%%%%%%%%%%%%%%%%%%%%%%%%%%%%%%%%%%%%%%%%%%%%%%%%%%%%%%
\mpcommentout{
\begin{remark}
One of the features of the Renewal policy is that a new cycle starts when all the on-ramps reach their quotas, i.e., when all the vehicles from the previous cycle are released. This may lead to scenarios where a single on-ramp with large enough quotas prevent all the other on-ramps from releasing new vehicles until the next cycle starts, thus leaving some parts of the freeway empty which is not practical. In the following sections, we introduce policies that use a \emph{fixed} cycle length in order avoid these situations. 
\end{remark}
}
%%%%%%%%%%%%%%%%%%%%%%%%%%%%%%%%%%%%%%%%%%%%%%%%%%%%%%%
%%%%%%%%%%%%%%%%%%%%%% Remark %%%%%%%%%%%%%%%%%%%%%%%%%
%%%%%%%%%%%%%%%%%%%%%%%%%%%%%%%%%%%%%%%%%%%%%%%%%%%%%%%

\subsection{Dynamic Release Rate Policy}\label{sec:release-rate}
This policy imposes dynamic minimum time gap criterion, in addition to (M1), between release of successive vehicles from the same on-ramp. Changing the time gap between release of successive vehicles by an on-ramp is similar to changing its release rate, and hence the name of the policy. 

%%%%%%%%%%%%%%%%%%%%%%%%%%%%%%%%%%%%%%%%%%%%%%%%%%%%%%%
%%%%%%%%%%%%%%%%%%%%% Definition %%%%%%%%%%%%%%%%%%%%%%
%%%%%%%%%%%%%%%%%%%%%%%%%%%%%%%%%%%%%%%%%%%%%%%%%%%%%%%
\begin{definition}\label{Def: Centralized Adaptive Fixed-Cycle Quota RM policy}
\textbf{(Dynamic Release Rate ($\GenRM$) ramp metering policy)} 
The policy works in cycles of \emph{fixed} length $\Cyclelength \timestep$, where $\Cyclelength \in \N$. At the start of the $k$-th cycle at $t_k = (k-1)\Cyclelength \timestep$, each on-ramp allocates itself a ``quota" equal to the number of vehicles at that on-ramp at $t_k$. At time $t \in [t_k, t_{k+1}]$ during the $k$-th cycle, on-ramp $i$ releases the ego vehicle if (M1)-(M4) and the following condition are satisfied:
\begin{itemize}
\item[(M5)] at least $\Releasetime{}(t)$ seconds has passed since the release of the last vehicle from on-ramp $i$, where $g(t)$ is a piecewise constant minimum time gap, updated periodically at $t = \updateperiod, 2\updateperiod, \ldots$, as described in Algorithm~\ref{Alg: Gen-aFCQ-RM policy}.
\end{itemize}
\par
Once an on-ramp reaches its quota, it pauses release during the rest of the cycle.
\par
%If $\delta_e + \hat{\delta}_e$ is zero for all vehicles in $[t-\updateperiod,t]$, then $\StateofDySys_{f_2}(t)=0$.
%%%%%%%%%%%%%%%%% Commented (V2)(10/31/22) %%%%%%%%%%%%%%%%%%%
\mpcommentout{
For the ego vehicle, let
\begin{equation*}
x_e^G(t):= 
\begin{cases}
\left(\min\{0, y_e(t) - S_e(t), \inf_{s \in [t_m, t_f]}\hat{y}_e(s) - \hat{S}_e(s)\}, v_e(t) - \speedlim, a_e(t)\right) & \text{if } I_e(t)=0 \\
\left(\min\{0, y_e(t) - S_e(t), \inf_{s \in [t_m, t_f]}\hat{y}_e(s) - \hat{S}_e(s)\}, 0, 0\right) & \text{otherwise}
\end{cases},
\end{equation*}
where the term $\inf_{s \in [t_m, t_f]}\hat{y}_e(s) - \hat{S}_e(s)$ is set to zero whenever the ego vehicle is not in a merging scenario. Note that $\min\{0, y_e(t) - S_e(t), \inf_{s \in [t_m, t_f]}\hat{y}_e(s) - \hat{S}_e(s)\} = 0$ implies: (i) $y_e(t) \geq S_e(t)$; so the ego vehicle is at a safe distance with respect to its leading vehicle at time $t$, and (ii) $\hat{y}_e(s) \geq \hat{S}_e(s)$ for all $s \in [t_m, t_f]$; so the ego vehicle predicts to be at a safe distance with respect to its virtual leading vehicle in $[t_m, t_f]$. The state $\StateofDySys^G$ is the collection of $x_e^G$ for all the vehicles on the mainline and acceleration lanes.
}
%%%%%%%%%%%%%%%%% Commented (V2)(10/31/22) %%%%%%%%%%%%%%%%%%%
%%%%%%%%%%%%%%%%%%%%%%%%%%%%%%%%%%%%%%%%%%%%%%%%%%%%%%%
%%%%%%%%%%%%%%%%%%%%% Figure %%%%%%%%%%%%%%%%%%%%%%%%%%
%%%%%%%%%%%%%%%%%%%%%%%%%%%%%%%%%%%%%%%%%%%%%%%%%%%%%%%
\begin{figure}[t!]
        \centering
        \includegraphics[width=0.40\textwidth]{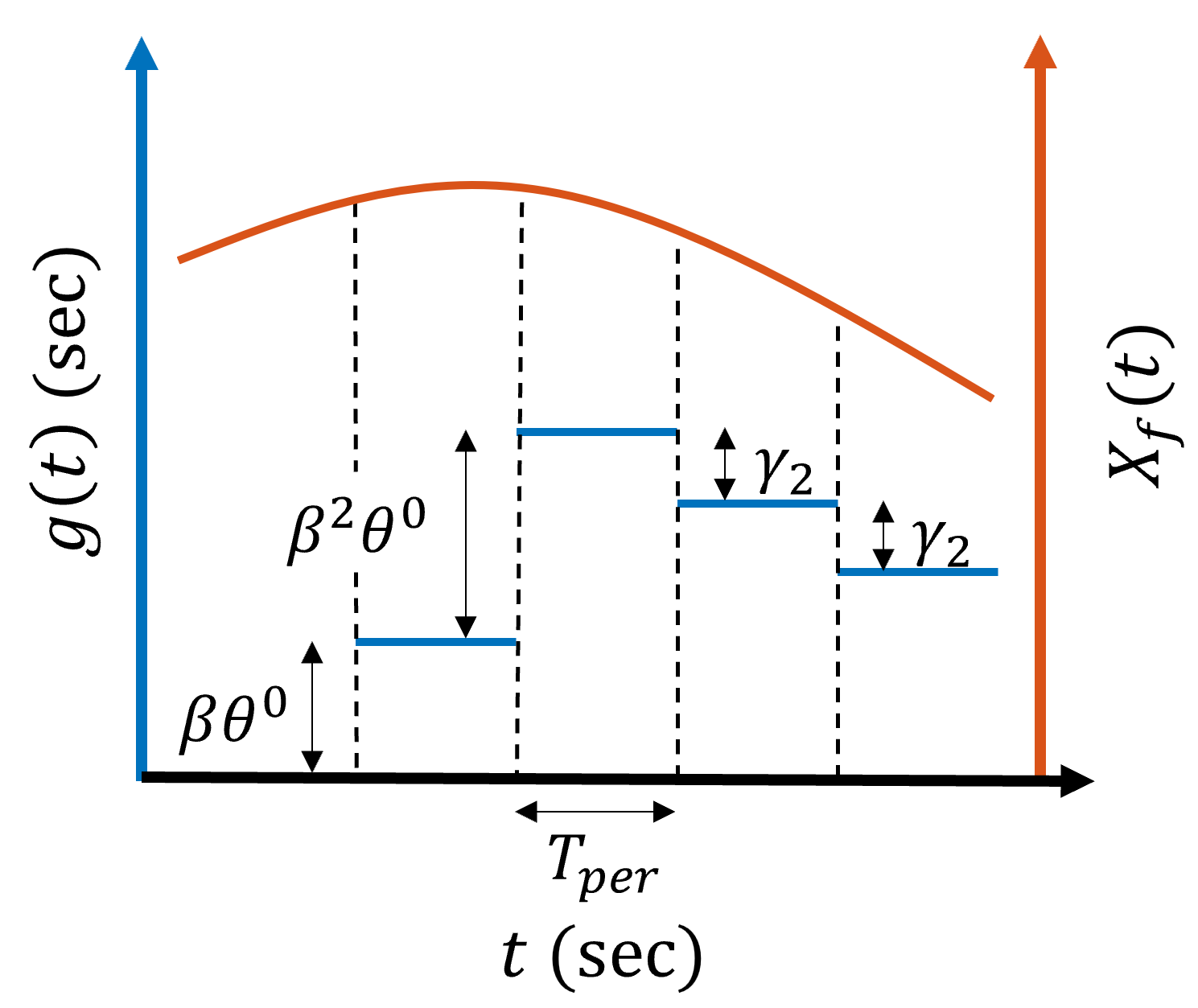}
        \vspace{0.2 cm}
    \caption{\centering\sf An illustration of the update rule for the minimum time gap $\Releasetime{}$ in Algorithm~\ref{Alg: Gen-aFCQ-RM policy}.}\label{fig:g-diagram}
\end{figure}
%%%%%%%%%%%%%%%%%%%%%%%%%%%%%%%%%%%%%%%%%%%%%%%%%%%%%%%
%%%%%%%%%%%%%%%%%%%%% Figure %%%%%%%%%%%%%%%%%%%%%%%%%%
%%%%%%%%%%%%%%%%%%%%%%%%%%%%%%%%%%%%%%%%%%%%%%%%%%%%%%%
\begin{algorithm}[H]
\caption{Update rule for the minimum time gap between release of vehicles under the $\GenRM$ policy}\label{Alg: Gen-aFCQ-RM policy}
\begin{algorithmic}
\Require \textbf{design constants:}~$\updateperiod>0, \DesDecreaseconst > 0$, $\ReleasetimeDecConst{2} > 0$, $\ReleasetimeInc{}^{\circ} > 0$, $\ReleasetimeAdj{} > 1$
\par
~\textbf{initial condition:}~
$\Releasetime{}(0) = 0,~ \ReleasetimeInc{}(0) = \ReleasetimeInc{}^{\circ}$,~$\StateofDySys_f(0) = \StateofDySys_{f_1}(0)$

\hspace{-0.475in} \textbf{for} $t = \updateperiod, 2\updateperiod, \cdots$ \textbf{do} 

%\If{$\|\StateofDySys^G(t)\| = 0$} 
%\State $\Releasetime{}(t) \gets 0$
\If{$\StateofDySys_f(t) \leq \max\{\StateofDySys_f(t-\updateperiod) - \DesDecreaseconst{}, 0\}$ } 
\State $\ReleasetimeInc{}(t) \gets \ReleasetimeInc{}(t-\updateperiod)$
\State $\Releasetime{}(t) \gets \max\{\Releasetime{}(t-\updateperiod) - \ReleasetimeDecConst{2}, 0\}$
\Else{}
\State $\ReleasetimeInc{}(t) \gets \ReleasetimeAdj{} \ReleasetimeInc{}(t-\updateperiod)$
\State $\Releasetime{}(t) \gets \Releasetime{}(t-\updateperiod) + \ReleasetimeInc{}(t)$
%\Else{}
%    \If{$\|\StateofDySys^G(t)\| < \|\StateofDySys^G(t - \updateperiod)\| - \DesDecreaseconst{}$}
%    \State $\Releasetime{}(t)  \gets \Releasetime{}(t-\updateperiod)$
%    \Else{}
%    \State $\Releasetime{}(t) \gets \Releasetime{}(t-\updateperiod)+\ReleasetimeInc{}$
%\EndIf

\EndIf

\hspace{-0.475in} \textbf{end for}
\end{algorithmic}
\end{algorithm}
In Algorithm \ref{Alg: Gen-aFCQ-RM policy}, $\StateofDySys_f(t) := \StateofDySys_{f_1}(t) + \StateofDySys_{f_2}(t)$, where $\StateofDySys_{f_1}$ and $\StateofDySys_{f_2}$ are defined as follows: $\StateofDySys_{f_1}(t) := \sum_{e \in [n]}(w_1|v_e(t)-\speedlim| + w_2|a_e(t)|)I_e(t)$, where $n$ is the number of vehicles on the mainline and acceleration lanes, and $w_1, w_2$ are normalization factors. Furthermore, for $t \geq \updateperiod$, $\StateofDySys_{f_2}(t) := \sum w_3(\delta_e(t) + \hat{\delta}_e(t))$, where $w_3$ is a normalization factor and the sum is over all the vehicles that either have: (i) violated the safety distance $S_e$ at some time in  $[t-\updateperiod,t]$, or (ii) predicted at some time in  $[t-\updateperiod,t]$ that they would violate the safety distance once they reach a merging point. The terms $\delta_e$ and $\hat{\delta}_e$ are, respectively, the maximum error and maximum predicted error in the relative spacing, and are given by 
\begin{equation*}
\begin{aligned}
    \delta_e(t) &= \max_{t' \in [t-\updateperiod, t]}|y_e(t')-S_e(t')|\mathbbm{1}_{\{y_e(t') < S_e(t')\}}, \\
    \hat{\delta}_e(t) &= \max_{t' \in [t-\updateperiod, t]}|\hat{y}_e(t_m|t')-\hat{S}_e(t_m|t')|\mathbbm{1}_{\{\hat{y}_e(t_m|t') < \hat{S}_e(t_m|t')\}},
\end{aligned}
\end{equation*}
where $\mathbbm{1}$ is the indicator function, and $t_m|t'$ is the predicted time of crossing a merging point based on the information available at time $t' \leq t_m$. If $\delta_e + \hat{\delta}_e$ is non-zero, then the ego vehicle communicates it either right before leaving the network, or at the update times $\updateperiod, 2\updateperiod, \ldots$, whichever comes earlier. Otherwise, it is not communicated by the ego vehicle.
\end{definition}
%%%%%%%%%%%%%%%%%%%%%%%%%%%%%%%%%%%%%%%%%%%%%%%%%%%%%%%
%%%%%%%%%%%%%%%%%%%%% Definition %%%%%%%%%%%%%%%%%%%%%%
%%%%%%%%%%%%%%%%%%%%%%%%%%%%%%%%%%%%%%%%%%%%%%%%%%%%%%%
%%%%%%%%%%%%%%%%%%%%%%%%%%%%%%%%%%%%%%%%%%%%%%%%%%%%%%%
%%%%%%%%%%%%%%%%%%%%%% Remark %%%%%%%%%%%%%%%%%%%%%%%%%
%%%%%%%%%%%%%%%%%%%%%%%%%%%%%%%%%%%%%%%%%%%%%%%%%%%%%%%
\begin{remark}\label{remark:X_f-bound}
Recall the safety distance $S_e$ in \eqref{eq:safety-distance}. Since the speeds of the vehicles are bounded, $S_e$ is also bounded. Hence, $\delta_e$ and $\hat{\delta}_e$ in $\StateofDySys_{f_2}$ are bounded. We use this in the proof of the next result.
\end{remark}
%%%%%%%%%%%%%%%%%%%%%%%%%%%%%%%%%%%%%%%%%%%%%%%%%%%%%%%
%%%%%%%%%%%%%%%%%%%%%% Remark %%%%%%%%%%%%%%%%%%%%%%%%%
%%%%%%%%%%%%%%%%%%%%%%%%%%%%%%%%%%%%%%%%%%%%%%%%%%%%%%%
%%%%%%%%%%%%%%%%%%%%%%%%%%%%%%%%%%%%%%%%%%%%%%%%%%%%%%%%%%%%
%%%%%%%%%%%%%%%%%%%%%%%% Proposition %%%%%%%%%%%%%%%%%%%%%%%
%%%%%%%%%%%%%%%%%%%%%%%%%%%%%%%%%%%%%%%%%%%%%%%%%%%%%%%%%%%%
%%%%%%%%%%%%%%%%%%%%%%%%%%%%%%%%%%%%%%%%%%%%%%%%%%%%%%%%%%%%
\begin{theorem}\label{Prop: stability of general aFCQ-RM}
For any initial condition, $\Cyclelength \in \N$, and design constants in Algorithm~\ref{Alg: Gen-aFCQ-RM policy}, the $\GenRM$ policy keeps the freeway under-saturated if 
%\begin{equation*}
   $(\frac{\timestep_i}{\timestep}-1) \Avgload{i} < 1$ for all $i \in [\numramps]$.
%\end{equation*}
\end{theorem} 
\begin{proof}
See Appendix \ref{Section: (Appx) Proof of general aFCQ-RM prop}.
\end{proof}
%%%%%%%%%%%%%%%%%%%%%%%%%%%%%%%%%%%%%%%%%%%%%%%%%%%%%%%%%%%%
%%%%%%%%%%%%%%%%%%%%%%%% Proposition %%%%%%%%%%%%%%%%%%%%%%%
%%%%%%%%%%%%%%%%%%%%%%%%%%%%%%%%%%%%%%%%%%%%%%%%%%%%%%%%%%%%
%%%%%%%%%%%%%%%%%%%%%%%%%%%%%%%%%%%%%%%%%%%%%%%%%%%%%%%%%%%%
\par
\textbf{V2I communication requirements}: The $\GenRM$ policy requires the vector queue sizes $\Queuelength{}=[\Queuelength{i}(t)]$ (if $\Cyclelength > 1$) and the state of all the vehicles $\StateofDySys$. Its worst-case communication cost is calculated as follows: after a finite time, $\StateofDySys$ is communicated to all on-ramps only at the end of each update period $\updateperiod$. After such finite time, the number of vehicles that constitute $\StateofDySys$ is no more than $\numcells + \numaccslots{a}$. Furthermore, at each time step during a cycle, the vehicles in the merging area of each on-ramp communicate their state to that on-ramp for \mpcomment{safety distance} evaluation (cf. (M3)-(M4)). The number of vehicles in all the merging areas is at most $\numaccslots{m}$. Finally, if $\Cyclelength > 1$, then all the vehicles in each on-ramp queue must communicate their presence in the queue to that on-ramp at the start of every cycle. The contribution of the queue size to the communication cost is
\begin{equation*}
    c_{2} := \left(\limsup_{K \ra \infty}\frac{1}{K}\sum_{k \leq K/\Cyclelength+1}\sum_{i\in [\numramps]}\Queuelength{i}((k-1)\Cyclelength \timestep)\right)\mathbbm{1}_{\{\Cyclelength > 1\}}.
\end{equation*}
\par
Hence, $C$ is upper bounded by $\numramps(\numcells + \numaccslots{a})/\updateperiod + \numaccslots{m} + c_{2}$ transmissions per $\timestep$ seconds.
%%%%%%%%%%%%%%%%%%%%%%%%%%%%%%%%%%%%%%%%%%%%%%%%%%%%%%%
%%%%%%%%%%%%%%%%%%%%% Figure %%%%%%%%%%%%%%%%%%%%%%%%%%
%%%%%%%%%%%%%%%%%%%%%%%%%%%%%%%%%%%%%%%%%%%%%%%%%%%%%%%
\begin{figure}[t!]
        \centering
        \includegraphics[width=0.4\textwidth]{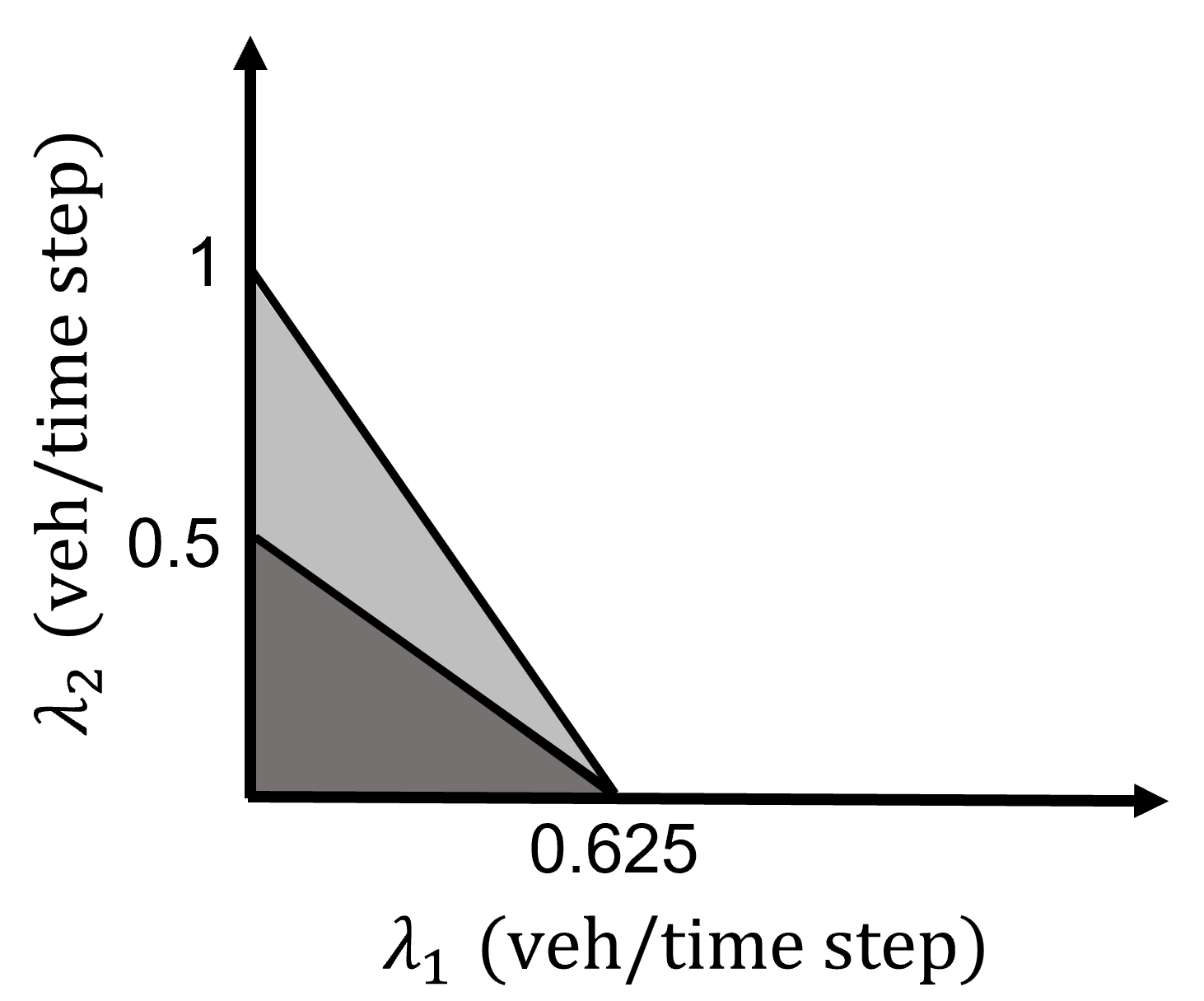}
        \vspace{0.2 cm}
    \caption{\centering\sf An inner-estimate of the under-saturation region (dark grey area) of the $\GenRM$ policy from Example~\ref{Ex: DRR-RM example}. The light grey area indicates the additional under-saturation region if we use the Renewal policy.}\label{Fig: Numerical example for DRRM}
\end{figure}
%%%%%%%%%%%%%%%%%%%%%%%%%%%%%%%%%%%%%%%%%%%%%%%%%%%%%%%
%%%%%%%%%%%%%%%%%%%%% Figure %%%%%%%%%%%%%%%%%%%%%%%%%%
%%%%%%%%%%%%%%%%%%%%%%%%%%%%%%%%%%%%%%%%%%%%%%%%%%%%%%%
%%%%%%%%%%%%%%%%%%%%%%%%%%%%%%%%%%%%%%%%%%%%%%%%%%%%%%%
%%%%%%%%%%%%%%%%%%%%%% Remark %%%%%%%%%%%%%%%%%%%%%%%%%
%%%%%%%%%%%%%%%%%%%%%%%%%%%%%%%%%%%%%%%%%%%%%%%%%%%%%%%
\begin{remark}
Similar to Theorem~\ref{Prop: stability of QRM policy for low merging speed}, Theorem~\ref{Prop: stability of general aFCQ-RM} provides an inner-estimate of the under-saturation region of the $\GenRM$ policy in terms of the induced loads $\Avgload{i}$ and the mainline flow capacity. The region specified by this estimate is the same for all $\Cyclelength \in \N$ and is contained in the one given for the Renewal policy in Theorem~\ref{Prop: stability of QRM policy for low merging speed}. However, this does not necessarily mean that the throughput of the $\GenRM$ policy is the same for different cycle lengths, or the Renewal policy gives a better throughput as the inner-estimates may not be exact. A simulation comparison of the throughput of the $\GenRM$ policy for different cycle lengths is provided in Section~\ref{Subsection: (Sim) Avg Q lengths}.
\end{remark}
%%%%%%%%%%%%%%%%%%%%%%%%%%%%%%%%%%%%%%%%%%%%%%%%%%%%%%%
%%%%%%%%%%%%%%%%%%%%%% Remark %%%%%%%%%%%%%%%%%%%%%%%%%
%%%%%%%%%%%%%%%%%%%%%%%%%%%%%%%%%%%%%%%%%%%%%%%%%%%%%%%
%%%%%%%%%%%%%%%%%%%%%%%%%%%%%%%%%%%%%%%%%%%%%%%%%%%%%%%
%%%%%%%%%%%%%%%%%%%%% Example %%%%%%%%%%%%%%%%%%%%%%%%%
%%%%%%%%%%%%%%%%%%%%%%%%%%%%%%%%%%%%%%%%%%%%%%%%%%%%%%%
\begin{example}\label{Ex: DRR-RM example}
Let the network parameters be as in Example~\ref{Ex: RRM example}. Then, the inner-estimate of the under-saturation region given by Theorem~\ref{Prop: stability of general aFCQ-RM} is 
\begin{equation*}
\begin{aligned}
    \{(\Arrivalrate{1},\Arrivalrate{2}): \Avgload{1} < 1,~ 2\Avgload{2}<1,~ \Avgload{3}<1\} &= \{(\Arrivalrate{1},\Arrivalrate{2}): \Arrivalrate{1}+0.5\Arrivalrate{2} < 1,~ 1.6\Arrivalrate{1}+2\Arrivalrate{2}<1,~ 0.1\Arrivalrate{1}+0.2\Arrivalrate{2}<1\} \\
    & =\{(\Arrivalrate{1},\Arrivalrate{2}):  1.6\Arrivalrate{1}+2\Arrivalrate{2}<1\},
    \end{aligned}
\end{equation*}
which is illustrated in Figure~\ref{Fig: Numerical example for DRRM} and compared with the inner-estimate of the Renewal policy in Example~\ref{Ex: RRM example}. 
%Also, comparing with Example~\ref{ex:throughput} where the merging speed at on-ramp $2$ was also $\speedlim$, one can see that the region specified in this example is smaller.
\end{example}
%%%%%%%%%%%%%%%%%%%%%%%%%%%%%%%%%%%%%%%%%%%%%%%%%%%%%%%
%%%%%%%%%%%%%%%%%%%%% Example %%%%%%%%%%%%%%%%%%%%%%%%%
%%%%%%%%%%%%%%%%%%%%%%%%%%%%%%%%%%%%%%%%%%%%%%%%%%%%%%%
\mpcomment{The $\GenRM$ policy is coordinated since each on-ramp requires the state of all the vehicles. The next policy is a distributed version of the $\GenRM$ policy, where each on-ramp receives information from its downstream on-ramps and the vehicles in its vicinity.}

\subsection{Distributed Dynamic Release Rate Policy}\label{sec:distributed}
This policy imposes dynamic minimum time gap criterion, just like its coordinated counterpart. However, each on-ramp only receives information from its downstream on-ramps and the vehicles in its vicinity.   

\begin{definition}\label{Def: Distributed Adaptive FCQ-RM}
\textbf{(Distributed Dynamic Release Rate ($\DisDRRRM$) ramp metering policy)} 
The policy works in cycles of fixed length $\Cyclelength\timestep$, where $\Cyclelength \in \N$. At the start of the $k$-th cycle at $t_k = (k-1)\Cyclelength \timestep$, each on-ramp allocates itself a ``quota" equal to the number of vehicles at that on-ramp at $t_k$. At time $t \in [t_k, t_{k+1}]$ during the $k$-th cycle, on-ramp $i$ releases the ego vehicle if (M1)-(M4) and the following condition are satisfied:
\begin{itemize}
\item[(M5)] at least $\Releasetime{i}(t)$ seconds has passed since the release of the last vehicle from on-ramp $i$, where $\Releasetime{i}(t)$ is a piecewise constant minimum time gap, updated periodically at $t = \updateperiod, 2\updateperiod, \ldots$ according to Algorithm \ref{Alg: Distributed aFCQ-RM policy}.
\end{itemize}
\par
Once an on-ramp reaches its quota, it pauses release during the rest of the cycle. 
\begin{algorithm}[H]
\caption{Update rule for the minimum time gap between release of vehicles under the $\DisDRRRM$ policy}\label{Alg: Distributed aFCQ-RM policy}
\begin{algorithmic}
\Require \textbf{design constants:}~$\updateperiod>0$, $\Tmax > 0$, $\DesDecreaseconst > 0$, $\ReleasetimeDecConst{2} > 0$, $\ReleasetimeInc{i}^{\circ} > 0$, $\ReleasetimeAdj{} > 1$ 
\par
~\textbf{initial condition:}~$(\Releasetime{i}(0), \Releasetime{i}(\updateperiod)) = (0, 0),~ (\ReleasetimeInc{i}(0), \ReleasetimeInc{i}(\updateperiod)) = (\ReleasetimeInc{i}^{\circ}, \ReleasetimeInc{i}^{\circ}),~ \StateofDySys_{f}^{i}(0) = \StateofDySys_{f_1}^{i}(0),~ i \in [\numramps]$

\hspace{-0.475in} \textbf{for} $t = 2\updateperiod, 3\updateperiod, \cdots$ \textbf{do} the following for each on-ramp $i \in [\numramps]$
\If{$\StateofDySys_{f}^{i}(t) \leq \max\{\StateofDySys_{f}^{i}(t-\updateperiod) - \DesDecreaseconst, 0\}$} 
\If{$\Releasetime{i+1}(t - \updateperiod) \leq \Tmax$~\textbf{or}~$\sum_{j > i}\StateofDySys_{f}^{j}(t) \leq \max\{\sum_{j > i}\StateofDySys_{f}^{j}(t-\updateperiod) - \DesDecreaseconst, 0\}$} 
\State $\ReleasetimeInc{i}(t) \gets \ReleasetimeInc{i}(t-\updateperiod)$
\State $\Releasetime{i}(t) \gets \max\{\Releasetime{i}(t-\updateperiod) - \ReleasetimeDecConst{2}, 0\}$
\Else{}
\State $\ReleasetimeInc{i}(t) \gets \ReleasetimeAdj{} \ReleasetimeInc{i}(t-\updateperiod)$
\State $\Releasetime{i}(t) \gets \Releasetime{i}(t-\updateperiod) + \ReleasetimeInc{i}(t)$
\EndIf
\Else{}
\If{$\StateofDySys_{f}^{i}(t-\updateperiod) \leq \max\{\StateofDySys_{f}^{i}(t-2\updateperiod) - \DesDecreaseconst, 0\}$}
\State $\ReleasetimeInc{i}(t) \gets \ReleasetimeAdj{} \ReleasetimeInc{i}(t-\updateperiod)$
%\State $\Releasetime{j}(t) \gets \Releasetime{j}(t-\updateperiod) + \ReleasetimeInc(t)$
%\algstore{alg2}
%\end{algorithmic}
%\end{algorithm}
%\begin{algorithm}
%\begin{algorithmic}
%\algrestore{alg2}
\Else{}
\State $\ReleasetimeInc{i}(t) \gets \ReleasetimeInc{i}(t-\updateperiod)$
%\State $\Releasetime{j}(t) \gets \Releasetime{j}(t-\updateperiod) + \ReleasetimeInc(t)$
\EndIf
\State $\Releasetime{i}(t) \gets \Releasetime{i}(t-\updateperiod) + \ReleasetimeInc{i}(t)$
\EndIf

\hspace{-0.475in} \textbf{end for}
\end{algorithmic}
\end{algorithm}
\par
For $i \in [\numramps]$, let $\StateofDySys_{f}^i$ be the part of $\StateofDySys_f$ associated with all the vehicles located between the $i$-th and $(i+1)$-th on-ramps. Thus, $\StateofDySys_{f} = \sum_{i \in [\numramps]}\StateofDySys_{f}^i$. We assume that $\StateofDySys_{f}^i$ is available to on-ramp $i$. Furthermore, if $\Releasetime{i+1}(t) > \Tmax$ for some design constant $\Tmax$, then all the on-ramps $j$ downstream of on-ramp $i$ communicate $\StateofDySys_{f}^j$ to on-ramp $i$. In that case, $\sum_{j > i}\StateofDySys_{f}^j$ is available to on-ramp $i$, where the notation ``$j > i$" means that on-ramp $j$ is downstream of on-ramp $i$. Note that for the ring road configuration, all the on-ramps downstream of on-ramp $i$ is the same as all the on-ramps. Hence, $\sum_{j > i}\StateofDySys_{f}^j = \StateofDySys_f$ for the ring road configuration.
%Note that for the ring road geometry, $\sum_{j > i}\StateofDySys_{f}^j = \StateofDySys_f$.
Informally, when $\Releasetime{i+1} \leq \Tmax$, $\Releasetime{i}$ depends only on the traffic condition in the vicinity of on-ramp $i$, whereas when $\Releasetime{i+1} > \Tmax$, it also depends on the downstream traffic condition. Naturally, higher values of $\Tmax$ make the policy more ``decentralized", but it may take longer to reach the free flow state.
\end{definition}

\begin{proposition}\label{Prop: stability of distributed aFCQ-RM}
For any initial condition, $\Cyclelength \in \N$, and design constants in Algorithm~\ref{Alg: Distributed aFCQ-RM policy}, the $\DisDRRRM$ policy keeps the freeway under-saturated if $(\frac{\timestep_i}{\timestep}-1) \Avgload{i} < 1$ for all $i \in [\numramps]$.
\end{proposition} 
\begin{proof}
See Appendix \ref{Section: (Appx) Proof of distributed aFCQ-RM prop}.
\end{proof}
\par\textbf{V2I communication requirements}: The $\DisDRRRM$ policy requires the vector queue sizes $\Queuelength{}=[\Queuelength{i}(t)]$ (if $\Cyclelength > 1$) and the state of all the vehicles $\StateofDySys$. Its worst-case communication cost is calculated similar to the $\GenRM$ policy, except that at the end of each update period $\updateperiod$, $\StateofDySys$ is not communicated to all on-ramps. Instead, for all $i \in [\numramps]$, only the part $\StateofDySys$ associated with the vehicles between the $i$-th and $(i+1)$-th on-ramps is communicated to on-ramp $i$. The contribution of the queue size to the communication cost is the same as the $\GenRM$ policy because of the same cycle mechanism. Thus, the worst-case communication cost is reduced to $(\numcells + \numaccslots{a})/\updateperiod + \numaccslots{m} + c_{2}$ transmissions per $\timestep$ seconds.

\subsection{Dynamic Space Gap Policy}\label{Section: Impact of Autonomy}
In this policy, the on-ramps require an additional space gap, in addition to the safety distances in (M3)-(M4), before releasing a vehicle. This additional space gap is updated periodically based on the state of all vehicles. Recall that the $\GenRM$ policy enforces an additional time gap between release of successive vehicles, which is updated based on the current state of the vehicles as well as their state in the past. The dynamic space gap policy only requires the current state of the vehicles.
%%%%%%%%%%%%%%%%%%%%%%%%%%%%%%%%%%%%%%%%%%%
%%%%%%%%%%%%%%%%% arXiv %%%%%%%%%%%%%%%%%%%
%%%%%%%%%%%%%%%%%%%%%%%%%%%%%%%%%%%%%%%%%%%
However, it requires the following additional assumptions on the vehicle controller: consider $n$ vehicles over a time interval during which at least one vehicle is in the $\Speedmode$ mode, and no vehicle leaves the freeway. Then, during this time interval:
\begin{itemize}
\item[(VC4)] each vehicle changes mode at most once. 

\item[(VC5)] if no vehicle changes mode, then $\StateofDySys_g := \StateofDySys_{g_1} + \StateofDySys_{g_2}$ converges to zero globally exponentially, where $\StateofDySys_{g_1}$ and $\StateofDySys_{g_2}$ are defined as follows: let $\StateofDySys_{g_1}(t) := \sum_{e \in [n]}(w_1|v_e(t)-\speedlim|+w_2|a_e(t)|)J_e(t)$, where $w_1, w_2$ are normalization factors, and $J_e$ is a binary variable which is equal to zero if the ego vehicle has been in the $\Speedmode$ mode at all times since being released from an on-ramp, and one otherwise. Moreover, 
\begin{equation*}
    \StateofDySys_{g_2} := \sum_{e \in [n]}w_4\left(|y_e(t)-S_e(t)|\mathbbm{1}_{\{y_e(t) < S_e(t)\}} + \max_{s \in [t',t_f]}|\hat{y}_e(s|t)-\hat{S}_e(s|t)|\mathbbm{1}_{\{\hat{y}_e(s|t) < \hat{S}_e(s|t)\}}\right),
\end{equation*}
where $w_4$ is a normalization factor, and the second term in the sum is set to zero if the ego vehicle is not in a merging area. If the ego vehicle is in a merging area but has not yet crossed the merging point, then $t'=t_m$. Otherwise, $t'=t$. Furthermore, recall from Section~\ref{sec:vehicle-control} that the acceleration lane of an on-ramp ends either at the merging point (if the merging speed is $\speedlim$), or on the mainline (if the merging speed is less than $\speedlim$). The term $t_f|t$ is the time the ego vehicle predicts to cross the endpoint of the acceleration lane based on the information available at time $t \leq t_f$.  
\end{itemize}

%%%%%%%%%%%%%%%%%%%%%%%%%%%%%%%%%%%%%%%%%%%
%\begin{remark}
%Assumption (VC4) implies that each vehicle is assigned a projected leading vehicle at most once. 
%\end{remark}
\begin{definition}\label{Def: Relaxed Fixed-Cycle Quota RM policy}
\textbf{(Dynamic Space Gap ($\ConRM$) ramp metering policy)} The policy works in cycles of fixed length $\Cyclelength\timestep$, where $\Cyclelength \in \N$. At the start of the $k$-th cycle at $t_k = (k-1)\Cyclelength\timestep$, each on-ramp allocates itself a ``quota" equal to the number of vehicles at that on-ramp at $t_k$. At time $t \in [t_k, t_{k+1}]$ during the $k$-th cycle, on-ramp $i$ releases the ego vehicle if (M1)-(M2) and the following condition are satisfied:
\begin{itemize}
\item[(M6)] at time $t$, $y_e(t) \geq \Releasedist_{1}(\StateofDySys(t))$, $\hat{y}_e(t_m|t) \geq \Releasedist_{2}(\StateofDySys(t))$, and $\hat{y}_{\hat{f}}(t_m|t) \geq \Releasedist_{3}(\StateofDySys(t))$, where $\hat{y}_{\hat{f}}$ is the predicted distance between the ego vehicle and its virtual following vehicle, $\Releasedist_{1}(\cdot)$ is the safety distance $S_e$ plus an additional space gap $f_1(\StateofDySys(t))$, $\Releasedist_{2}(\cdot)$ and $\Releasedist_{3}(\cdot)$ are the distances in (M4) required to ensure safety between merging and exiting the acceleration lane, plus additional space gaps $f_2(\StateofDySys(t))$ and $f_3(\StateofDySys(t))$, respectively.
\end{itemize}
\par
Once an on-ramp reaches its quota, it pauses release during the rest of the cycle. 
%%%%%%%%%%%%%%%%%%%%%%%%%%%%%%%%%%%%%%%%%%%
%%%%%%%%%%%%%%%%% arXiv %%%%%%%%%%%%%%%%%%%
%%%%%%%%%%%%%%%%%%%%%%%%%%%%%%%%%%%%%%%%%%%
The additional space gaps $f_i(\StateofDySys(t))$, $i=1,2,3$, are piecewise constant and updated periodically at each time step according to some rule which will be determined in the proof of Theorem \ref{Prop: stability of rFCQ-RM}. 
%%%%%%%%%%%%%%%%%%%%%%%%%%%%%%%%%%%%%%%%%%%
%according to the rule $D(t) = K_{n(t)}\|\StateofDySys(t)\|$, where $n(t)$ is the number of vehicles on the mainline and acceleration lanes at time $t$, and $K_{n(t)}$ is a constant that depends on $n(t)$ and the vehicle dynamics. Since the number of vehicles on the mainline and acceleration lanes does not exceed $n_j := (\Perimeter + \Perimeter_a)/(\standstilldist + \vehiclelength)$, the values of $K_{1}, \ldots, K_{n_j - 1}$ completely determine the map from $\|\StateofDySys(\cdot)\|$ to $D(\cdot)$.
\end{definition}

%\kscomment{
%\begin{definition}\label{Def: Stability of platoon}
%\textbf{(Platoon Stability)} We say that a platoon of $k$ vehicles, $k \in [\numcells]$, is \textit{stable} if $\StateofDySys_k(t) \ra 0$ as $t \ra \infty$. We say that platoons of vehicles are stable if all platoons of size $k \in [\numcells]$ are stable. 
%\end{definition}
%}
%%%%%%%%%%%%%%%%%%%%%%%%%%%%%%%%%%%%%%%%%%%
%%%%%%%%%%%%%%%%% arXiv %%%%%%%%%%%%%%%%%%%
%%%%%%%%%%%%%%%%%%%%%%%%%%%%%%%%%%%%%%%%%%%
\begin{theorem}\label{Prop: stability of rFCQ-RM}
There exists $f_i$, $i=1,2,3$, such that for any initial condition, and $\Cyclelength \in \N$, the $\ConRM$ policy keeps the freeway under-saturated if $(\frac{\timestep_i}{\timestep}-1) \Avgload{i} < 1$ for all $i \in [\numramps]$.
\end{theorem} 
\begin{proof}
See Appendix~\ref{Section: (Appx) Proof of rFCQ-RM prop}.
%See Appendix \ref{Section: (Appx) Proof of rFCQ-RM prop} for the details. The key is to use the platoon stability assumption to show that the merging vehicles do not cause switching to the vehicle following mode in $[0, \Tempty]$. This implies that $\|\StateofDySys(\Tempty)\| = 0$. Then, Theorem \ref{Prop: Stability of FCQ-RM policy for all T} applies. 
\end{proof}
%%%%%%%%%%%%%%%%%%%%%%%%%%%%%%%%%%%%%%%%%%%
\par\textbf{V2I communication requirements}: The $\ConRM$ policy requires the vector queue sizes $\Queuelength{}=[\Queuelength{i}(t)]$ (if $\Cyclelength > 1$) and the state of all the vehicles $\StateofDySys$. At each time step, $\StateofDySys$ is communicated to every on-ramp. After a finite time, the number of vehicles that constitute $\StateofDySys$ is no more than $\numcells + \numaccslots{a}$. Moreover, the contribution of the queue size to the communication cost is the same as the $\GenRM$ policy because of the same cycle mechanism. Thus, the communication cost is upper bounded by $(\numcells + \numaccslots{a})\numramps + c_{2}$ transmissions per $\timestep$ seconds.
%%%%%%%%%%%%%%%%%%%%%%%%%%%%%%%%%%%%%%%%%%%
%%%%%%%%%%%%%%%%% arXiv %%%%%%%%%%%%%%%%%%%
%%%%%%%%%%%%%%%%%%%%%%%%%%%%%%%%%%%%%%%%%%%

%%%%%%%%%%%%%%%%%%%%%%%%%%%%%%%%%%%%%%%%%%%%%%%%%%%%%%%
%%%%%%%%%%%%%%%%%%%%% Remark %%%%%%%%%%%%%%%%%%%%%%%%%%
%%%%%%%%%%%%%%%%%%%%%%%%%%%%%%%%%%%%%%%%%%%%%%%%%%%%%%%
\begin{remark}
The choice of the additional space gaps $f_i$, $i=1,2,3$, in the $\ConRM$ policy is not limited to the expressions found in the proof of Theorem~\ref{Prop: stability of rFCQ-RM}. An alternative expression is simulated in Section~\ref{Subsection: (Sim) Average travel time}
\end{remark}
%%%%%%%%%%%%%%%%%%%%%%%%%%%%%%%%%%%%%%%%%%%%%%%%%%%%%%%
%%%%%%%%%%%%%%%%%%%%% Remark %%%%%%%%%%%%%%%%%%%%%%%%%%
%%%%%%%%%%%%%%%%%%%%%%%%%%%%%%%%%%%%%%%%%%%%%%%%%%%%%%%

%%%%%%%%%%%%%%%%%%%%%%%%%%%%%%%%%%%%%%%%%%%

%%%%%%%%%%%%%%%%%%%%%%%%%%%%%%%%%%%%%%%%%%%%%%%%%%%%%%%
%%%%%%%%%%%%%%%%%%%%% Remark %%%%%%%%%%%%%%%%%%%%%%%%%%
%%%%%%%%%%%%%%%%%%%%%%%%%%%%%%%%%%%%%%%%%%%%%%%%%%%%%%%
\mpcommentout{
\begin{remark}
\mpcomment{The gaps $\Releasedist_{i}(\StateofDySys(t))$, $i=1,2,3$, are determined from additional assumptions on the vehicle controller, such as the control logic in the $\Safetymode$ mode and the logic for switching back to the $\Speedmode$ mode.} Such details as well as the throughput of the $\ConRM$ policy are omitted for brevity. They can be found in the arXiv version of this paper (\cite{pooladsanj2022saturation}).
\end{remark}
}
%%%%%%%%%%%%%%%%%%%%%%%%%%%%%%%%%%%%%%%%%%%%%%%%%%%%%%%
%%%%%%%%%%%%%%%%%%%%% Remark %%%%%%%%%%%%%%%%%%%%%%%%%%
%%%%%%%%%%%%%%%%%%%%%%%%%%%%%%%%%%%%%%%%%%%%%%%%%%%%%%%

\subsection{Local and Greedy Policies}
\label{sec:greedy}
The actions of the three policies introduced in Sections~\ref{sec:release-rate}-\ref{Section: Impact of Autonomy} can be divided into two phases. The first phase concerns the transient from the initial condition to the free flow state, and the second phase is from the free flow state onward. Since throughput is an asymptotic notion, it is natural to examine the policies specifically in the second phase. Indeed, in the second phase, the actions of all the three policies can be shown to be equivalent to the following policy: 
\begin{definition}\label{Def: Fixed-Cycle Quota RM policy}
\textbf{(Fixed-Cycle Quota ($\FCQRM$) ramp metering policy)} The policy works in cycles of fixed length $\Cyclelength\timestep$, where $\Cyclelength \in \N$. At the start of the $k$-th cycle at $t_k = (k-1)\Cyclelength\timestep$, each on-ramp allocates itself a ``quota" equal to the number of vehicles at that on-ramp at $t_k$. During a cycle, the $i$-th on-ramp releases the ego vehicle if (M1)-(M4) are satisfied. Once an on-ramp reaches its quota, it pauses release during the rest of the cycle.  
\end{definition}
\par\textbf{V2I communication requirements}: The $\FCQRM$ policy requires the vector queue sizes $\Queuelength{}=[\Queuelength{i}(t)]$ (if $\Cyclelength > 1$) and part of the state of all the vehicles $\StateofDySys$. At each time step during a cycle, the vehicles in the merging area of each on-ramp communicate their state to that on-ramp. Moreover, the contribution of the queue size to the communication cost is the same as the $\GenRM$ policy because of the same cycle mechanism. Hence, $C$ is upper bounded by $\numaccslots{m} + c_{2}$ transmissions per $\timestep$ seconds. 
\par
Note that the $\FCQRM$ policy is local \mpcomment{since each on-ramp only requires the state of the vehicles in its vicinity}. In the special case of $\Cyclelength=1$, the $\FCQRM$ policy becomes the simple \emph{Greedy} policy under which the on-ramps do not need to keep track of their quota. Therefore, the communication cost of the Greedy policy is upper bounded by $\numaccslots{m}$. One can see from the proof of the $\GenRM$ policy that, starting from the aforementioned second phase, the freeway is under-saturated under the $\FCQRM$ policy if $(\timestep_i/\timestep-1) \Avgload{i} < 1$ for all $i \in [\numramps]$. It is natural to wonder if such a result holds if we start from \emph{any} initial condition, or under \emph{any} other greedy policy (not just the slot-based). In Section~\ref{Subsection: (Sim) Relaxing V2X requirements}, we provide intuition as to when the former statement might be true.

\subsection{An Outer-Estimate}
\label{sec:necessary}
We now provide an outer-estimate of the under-saturation region of \emph{any} RM policy, \mpcomment{including macroscopic-level RM and policies with prior knowledge of the demand}. This outer-estimate can be thought of as the network analogue of the flow capacity of an isolated on-ramp. We benchmark the inner-estimates of the proposed policies against this outer-estimate, and show that the proposed policies can maximize throughput.   
\par
Let $\Cumuldepartures{\pi,p}(k\timestep)$ be the cumulative number of vehicles that has crossed point $p$ on the mainline up to time $k\timestep$, $k \in \N_{0}$, under the RM policy $\pi$. Then, the crossing rate at point $p$ is defined as $\Cumuldepartures{\pi,p}(k\timestep)/k$ and the ``long-run" crossing rate is $\limsup_{k \ra \infty}\Cumuldepartures{\pi,p}(k\timestep)/k$. Note that $\Cumuldepartures{\pi,p}((k+1)\timestep)-\Cumuldepartures{\pi,p}(k\timestep)$ represents the \emph{traffic flow} in terms of the number of vehicles per $\timestep$ seconds at point $p$. Macroscopic traffic models suggest that the traffic flow is no more than the mainline flow capacity, which is $1$ vehicle per $\timestep$ seconds. Hence, $\Cumuldepartures{\pi,p}((k+1)\timestep)-\Cumuldepartures{\pi,p}(k\timestep) \leq 1$ for all $k \in \N_{0}$ and any RM policy $\pi$. This implies that 
\begin{equation}\label{eq:crossing-rate-inequality}
    \limsup_{k \ra \infty}\Cumuldepartures{\pi,p}(k\timestep)/k \leq 1,
\end{equation}
for any RM policy $\pi$. We take \eqref{eq:crossing-rate-inequality} as an assumption to the next Theorem.

\begin{theorem}\label{Prop: Necessary condition for stability}
%Consider the ring road network and let the average load given by $\Avgload{}$. 
If a policy $\pi$ keeps the freeway under-saturated and satisfies \eqref{eq:crossing-rate-inequality} for at least one point on each link, then the demand must satisfy $\Avgload{} \leq 1$.
\end{theorem}
\begin{proof}
See Appendix \ref{Section: (Appx) Proof of necessary condition for regularity}.
\end{proof}
%%%%%%%%%%%%%%%%%%%%%%%%%%%%%%%%%%%%%%%%%%%%%%%%%%%%%%%
%%%%%%%%%%%%%%%%%%%%% Remark %%%%%%%%%%%%%%%%%%%%%%%%%%
%%%%%%%%%%%%%%%%%%%%%%%%%%%%%%%%%%%%%%%%%%%%%%%%%%%%%%%
\begin{remark}
In all of the policies studied in previous sections, \mpcomment{vehicles reach the free flow state after a finite time.} Therefore, there exists some point $p_i$ on link $i$ for every $i \in [\numramps]$ where vehicles cross at the speed $\speedlim$. This and the constant time headway rule imply that the number of vehicles that can cross $p_i$ at each time step is no more than one. Therefore, $\limsup_{k \ra \infty}\Cumuldepartures{\pi,p_i}(k\timestep)/k \leq 1$ for all $i \in [\numramps]$, and the long-run crossing rate condition in \eqref{eq:crossing-rate-inequality} holds.
\end{remark}
%%%%%%%%%%%%%%%%%%%%%%%%%%%%%%%%%%%%%%%%%%%%%%%%%%%%%%%

%%%%%%%%%%%%%%%%%%%%%%%%%%%%%%%%%%%%%%%%%%%%%%%%%%%%%%%
%%%%%%%%%%%%%%%%%%%%% Remark %%%%%%%%%%%%%%%%%%%%%%%%%%
%%%%%%%%%%%%%%%%%%%%%%%%%%%%%%%%%%%%%%%%%%%%%%%%%%%%%%%
%\begin{remark}\label{remark:max-throughput-high-merge}
%\kscomment{
%\begin{itemize}
%\item 
\mpcomment{Recall that one of our contributions is showing that the proposed policies can maximize throughput. We end this section by discussing the conditions for achieving maximum throughput.} Recall the definition of $\tau_i$ for on-ramp $i$. If the merging speed at on-ramp $i$ is $\speedlim$, then $\tau_i=2\timestep$, i.e., $\tau_i$ is twice the safe time headway for vehicle following in the longitudinal direction. If this holds for every on-ramp, i.e., $\tau_i=2\tau$ for all $i \in [\numramps]$, then the inner-estimate of the under-saturation region given in Theorem~\ref{Prop: stability of QRM policy for low merging speed}, Theorem~\ref{Prop: stability of general aFCQ-RM}, and Proposition~\ref{Prop: stability of distributed aFCQ-RM} become $\Avgload{}<1$. Comparing with Theorem~\ref{Prop: Necessary condition for stability}, this implies that the Renewal, $\GenRM$, and $\DisDRRRM$ policies maximize throughput for all practical purposes. \mpcomment{In other words, when the merging speed at all the on-ramps equals the free flow speed, the proposed policies maximize the throughput for all practical purposes. This scenario occurs when all the on-ramps are sufficiently long.} The following numerical example illustrates what we mean by ``all practical purposes".
%%%%%%%%%%%%%%%%%%%%%%%%%%%%%%%%%%%%%%%%%%%%%%%%%%%%%%%
%\end{itemize}
%} 
%\end{remark}
%%%%%%%%%%%%%%%%%%%%%%%%%%%%%%%%%%%%%%%%%%%%%%%%%%%%%%%
%%%%%%%%%%%%%%%%%%%%% Figure %%%%%%%%%%%%%%%%%%%%%%%%%%
%%%%%%%%%%%%%%%%%%%%%%%%%%%%%%%%%%%%%%%%%%%%%%%%%%%%%%%
\begin{figure}[t!]
        \centering
        \includegraphics[width=0.4\textwidth]{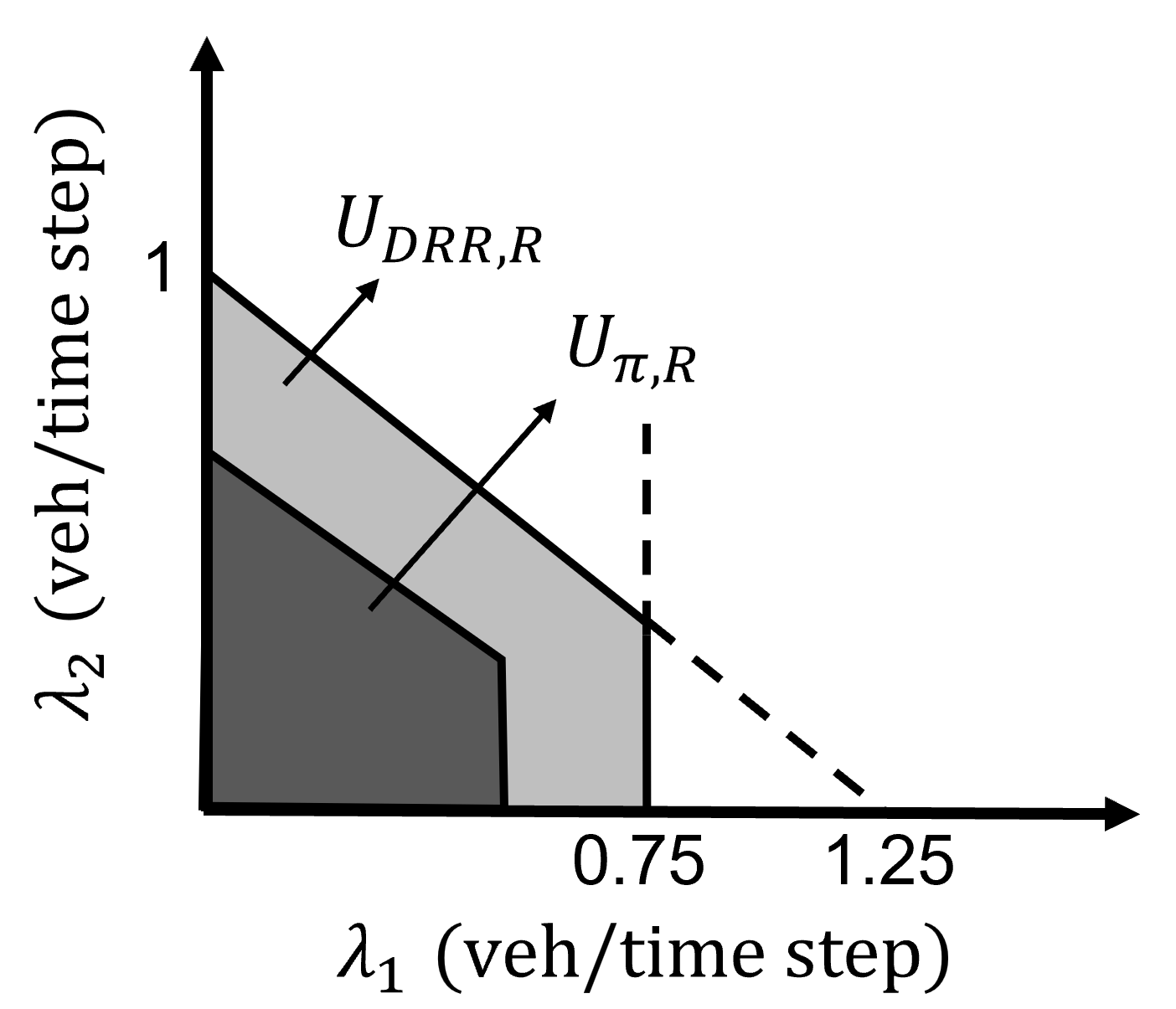}
        \vspace{0.2 cm}
    \caption{\centering\sf An illustration of the under-saturation region of some policy $\pi'$ (dark grey area) and the $\GenRM$ policy (dark + light grey areas) from Example~\ref{ex:throughput-numerical}.}\label{Fig: Numerical example for throughput}
\end{figure}
%%%%%%%%%%%%%%%%%%%%%%%%%%%%%%%%%%%%%%%%%%%%%%%%%%%%%%%
%%%%%%%%%%%%%%%%%%%%%%%%%%%%%%%%%%%%%%%%%%%%%%%%%%%%%%%
%%%%%%%%%%%%%%%%%%%%% Example %%%%%%%%%%%%%%%%%%%%%%%%%
%%%%%%%%%%%%%%%%%%%%%%%%%%%%%%%%%%%%%%%%%%%%%%%%%%%%%%%
\begin{example}\label{ex:throughput-numerical}
\textbf{(Example~\ref{ex:throughput} cont'd)} Consider again the $3$-ramp network with
\begin{equation*}
    \routingmatrix = \begin{pmatrix}
    0.2 & 0.7 & 0.1 \\
    0 & 0.8 & 0.2 \\
    0.5 & 0 & 0.5
    \end{pmatrix},~ \Arrivalrate{3}=0.5~[\text{veh}/\text{time step}],
\end{equation*}
and suppose that the merging speed at all the on-ramps is $\speedlim$, i.e, $\tau_i = 2\timestep$ for $i=1,2,3$. Let us consider the $\GenRM$ policy presented in Section~\ref{sec:release-rate}. By Theorem~\ref{Prop: stability of general aFCQ-RM}, the under-saturation region of this policy is given by
\begin{equation*}
    U_{\GenRM,\routingmatrix} = \{(\Arrivalrate{1},\Arrivalrate{2}): \Avgload{} = \max_{i=1,2,3}\Avgload{i} < 1\} = \{(\Arrivalrate{1},\Arrivalrate{2}): \Arrivalrate{1} < 0.75,~0.8\Arrivalrate{1}+\Arrivalrate{2}<1\},
\end{equation*}
which is illustrated in Figure~\ref{Fig: Numerical example for throughput}. Furthermore, by Theorem~\ref{Prop: Necessary condition for stability}, for any other policy $\pi'$ we have $U_{\pi', \routingmatrix} \subseteq \{(\Arrivalrate{1},\Arrivalrate{2}): \Avgload{} \leq 1\}$, which is a subset of $U_{\GenRM,\routingmatrix}$ except, maybe, at the boundary of $U_{\GenRM,\routingmatrix}$. Since the boundary of $U_{\GenRM,\routingmatrix}$ has zero volume, the vector $(\Arrivalrate{1},\Arrivalrate{2})$ lies either inside or outside of $U_{\GenRM,\routingmatrix}$ in practice. Therefore, $U_{\pi', \routingmatrix} \subseteq U_{\GenRM,\routingmatrix}$ for all practical purposes. Note that the previous conclusion also holds for any other choice of the routing matrix, which implies that the $\GenRM$ policy maximizes the throughput for all practical purposes.
\end{example}
%%%%%%%%%%%%%%%%%%%%%%%%%%%%%%%%%%%%%%%%%%%%%%%%%%%%%%%

\subsection{Discussion of the Straight Road Configuration}\label{Section: Straight Line}
In this section, we extend our results to the straight road configuration such as the one shown in Figure~\ref{Fig: line network}. Let the number of on- and off-ramps be $\numramps-1$; suppose that they are placed alternatively, and they are numbered in an increasing order along the direction of travel. Recall that the upstream entry point in this configuration acts as a virtual on-ramp, indicating that all the previous on-ramps are metered.
\par
Roughly speaking, the straight road configuration is a special case of the ring road configuration with an upper triangular routing matrix. Therefore, it is natural to expect that all of our assumptions and results will apply to this setting as well. In fact, all of our assumptions remain unchanged, except for (VC3) in Section~\ref{sec:vehicle-control}, which is slightly modified as follows:  
\begin{itemize}
    \item [(VC3)] there exists $\Tempty$ such that for any initial condition, if no vehicle is released from the entry points $1, \ldots, j$ for some $j \in \{1,\ldots,\numramps\}$, and all the vehicles downstream of link $j$ are at the free flow state, then all the vehicles reach the free flow state after at most $\Tempty$ seconds. %Moreover, the vehicles from the upstream entry point enter at the constant speed $\speedlim$ after reaching the free flow state.
\end{itemize}
\par
\mpcommentout{
%%%%%%%%%%%%%%%%%%%%%%%%%%%%%%%%%%%%%%%%%%%%%%%%%%%%%%%
%%%%%%%%%%%%%%%%%%%%% Figure %%%%%%%%%%%%%%%%%%%%%%%%%%
%%%%%%%%%%%%%%%%%%%%%%%%%%%%%%%%%%%%%%%%%%%%%%%%%%%%%%%
\begin{figure}[t!]
        \centering
        \includegraphics[width=0.6\textwidth]{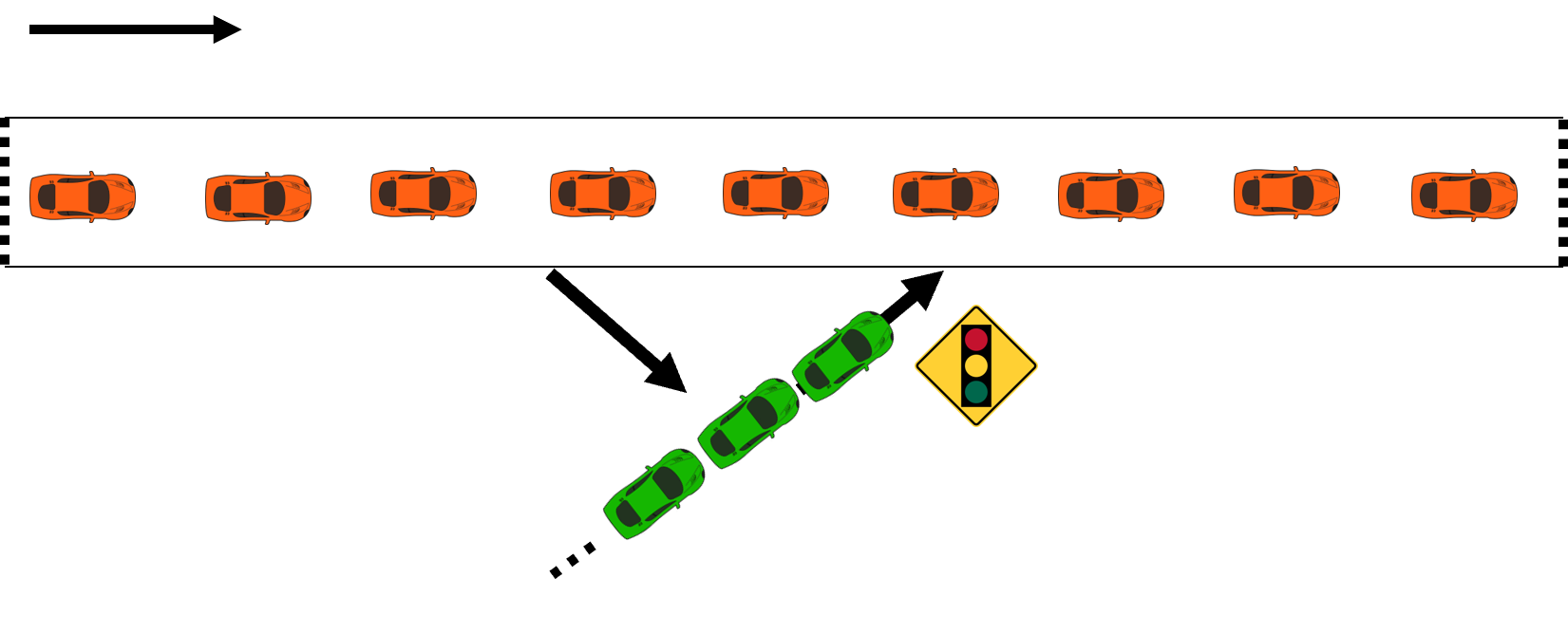}
        \vspace{0.2 cm}
    \caption{\centering\sf A scenario in the straight road configuration, where a queue appears at the on-ramp under the proposed policies for sufficiently high upstream inflow.}\label{fig:straight-road-limitation}
\end{figure}
%%%%%%%%%%%%%%%%%%%%%%%%%%%%%%%%%%%%%%%%%%%%%%%%%%%%%%%
}
%Note that condition (VC3) imposes macroscopic-level constraints on the upstream inflow. \mpcommentout{This is justified if, e.g., the on-ramps upstream of the entry point are also controlled.} We also need to specify vehicle arrivals and their destination at the microscopic level after reaching the free flow state. To this end, we may adopt an i.i.d Bernoulli process with parameter $\Arrivalrate{1}$ and an i.i.d Bernoulli routing process that are independent of the on-ramp processes discussed in Section~\ref{sec:demand}.
We may also assume, without loss of generality, that vehicles enter from the upstream entry point at the speed $\speedlim$ in the free flow state. This is justified since the upstream inflow is not restricted by the inflow from the on-ramps in the free flow state. With these slight modifications, the description of all the proposed policies remain unchanged for the straight road configuration. Their performance will also be the same with a slight change in the proof of Theorem~\ref{Prop: stability of general aFCQ-RM}, which is stated as needed during the proof.
%Note that the routing matrix for the straight road geometry is upper triangular. 
\par

%, with the exception of the $\DisDRRRM$ policy which is slightly modified as follows: $\StateofDySys_{f}^2$ in the $\DisDRRRM$ is expanded to include all vehicles between the upstream entry point and on-ramp $2$. 
\begin{remark}
The results can be extended even further when we have partial control over the on-ramps, e.g., when some of the on-ramps preceding the upstream entry point in the straight road configuration are not metered. However, in that case, the throughput of the metered on-ramps may become zero for sufficiently high upstream inflow under the proposed policies. This is due to the safety considerations that prevent entry from the on-ramps, resulting in a queue forming at those on-ramps. In order to deal with this issue, an alternative approach is for the vehicles to coordinate locally, allowing the on-ramp vehicles to enter the freeway. This would shift the on-ramp queue to the freeway, which is more desirable in practice. It is important to note that vehicle coordination does not offer any additional benefits when we have full control over every on-ramp. In such cases, the need for coordination is already addressed by the proposed RM policies in place. The authors are currently exploring vehicle coordination for the partial control case in a separate paper.
\end{remark}

\section{Simulations}\label{Section: Simulation}
The following setup is common to all the simulations in this section. We consider a ring road of length $\Perimeter = 1860~[\text{m}]$ and $\numramps=3$ on- and off-ramps. Let $\timeheadway = 1.5 ~[\text{sec}]$, $\standstilldist = 4~ [\text{m}]$, $\vehiclelength = 4.5 ~[\text{m}]$, and $\speedlim = 15 ~[\text{m/sec}]$. For these parameters, we obtain $n_c=60$. The on-ramps are located symmetrically at $0$, $\Perimeter/3$, and $2\Perimeter/3$; the off-ramps are also located symmetrically $155~[\text{m}]$ upstream of each on-ramp. The initial queue size of all the on-ramps is set to zero. Vehicles arrive at the on-ramps according to i.i.d Bernoulli processes with the same rate $\Arrivalrate{}$; their destinations are determined by 
\begin{equation*}
    \routingmatrix = \begin{pmatrix}
    0.2 & 0.7 & 0.1 \\
    0 & 0.8 & 0.2 \\
    0.5 & 0 & 0.5
    \end{pmatrix}. 
\end{equation*}
\par
Note that, on average, most of the vehicles want to exit from off-ramp $2$. Thus, one should expect that on-ramp $3$ finds more safe merging gaps between vehicles on the mainline than the other two on-ramps. As a result, on-ramp $3$'s queue size is expected to be less than the other two on-ramps. All the simulations were performed in MATLAB. The details of the vehicle controller is provided in Section~\ref{Subsection: (Sim) Relaxing V2X requirements}.

\subsection{Greedy Policy for Low Merging Speed}\label{Subsection: (Sim) Greedy policy}
%%%%%%%%%%%%%%%%%%%%%%%%%%%%%%%%%%%%%%%%%%%%%%%%%%%%%%%
%%%%%%%%%%%%%%%%%%%%% Figure %%%%%%%%%%%%%%%%%%%%%%%%%%
%%%%%%%%%%%%%%%%%%%%%%%%%%%%%%%%%%%%%%%%%%%%%%%%%%%%%%%
\begin{figure}[htb!]
%\begin{center}
%    \begin{subfigure}{.4\textwidth}
        \centering
        \includegraphics[width=0.5\textwidth]{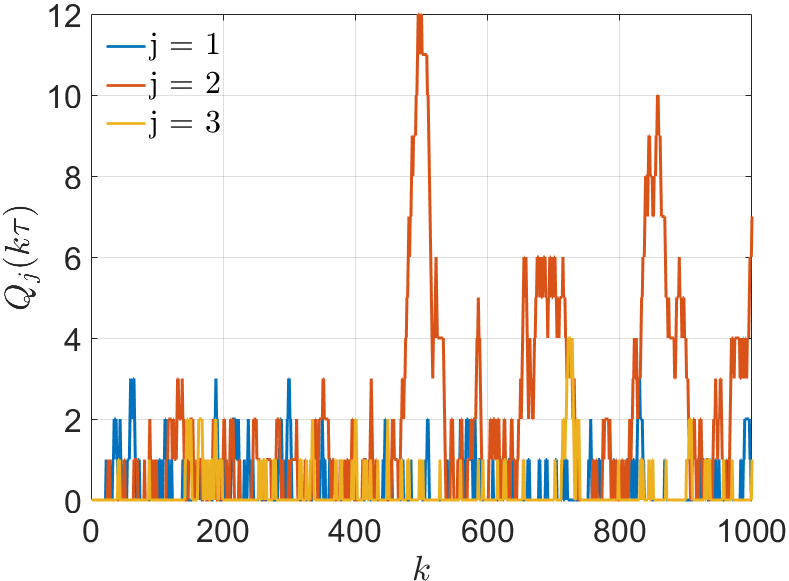}
        \vspace{0.2 cm}
        \caption{\centering\sf On-ramp queue size profiles under the $\FCQRM$ policy, when all the on-ramps are long.} \label{Fig: Qlen long ramps}
%    \end{subfigure}
%    \begin{subfigure}{.4\textwidth}
%        \centering
%        \includegraphics[width=\textwidth]{Figures/Qlen_Short.png}
%        \caption{} \label{Fig: Qlen short ramp}
%    \end{subfigure}
%    \end{center}
    
    %\caption{\sf On-ramp queue length dynamics when both on-ramps are long, under the $\FCQRM$ policy.}\label{Fig: Qlen profiles}(b) on-ramp $2$ is short, under the $\FCQRM$ policy 
\end{figure}
%%%%%%%%%%%%%%%%%%%%%%%%%%%%%%%%%%%%%%%%%%%%%%%%%%%%%%%
%%%%%%%%%%%%%%%%%%%%% Figure %%%%%%%%%%%%%%%%%%%%%%%%%%
%%%%%%%%%%%%%%%%%%%%%%%%%%%%%%%%%%%%%%%%%%%%%%%%%%%%%%%
The mainline and acceleration lanes are assumed to be initially empty in this section and Section~\ref{Subsection: (Sim) Avg Q lengths}. The on-ramps use the Greedy policy, i.e., the $\FCQRM$ policy with $\Cyclelength=1$, from Section~\ref{sec:greedy}. When the merging speed at all the on-ramps is $\speedlim$, then \mpcomment{the throughput of the $\FCQRM$ policy} is $\Arrivalrate{}= 5/9~[\text{veh}/\text{time step}]$ for the given $\routingmatrix$. Note that the throughput is a single point, \mpcomment{rather than a set of points}, since the arrival rates are assumed to be the same. The queue size profiles for $\Arrivalrate{}= 0.5~[\text{veh}/\text{time step}]$, which corresponds to $\Avgload{} = 0.9$ (heavy demand), is shown in Figure \ref{Fig: Qlen long ramps}. As expected, $\Queuelength{3}(\cdot)$ is generally less than the other two on-ramps.
\par
We next consider the case when the merging speed at on-ramps $1$ and $3$ are $\speedlim$, i.e., $\tau_1 = \tau_3 =2\timestep$, and is approximately $\speedlim/3= 5~[\text{m/sec}]$ at on-ramp $2$ which corresponds to $\tau_2=3\timestep$. In the heavy demand regime, i.e., $\Avgload{} = 0.9$, $\Queuelength{2}(\cdot)$ increases steadily which suggests that the freeway becomes saturated even though $\Avgload{}<1$. The \mpcomment{throughput of the $\FCQRM$ policy} in this case is estimated from simulations to be $\Arrivalrate{} = 0.44~[\text{veh}/\text{time step}]$, while the estimate given by Theorem~\ref{Prop: stability of general aFCQ-RM} is $\Arrivalrate{} = 0.33~[\text{veh}/\text{time step}]$. Moreover, the throughput estimate of the Renewal policy given by Theorem~\ref{Prop: stability of QRM policy for low merging speed} is $\Arrivalrate{} = 0.5~[\text{veh}/\text{time step}]$. Combining this with the simulation results, we see that the Renewal policy has a better throughput than the Greedy policy.
\par
%%%%%%%%%%%%%%%%% Commented (06/05/22) %%%%%%%%%%%%%%%%%%%
\mpcommentout{
Now consider a ``congested" initial condition on the mainline with $\numvehicles=30 > \numcells$ vehicles, each with initial speed $(\Perimeter/\numvehicles - \standstilldist - \vehiclelength)/\timeheadway = 8.1~[m/s] < \speedlim$ and with initial inter-vehicle separations all equal to $(\Perimeter/\numvehicles - \vehiclelength)$. Let the merging speed at both the on-ramps be $\speedlim$, and let $\Arrivalrate{1} = \Arrivalrate{2} = 0.5$. Figure~\ref{fig:greedy-congested-initial-condition} shows queue size profiles for this setting. The boundedness of queue size for $\Avgload{}$ close to 1 suggests that the greedy policy may keep the freeway under-saturated for every initial condition, at least when the merging speed at all the on-ramps is $\speedlim$.
}
%%%%%%%%%%%%%%%%% Commented (06/05/22) %%%%%%%%%%%%%%%%%%%
 
\subsection{Effect of Cycle Length on Queue Size}\label{Subsection: (Sim) Avg Q lengths}
%%%%%%%%%%%%%%%%%%%%%%%%%%%%%%%%%%%%%%%%%%%%%%%%%%%%%%%
%%%%%%%%%%%%%%%%%%%%% Figure %%%%%%%%%%%%%%%%%%%%%%%%%%
%%%%%%%%%%%%%%%%%%%%%%%%%%%%%%%%%%%%%%%%%%%%%%%%%%%%%%%
\begin{figure}[t!]
\begin{center}
    \begin{subfigure}{.436\textwidth}
        \centering
        \includegraphics[width=\textwidth]{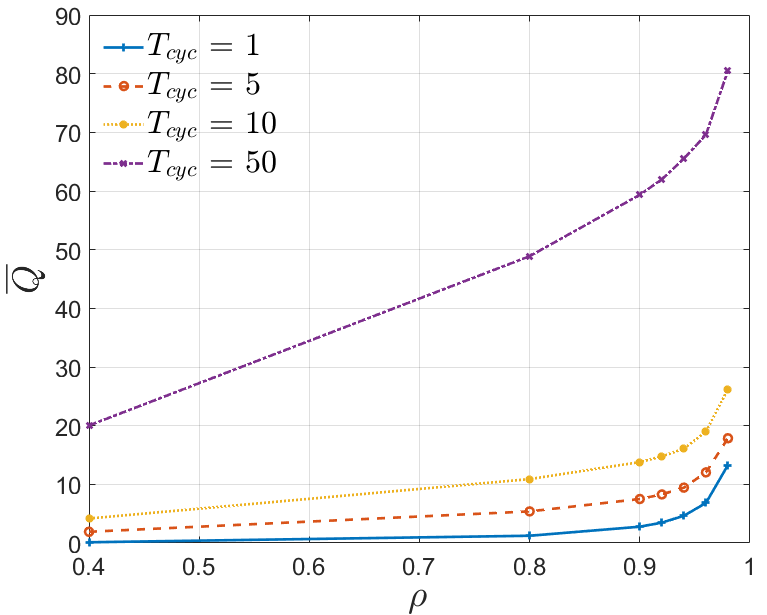}
        \caption{} \label{Fig: Avg qlen for long ramps}
    \end{subfigure}
    \begin{subfigure}{.45\textwidth}
        \centering
        \includegraphics[width=\textwidth]{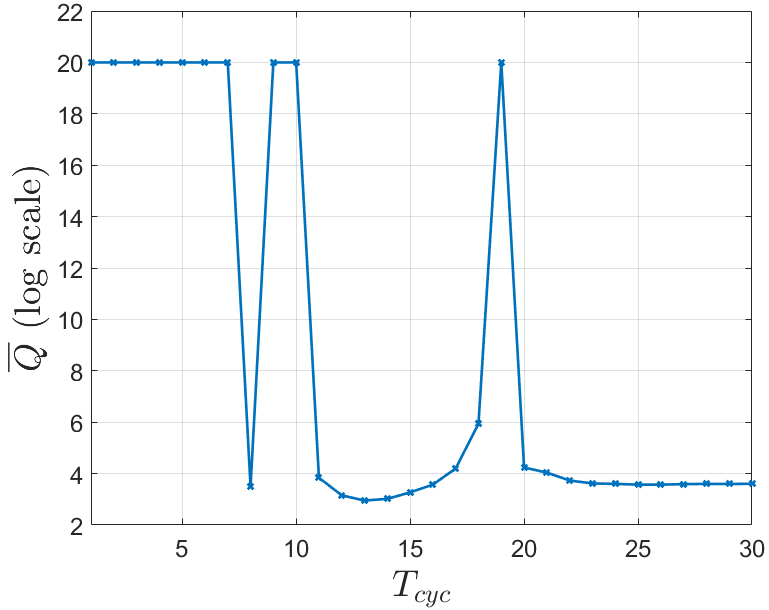}
        \caption{} \label{Fig: Avg qlen for short ramps}
    \end{subfigure}
    \end{center}
    \vspace{0.2 cm}
    \caption{\centering\sf Effect of cycle length $\Cyclelength$ on the long-run expected queue size (a) for different $\Avgload{}$ when both on-ramps are long, (b) for a fixed $\Avgload{}$ when on-ramp $2$ is short. Both plots are under the $\FCQRM$ policy. In plot (b), the logarithm of the expected queue size is set to $20$ whenever the freeway becomes saturated.}
\end{figure}
%%%%%%%%%%%%%%%%%%%%%%%%%%%%%%%%%%%%%%%%%%%%%%%%%%%%%%%
%%%%%%%%%%%%%%%%%%%%% Figure %%%%%%%%%%%%%%%%%%%%%%%%%%
%%%%%%%%%%%%%%%%%%%%%%%%%%%%%%%%%%%%%%%%%%%%%%%%%%%%%%%
The long-run expected queue size, i.e., $\overline{Q} := \limsup_{k \ra \infty} \E{\sum_{i}\Queuelength{i}(k\timestep)}$, are compared under the $\FCQRM$ policy for different $\Cyclelength$. The expected queue size is computed using the \textit{batch means} approach, with warm-up period of $10^5$, i.e., the first $10^5$ observations are not used, and batch size of $10^5$. In each case, the simulations are run until the margin of error of the $95\%$ confidence intervals are $1\%$.

Figure~\ref{Fig: Avg qlen for long ramps} shows $\overline{Q}$ vs. $\Avgload{}$ for different $\Cyclelength$ when the merging speed at all the on-ramps is $\speedlim$. The plot suggests that for all $\Avgload{}$, $\overline{Q}$ increases monotonically with $\Cyclelength$. However, this does not hold true when the merging speeds are low. For example, Figure~\ref{Fig: Avg qlen for short ramps} shows $\overline{Q}$ in log scale vs. $\Cyclelength$ for $\Arrivalrate{}=0.455~[\text{veh}/\text{time step}]$, which corresponds to $\Avgload{}\approx 0.82$, when the merging speed at on-ramps $1$ and $3$ is $\speedlim$, and is $\speedlim/3$ at on-ramp $2$. As the plot shows, depending on the choice of $\Cyclelength$, the freeway may become saturated which causes the expected queue size to grow over time, i.e., $\overline{Q}=\infty$ (in the plot, $\log\overline{Q}$ is set to $20$ whenever this occurs). Moreover, one can see from Figure~\ref{Fig: Avg qlen for short ramps} that the dependence of $\overline{Q}$ on $\Cyclelength$ is not monotonic in the low merging speed case. In fact, the optimal choice of $\Cyclelength$ in this case is $13$.

\subsection{Relaxing the V2X Requirements}\label{Subsection: (Sim) Relaxing V2X requirements}
In this scenario, we evaluate the performance of the Greedy policy for $\Avgload{}=0.9$ (heavy demand) when the merging speed at all the on-ramps is $\speedlim$, and the mainline is initially not empty. The $\Safetymode$ mode only consists of the following submode: in a non-merging scenario, the ego vehicle follows the leading vehicle by keeping a safe time headway as in, e.g., \citet{pooladsanj2021vehicle}. For this control law, platoons of vehicles are \emph{stable} and \emph{string stable} (\citet{pooladsanj2021vehicle}). In a merging scenario, it applies the above control law with respect to both its leading and virtual leading vehicles. The most restrictive acceleration is used by the ego vehicle. 
\mpcommentout{
once the ego vehicle is released from an on-ramp, it uses the aforementioned control law for $\speedlim/\maxaccel$ seconds, after which it reaches the speed $\speedlim$ if there are no vehicles ahead. After $\speedlim/\maxaccel$ seconds, each vehicle involved in the merging scenario simultaneously applies the same longitudinal control law with respect to its leading and virtual leading vehicles.} 
\par
The initial number of vehicles on the mainline, their location, and their speed are chosen at random such that the safety distance is not violated. The initial acceleration of all the vehicles is set to zero. We conducted $10$ rounds of simulation with different random seed for each round. It is observed in all scenarios that vehicles reach the free flow state after a finite time using the Greedy policy. We conjecture that this observation holds for \emph{any} initial condition if platoons of vehicles are stable and string stable. The intuition behind this conjecture is as follows: under the Greedy policy, vehicles are released only if there is sufficient gap between vehicles on the mainline, and successive releases are at least $\timestep$ seconds apart. In between releases, vehicles on the mainline try to reach the free flow state because of platoon stability. Once a vehicle is released, it may create disturbance, e.g., in the acceleration of the upstream vehicles. However, string stability ensures that such disturbance is attenuated. Therefore, it is natural to expect that the vehicles reach the free flow state after a finite time. Thereafter, by Theorem~\ref{Prop: stability of general aFCQ-RM}, the Greedy policy keeps the freeway under-saturated if $(\timestep_i/\timestep-1) \Avgload{i} < 1$ for all $i \in [\numramps]$.
%\par 
%Figure shows the queue length profiles under the greedy policy. Note that even though the initial condition does not satisfy the requirements of Theorem \ref{Prop: Stability of FCQ-RM policy for all T}, the greedy policy keeps the network under-saturated. This suggests that, in practice, the V2X requirements of the aFCQ-RM policy can be relaxed considerably without diminishing the performance. We conjecture that this conclusion holds \textit{for any} initial condition and network size if and only if platoons of vehicles are stable and string stable. 
\subsection{Comparing the Total Travel Time}\label{Subsection: (Sim) Average travel time}
We evaluate the total travel time under the Renewal, $\GenRM$, $\DisDRRRM$, $\ConRM$, and Greedy policies, \mpcomment{which use accurate traffic measurements obtained by V2I communication.} We compare these results against the ALINEA RM policy (\citet{papageorgiou1991alinea}), \mpcomment{which relies on local traffic measurements obtained by roadside sensors. The purpose of this comparison is to illustrate how much V2I communication can improve performance when comparing with a well-known RM policy which relies only on roadside sensors.} The ALINEA policy controls the outflow of an on-ramp according to the following equation:
\begin{equation*}
    r(k) = r(k-1) + K_r(\hat{o}-o(k)).
\end{equation*}
\par
Here, $r(\cdot)$ is the on-ramp outflow, $K_r$ is a positive design constant, $o(\cdot)$ is the occupancy of the mainline downstream of the on-ramp, and $\hat{o}$ is the desired occupancy. For the ALINEA policy, we use time steps of size $60~[\text{sec}]$, $K_r = 70~[\text{veh/h}]$ (\citet{papageorgiou1991alinea}), and $\hat{o} = 13~\%$ corresponding to the mainline flow capacity. We also add the following vehicle following safety filter on top of ALINEA: the ego vehicle is released only if it is predicted to be at a safe distance with respect to its virtual leading and following vehicles at the moment of merging. We use the name Safe-ALINEA to refer to this policy. For the $\GenRM$ and $\DisDRRRM$ policies, we use the following parameters: $\Cyclelength = 13$, $\updateperiod=2\timestep$, $\DesDecreaseconst = 50$, $\ReleasetimeDecConst{2} = 10$, $\ReleasetimeInc{i}^{\circ} = 0.1$, $\ReleasetimeAdj{} = 1.01$, and $\Tmax = 100~[\text{sec}]$. Informally, $\DesDecreaseconst$ is set to a high value so that the release times $\Releasetime{i}(\cdot)$ increase if the traffic condition, i.e., the value of $\StateofDySys_f$, does not improve significantly in between the update periods; $\ReleasetimeDecConst{2}$ and $\ReleasetimeInc{i}^{\circ}$ are set to low values so that the jump sizes in $\Releasetime{i}(\cdot)$ are small; $\Tmax$ is set to a high value so that on-ramps update their release time in a completely decentralized way. For the $\ConRM$ policy, we set $\Cyclelength = 13$ and use the additional space gap $\kappa_1 \StateofDySys_{f_1} + \kappa_2\sum_{e \in [n]}|y_e - (\timeheadway v_e + \standstilldist)|I_e$ on top of (M3)-(M4), where $\kappa_1 = \kappa_2 = 0.01$ are chosen so that the additional space gap is relatively small ($\approx 8~[\text{m}]$ for the initial condition described next). 
\par
We let the merging speed at on-ramps $1$ and $3$ be $\speedlim$, and $\speedlim/3$ at on-ramp $2$. We let $\Arrivalrate{} = 0.455~[\text{veh}/\text{time step}]$, and consider a congested initial condition where the initial number of vehicles on the mainline is $100 > \numcells$; each vehicle is at the constant speed of $6.7~[\text{m/sec}]$, and is at the distance $\timeheadway \times 6.7 + \standstilldist \approx 14.1~[\text{m}]$ from its leading vehicle. 
%The travel delay for an individual vehicle is defined as the difference between the actual travel time, i.e., the time the vehicle spends in the network upon arrival, and the ideal travel time, i.e., the travel time assuming that the waiting time at the on-ramp is zero and the vehicle is in the $\Speedmode$ mode until it reaches its destination. Let $d_1, d_2, \ldots, d_n$ be the travel delay of the first $n$ vehicles that complete their travels. Their Cumulative Average Delay ($\text{CAD}_n$) is computed as follows:
%\begin{equation*}
%    \text{CAD}_n = \frac{1}{n}\sum_{i = 1}^{n}d_i.
%\end{equation*}
\par
%%%%%%%%%%%%%%%%%%%%%%%%%%%%%%%%%%%%%%%%%%%%%%%%%%%%%%%
%%%%%%%%%%%%%%%%%%%%% Figure %%%%%%%%%%%%%%%%%%%%%%%%%%
%%%%%%%%%%%%%%%%%%%%%%%%%%%%%%%%%%%%%%%%%%%%%%%%%%%%%%%
\begin{figure}[t!]
        \centering
        \includegraphics[width=0.9\textwidth]{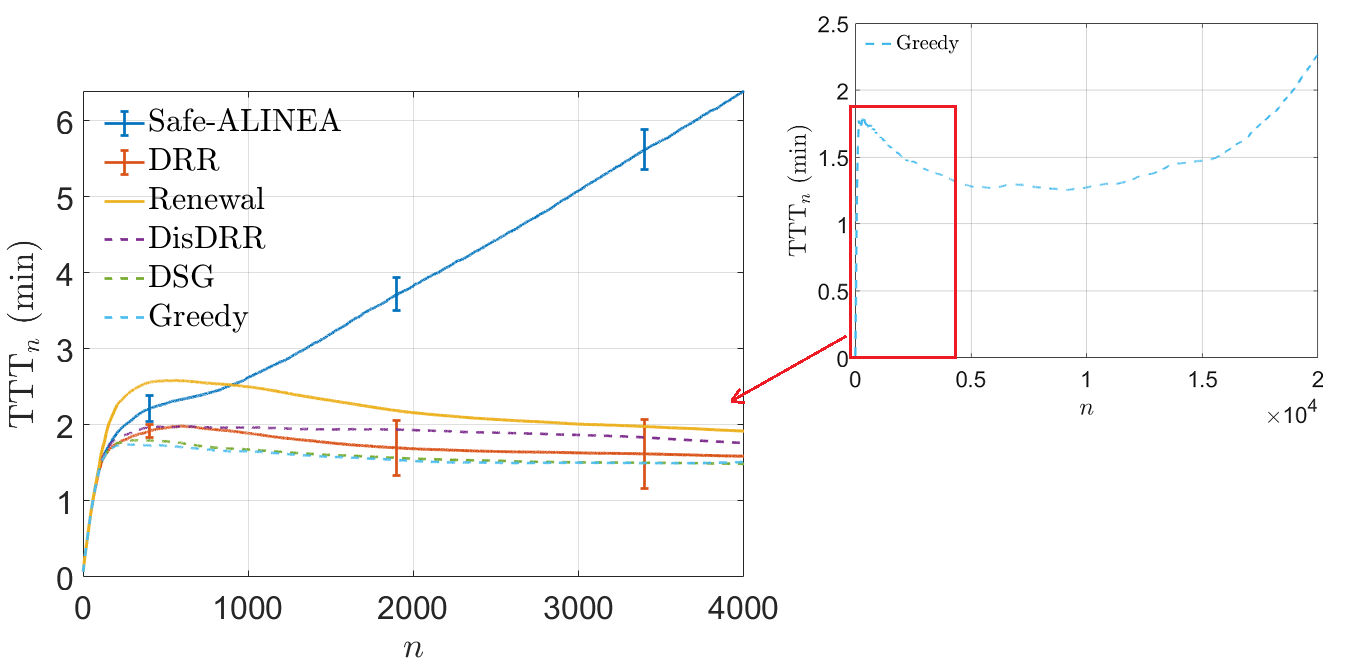}
        \vspace{0.2 cm}
    \caption{\centering\sf The average value of the total travel time of the first $n$ vehicles that complete their trips ($\text{TTT}_n$) for the fixed demand $\Avgload{} = 0.82$. The vertical lines in the Safe-ALINEA and $\GenRM$ policies represent the standard deviation. The standard deviation of the Renewal, $\DisDRRRM$, $\ConRM$, and Greedy policies were similar to the $\GenRM$ policy and are omitted for clarity. The smaller figure on the top right shows $\text{TTT}_n$ under the Greedy policy over a longer time horizon, with trajectories of the $\GenRM$, Renewal, $\DisDRRRM$, and $\ConRM$ policies removed due to their non-increasing behavior.}\label{Fig: Comaring CAD}
\end{figure}
%%%%%%%%%%%%%%%%%%%%%%%%%%%%%%%%%%%%%%%%%%%%%%%%%%%%%%%
%%%%%%%%%%%%%%%%%%%%% Figure %%%%%%%%%%%%%%%%%%%%%%%%%%
%%%%%%%%%%%%%%%%%%%%%%%%%%%%%%%%%%%%%%%%%%%%%%%%%%%%%%%

%%%%%%%%%%%%%%%%%%%%%%%%%%%%%%%%%%%%%%%%%%%%%%%%%%%%%%%
%%%%%%%%%%%%%%%%%%%%% Figure %%%%%%%%%%%%%%%%%%%%%%%%%%
%%%%%%%%%%%%%%%%%%%%%%%%%%%%%%%%%%%%%%%%%%%%%%%%%%%%%%%
\begin{figure}[t!]
        \centering
        \includegraphics[width=0.5\textwidth]{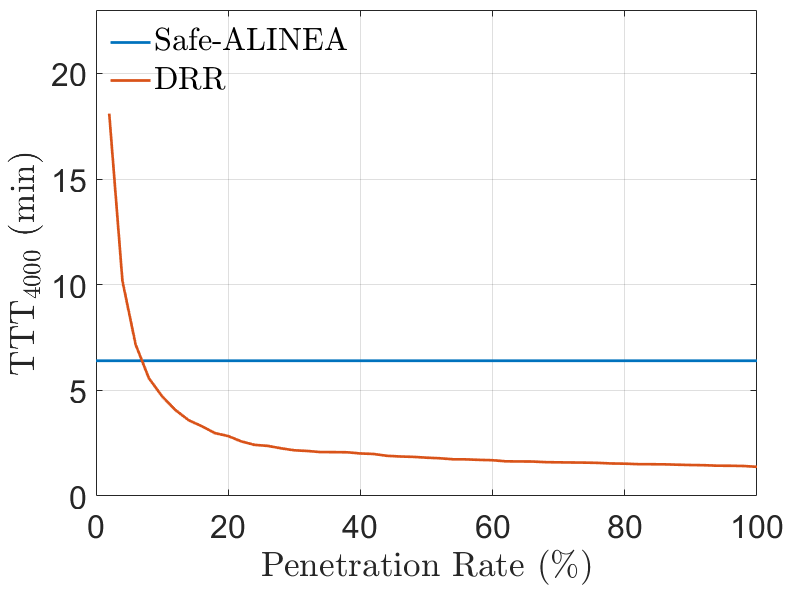}
        \vspace{0.2 cm}
    \caption{\centering\sf \mpcomment{The average travel time, i.e., $\text{TTT}_{4000}$, under the Safe-ALINEA and $\GenRM$ policies at different penetration rates of connected vehicles for the fixed demand $\Avgload{} = 0.82$.}}\label{fig:max-travel-time}
\end{figure}
%%%%%%%%%%%%%%%%%%%%%%%%%%%%%%%%%%%%%%%%%%%%%%%%%%%%%%%
%%%%%%%%%%%%%%%%%%%%% Figure %%%%%%%%%%%%%%%%%%%%%%%%%%
%%%%%%%%%%%%%%%%%%%%%%%%%%%%%%%%%%%%%%%%%%%%%%%%%%%%%%%

%%%%%%%%%%%%%%%%%%%%%%%%%%%%%%%%%%%%%%%%%%%%%%%%%%%%%%%
%%%%%%%%%%%%%%%%%%%%% Figure %%%%%%%%%%%%%%%%%%%%%%%%%%
%%%%%%%%%%%%%%%%%%%%%%%%%%%%%%%%%%%%%%%%%%%%%%%%%%%%%%%
\begin{figure}[t!]
        \centering
        \includegraphics[width=0.5\textwidth]{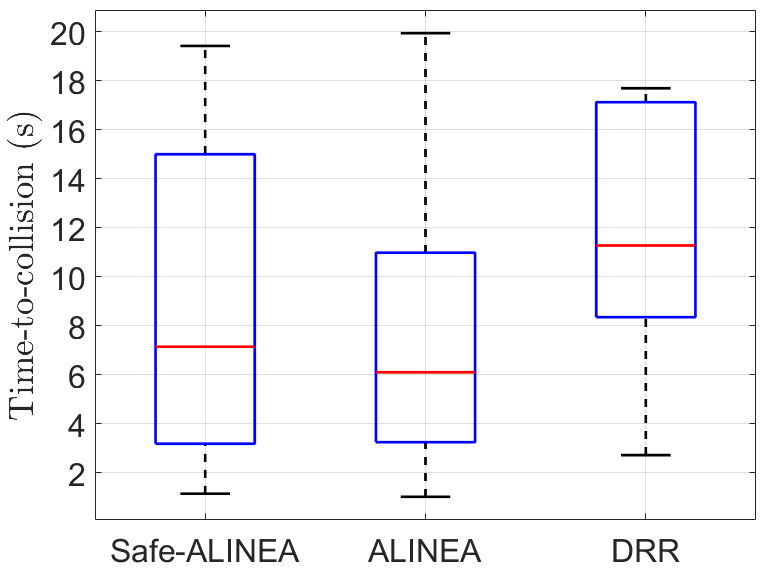}
        \vspace{0.2 cm}
    \caption{\centering\sf Time-to-collision near on-ramp $2$ under the Safe-ALINEA, ALINEA, and $\GenRM$ policies. On each box, the red line indicates the median, and the bottom and top edges of the box indicate the $25$-th and $75$-th percentiles, respectively. The whiskers extend to the most extreme data points not considered outliers.}\label{Fig: Comaring TTC}
\end{figure}
%%%%%%%%%%%%%%%%%%%%%%%%%%%%%%%%%%%%%%%%%%%%%%%%%%%%%%%
%%%%%%%%%%%%%%%%%%%%% Figure %%%%%%%%%%%%%%%%%%%%%%%%%%
%%%%%%%%%%%%%%%%%%%%%%%%%%%%%%%%%%%%%%%%%%%%%%%%%%%%%%%
We compute the average value of Total Travel Time (TTT$_n$), which is computed by averaging the total travel time of the first $n$ vehicles that complete their trips. The total travel time of a vehicle equals its waiting time in an on-ramp queue plus the time it spends on the freeway to reach its destination. In order to show consistency, we have conducted $10$ rounds of simulations with different random seed for each round and averaged the results over the $10$ simulation rounds. \mpcommentout{In each round, we used the same seed across different policies.}
\par
For the setup described above, the Safe-ALINEA and Greedy policies fail to keep the freeway under-saturated, whereas the other policies keep it under-saturated. From Figure~\ref{Fig: Comaring CAD} and Table~\ref{table: Performance}, we can see that the $\GenRM$ policy provides significant improvement in the average travel time, i.e., TTT$_{4000}$, compared to the Safe-ALINEA policy (approximately $75\%$). The Renewal, $\DisDRRRM$, and $\ConRM$ policies show similar improvements. Figure~\ref{Fig: Comaring CAD} also shows that while the Greedy policy initially improves TTT$_n$, this improvement disappears over time since the freeway becomes saturated for $\Cyclelength = 1$ as shown in Section~\ref{Subsection: (Sim) Avg Q lengths}. On the other hand, the choice of $\Cyclelength = 13$
in the $\GenRM$, $\DisDRRRM$, and $\ConRM$ policies will show steady improvement. 
\par
\mpcomment{Next, we evaluate the performance of the $\GenRM$ policy in scenarios where some vehicles lack V2I communication capabilities. In such cases, we assume that the $\GenRM$ policy relies only on the measurements obtained from the connected vehicles and ignores the non-connected vehicles. Figure~\ref{fig:max-travel-time} shows the average travel time, i.e., TTT$_{4000}$, under the safe-ALINEA and $\GenRM$ policies at different penetration rates of connected vehicles. As can be seen, the average travel time increases as the penetration rate decreases. However, even at low penetration rates, the $\GenRM$ policy provides improvement over the Safe-ALINEA policy. 
}
\par
Finally, \mpcomment{at a 100$\%$ penetration rate of connected vehicles}, Table~\ref{table: Performance} provides the time average queue size, i.e., $\overline{\Queuelength{}}(K\timestep) =\frac{1}{\numramps K} \sum_{i,j}\Queuelength{i}(j\timestep)$, where $\Queuelength{i}$ is averaged over the $10$ simulation rounds and $K\timestep$ is the simulation time. It can be seen that the Safe-ALINEA policy has the largest queue size compared to the other policies. This is mainly because of the safety filter added on top of the ALINEA policy. Without the safety filter, ALINEA releases vehicles more frequently, which results in shorter on-ramp queues on average but degrades safety. To see the latter, we use the Time-To-Collision (TTC) metric to compare safety under the Safe-ALINEA, ALINEA, and $\GenRM$ policies. The TTC is defined as the time left until a collision occurs between two vehicles if both vehicles continue on the same course and maintain the same speed (\citet{vogel2003comparison}). Generally, scenarios with a TTC of at least $6$ seconds is considered to be safe (\citet{vogel2003comparison}). Figure~\ref{Fig: Comaring TTC} compares the TTC statistics near on-ramp $2$, where we have discarded the TTC values higher than $20$ seconds. As expected, the ALINEA policy leads to more unsafe situations compared to the other two policies. Furthermore, one can see that the safety filter in the Safe-ALINEA policy cannot remove the low TTC values in the ALINEA policy. This is due to the fact that the traffic measurements in the Safe-ALINEA policy are still local and not as accurate as the $\GenRM$ policy.
\par
In conclusion, with proper choice of design parameters, the proposed policies \mpcomment{reduce the risk of collision at the merging bottlenecks without compromising travel time}.

\begin{table}[htb!]
\begin{center}
\caption{\sf Performance of the policies.}\label{table: Performance}
\begin{tabularx}{\textwidth}{ l l l l l l l }
 \hline
  Ramp Metering Policy & Renewal & $\GenRM$ & $\DisDRRRM$ & $\ConRM$ & Greedy & Safe-ALINEA \\
 \hline
 \rule{0pt}{3ex}
  Average travel time $[\text{min}]$ & 1.9 & 1.6 & 1.8 & 1.5 & 1.5 & 6.4\\
 \hline
 \rule{0pt}{3ex}
Average queue size & 156 & 8 & 15 & 9 & 14 & 677 \\
 \hline
\end{tabularx}
\end{center}
\end{table}

\subsection{The Capacity Drop Phenomenon}\label{Subsection: (Sim) capacity drop}
It is known that the maximum achievable traffic flow downstream of a merging bottleneck may decrease when queues are formed, i.e., the capacity drop phenomenon (\citet{hall1991freeway}). In this section, we compare the capacity drop downstream of on-ramp $2$ under the $\GenRM$ and Safe-ALINEA policies. For both policies, we use the parameters of Section~\ref{Subsection: (Sim) Average travel time}. Similar to Section~\ref{Subsection: (Sim) Average travel time}, we let the merging speed at on-ramps $1$ and $3$ be $\speedlim$, and $\speedlim/3$ at on-ramp $2$, and we let $\Arrivalrate{} = 0.455~[\text{veh}/\text{time step}]$. We consider an initial condition where the density is at its critical value (at which the mainline flow is maximized), i.e., the initial number of vehicles on the mainline is $60=\numcells$; each vehicle is at the constant free flow speed, and is at the safe distance $\timeheadway \speedlim + \standstilldist = 26.5~[\text{m}]$ from its leading vehicle. 
\par
Figure~\ref{Fig: Comaring flow} shows the simulation results, where we have let the simulation run idly in the first $5$ minutes without any vehicle entering or exiting the network. As can be seen, both policies lead to a capacity drop downstream of on-ramp $2$, but the capacity drop under the Safe-ALINEA policy is on average $15\%$ worse than the $\GenRM$ policy. This is not surprising as the $\GenRM$ policy uses more accurate traffic measurements to maintain the free flow state. We should note, however, that the capacity drop in both cases occur mainly because of the vehicle following safety considerations and the low merging speed at on-ramp $2$, rather than the formation of queues near the merging point, which is the usual reason for capacity drop. In particular, because of the low merging speed at on-ramp $2$, the minimum distance between vehicles on the mainline that is required by both policies is greater than $\timeheadway \speedlim + \standstilldist$ at the free flow state. \mpcomment{This would reduce the mainline flow when the mainline vehicles are at a distance greater than $\timeheadway \speedlim + \standstilldist$ but less than the threshold required by both policies}. It is worth discussing the case where the merging speed at an on-ramp is $\speedlim$. In this case, the capacity drop under the $\GenRM$ policy with $\Cyclelength=1$ will disappear after the vehicles reach the free flow state. This is because, at the free flow state, the minimum distance between vehicles on the mainline that is required by the $\GenRM$ policy is $\timeheadway \speedlim + \standstilldist$. Since $\Cyclelength=1$, every gap that is at least $\timeheadway \speedlim + \standstilldist$ will be used greedily by the on-ramp, thus maintaining the maximum flow.

%%%%%%%%%%%%%%%%%%%%%%%%%%%%%%%%%%%%%%%%%%%%%%%%%%%%%%%
%%%%%%%%%%%%%%%%%%%%% Figure %%%%%%%%%%%%%%%%%%%%%%%%%%
%%%%%%%%%%%%%%%%%%%%%%%%%%%%%%%%%%%%%%%%%%%%%%%%%%%%%%%
\begin{figure}[t!]
        \centering
        \includegraphics[width=0.5\textwidth]{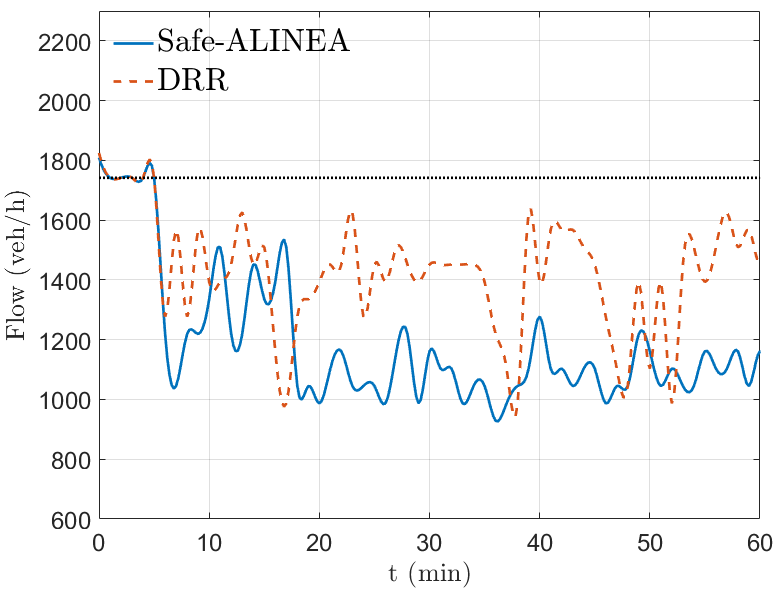}
        \vspace{0.2 cm}
    \caption{\centering\sf The traffic flow downstream of on-ramp $2$ under the Safe-ALINEA and $\GenRM$ policies. The dotted black line indicates the mainline capacity.}\label{Fig: Comaring flow}
\end{figure}

\section{Conclusion and Future Work}\label{Section: Conclusion}
We provided a microscopic-level RM framework subject to vehicle following safety constraints. This allows to explicitly take into account vehicle following safety and V2X communication scenarios in the design of RM, and study their impact on the freeway performance. We specifically provided RM policies \mpcomment{that operate under vehicle following safety constraints}, and analyzed performance in terms of their throughput. The proposed policies work in synchronous cycles during which an on-ramp does not release more vehicles than the number of vehicles waiting in its queue at the start of the cycle. The cycle length was variable in the Renewal policy, which could increase the total travel time. To prevent this issue, we used a fixed cycle length in the other policies. A fixed cycle length, however, can reduce the throughput when the merging speeds are less than the free flow speed. When the merging speeds are all equal to the free flow speed, all the proposed policies maximize the throughput. We compared the performance of the proposed policies with a well-known macroscopic RM policy \mpcomment{that relies on local traffic measurements obtained from roadside sensors.} We observed considerable improvements in the total travel time, capacity drop, and time-to-collision performance metrics.
\par
There are several avenues for generalizing the setup and methodologies initiated in this paper. Of immediate interest would be to consider a general network structure, and to derive sharper inner-estimates on the under-saturation region for the low merging speed case. We are also interested in expanding performance analysis to include travel time, possibly by leveraging waiting time analysis from queuing theory. \mpcomment{Finally, we plan to conduct simulations using a high-fidelity traffic simulator to evaluate the performance of the proposed policies in mixed traffic scenarios, where some vehicles are manually driven.}   
\mpcommentout{
In this paper, we studied the problem of determining ramp metering policies that $(i)$ are safe, $(ii)$ ensure a constant speed of flow, and $(iii)$ keep the freeway network undersaturated. The setup we considered was a single-lane ring road with multiple on/off-ramps populated with autonomous vehicles that obey safety, speed, and comfort rules. The arrival-destination of vehicles were stochastic which accounted for the uncertainty in the traffic demand. We first formalized a notion of "throughput" which determines a demand region beyond which at least one of the control objectives cannot be satisfied, no matter what type of ramp metering policy is adopted. We next designed a class of local ramp metering policies and showed that if the demand is within the throughput and the flow of vehicles is initially moving at the constant free flow speed, then the proposed policies achieve all control objectives. In other words, no coordination at the ramp level is required. We next showed that the initial speed of flow assumption can be relaxed in the presence of autonomous vehicles. Finally, we designed a class of coordinated policies that required no assumption on the initial speed or the level of autonomy in the network. We showed that these policies achieve all control objectives as long as the demand lies within the throughput.
\par
The effect of lane changing is not considered.
\par
The queue capacity of on-ramps is not infinite. }

\appendix

%%%%%%%%%%%%%%%%%%%%%%%%%%%%%%%%%%%%%%%%%%%%%%%%%%%%%%%%%%%%
%%%%%%%%%%%%%%%%%%%%%%%%%%%%%%%%%%%%%%%%%%%%%%%%%%%%%%%%%%%%
%%%%%%%%%%%%%%%%%%%%%%%% Appendix %%%%%%%%%%%%%%%%%%%%%%%%%%
%%%%%%%%%%%%%%%%%%%%%%%%%%%%%%%%%%%%%%%%%%%%%%%%%%%%%%%%%%%%
%%%%%%%%%%%%%%%%%%%%%%%%%%%%%%%%%%%%%%%%%%%%%%%%%%%%%%%%%%%%
\section{Network Specifications}\label{Section: (Appx) Model specifications}
%%%%%%%%%%%%%%%%%%%%%%%%%%%%%%%%%%%%%%%%%%%%%%%%%%%%%%%%%%%%
%%%%%%%%%%%%%%%%%%%%%%%%%%%%%%%%%%%%%%%%%%%%%%%%%%%%%%%%%%%%
%%%%%%%%%%%%%%%%%%%%%%%% Appendix %%%%%%%%%%%%%%%%%%%%%%%%%%
%%%%%%%%%%%%%%%%%%%%%%%%%%%%%%%%%%%%%%%%%%%%%%%%%%%%%%%%%%%%
%%%%%%%%%%%%%%%%%%%%%%%%%%%%%%%%%%%%%%%%%%%%%%%%%%%%%%%%%%%%
\subsection{Dynamics in the Speed Tracking Mode}\label{Section: (Appx) Cruise control dynamics}
Suppose that the ego vehicle is in the $\Speedmode$ mode for all $t \geq 0$, $v_e(0) = v_0 \in [0,\speedlim]$, and $a_e(0) = a_0 \in [\minaccel, \maxaccel]$, where $a_e$ is the acceleration and $\maxaccel$ is the maximum possible acceleration. Then, for a third-order vehicle dynamics, we have 
\begin{equation*}
    \begin{aligned}
        v_e(t) &= v_0 + \int_{0}^{t}a_e(\xi)d\xi \\
        a_e(t) &= \begin{cases*}
                \maxjerk t + a_0 & \mbox{if } $0 \leq t < t_1$ \\
                \maxaccel & \mbox{if } $t_1 \leq t < t_1 + t_2$ \\
                \maxaccel - \maxjerk t & \mbox{if } $t_1 + t_2 \leq t < t_1 + t_2 + t_3 $ \\
                0 & \mbox{if } $t \geq t_1 + t_2 + t_3$
            \end{cases*}, 
    \end{aligned}
\end{equation*}
where $\maxjerk$ is the maximum possible jerk, $t_1$ is the time at which the ego vehicle reaches the maximum acceleration, $t_2$ is the additional time required to reach a desired speed before it decelerates, and $t_3$ is the time required to reach the zero acceleration in order to avoid exceeding the speed limit $\speedlim$. Hence, $v_e(t_1 + t_2 + t_3) = \speedlim$, and $a_e(t_1 + t_2 + t_3) = 0$. The dynamics for the $v_0 > \speedlim$ case is similar.
%%%%%%%%%%%%%%%%%%%%%%%%%%%%%%%%%%%%%%%%%%%%%%%%%%%%%%%%%%%%%%%%%%%%%%%%%%%%%
%%%%%%%%%%%%%%%%%%%%%%%%%%%%%%%%%%%%%%%%%%%%%%%%%%%%%%%%%%%%%%%%%%%%%%%%%%%%%
%%%%%%%%%%%%%%%%%%%%%%%%%%%%%%% Sub-Section %%%%%%%%%%%%%%%%%%%%%%%%%%%%%%%%%
%%%%%%%%%%%%%%%%%%%%%%%%%%%%%%%%%%%%%%%%%%%%%%%%%%%%%%%%%%%%%%%%%%%%%%%%%%%%%
%%%%%%%%%%%%%%%%%%%%%%%%%%%%%%%%%%%%%%%%%%%%%%%%%%%%%%%%%%%%%%%%%%%%%%%%%%%%%

%\input{Estimation_MP}

%\input{Ramp_Spec_MP}

%%%%%%%%%%%%%%%%%%%%%%%%%%%%%%%%%%%%%%%%%%%%%%%%%%%%%%%%%%%%
%%%%%%%%%%%%%%%%%%%%%%%%%%%%%%%%%%%%%%%%%%%%%%%%%%%%%%%%%%%%
%%%%%%%%%%%%%%%%%%%%%%%% Appendix %%%%%%%%%%%%%%%%%%%%%%%%%%
%%%%%%%%%%%%%%%%%%%%%%%%%%%%%%%%%%%%%%%%%%%%%%%%%%%%%%%%%%%%
%%%%%%%%%%%%%%%%%%%%%%%%%%%%%%%%%%%%%%%%%%%%%%%%%%%%%%%%%%%%
\subsection{Communication Cost of a Ramp Metering Policy}\label{Section: (Appx) comm cost}
The calculation of the communication cost for a policy is inspired by the robotics literature (\citet[Remark 3.27]{bullo2009distributed}). We let $C_{ij}(k) = 1$ if vehicle $i$ communicates with on-ramp $j$ at time $k\timestep$, and $C_{ij}(k) = 0$ otherwise. Then, the total number of transmissions at time $k\timestep$ is $C(k) = \sum_{i \in [n], j \in [\numramps]}C_{ij}(k)$, where $n$ is the total number of vehicles in the network. The communication cost of the policy, measured in number of transmissions per $\timestep$ seconds, is obtained as follows:
\begin{equation*}
    C = \limsup_{K \ra \infty}\frac{1}{K}\sum_{k = 0}^{K-1}C(k).
\end{equation*}

%%%%%%%%%%%%%%%%%%%%%%%%%%%%%%%%%%%%%%%%%%%%%%%%%%%%%%%%%%%%
%%%%%%%%%%%%%%%%%%%%%%%%%%%%%%%%%%%%%%%%%%%%%%%%%%%%%%%%%%%%
%%%%%%%%%%%%%%%%%%%%%%%% Appendix %%%%%%%%%%%%%%%%%%%%%%%%%%
%%%%%%%%%%%%%%%%%%%%%%%%%%%%%%%%%%%%%%%%%%%%%%%%%%%%%%%%%%%%
%%%%%%%%%%%%%%%%%%%%%%%%%%%%%%%%%%%%%%%%%%%%%%%%%%%%%%%%%%%%
\section{Performance Analysis Tool}
\label{sec:performance-analysis}
Recall the initial condition from Section~\ref{Section: RM rules}, where the vehicles are in the free flow state and the location of each vehicle coincides with a slot for all times in the future. For this initial condition and under the proposed RM policies, we can cast the network as a discrete-time Markov chain. With a slight abuse of notation, we let $t = 0, 1, \ldots$ with time steps of duration $\timestep$ whenever we are talking about the underlying Markov chain. Let $M_{i}^{q}(t)$ be the vector of destination off-ramps of the vehicles that are waiting at on-ramp $i$ at $t$, arranged in the order of their arrival. Furthermore, let $M^s(t)=[M_{\ell}^{s}(t)]$ be the vector of the destination off-ramps of the occupants of the $\numcells + \numaccslots{a}$ slots: $M_{\ell}^{s}(t)=j$ if the vehicle occupying slot $\ell$ at time $t$ wants to exit from off-ramp $j$, and $M_{\ell}^{s}(t)=0$ if slot $\ell$ is empty at time $t$. %We use $|M(t)|$ to denote the number of vehicles on the mainline and acceleration lanes at time $t$. 
Consider the following discrete-time Markov chain with the state 
\begin{equation*}
    \StateofMC(t) := (M_{1}^{q}(t),\ldots,M_{\numramps}^{q}(t),M^s(t)).
%    \StateofMCExp_{\triangle}(t) := \StateofMC(t\triangle), \:\: t \geq 0,
\end{equation*}
\par
The transition probabilities of this chain are determined by the RM policy being analyzed, but will not be specified explicitly for brevity. For all the RM policies considered in this paper, the state $\StateofMC(t) = (0, 0)$ is reachable from all other states, and $\mathbb{P}\left(\StateofMC(t+1) = (0, 0)~|~ \StateofMC(t) = (0,0)\right) > 0$. Hence, the chain $\StateofMC$ is irreducible and aperiodic. In the the proofs of Theorem~\ref{Prop: stability of QRM policy for low merging speed} and \ref{Prop: stability of general aFCQ-RM}, we construct new Markov chains (that are also irreducible and aperiodic) by ``thinning" the chain $\StateofMC$.
%Let $\kscomment{\StateSpace} := \Z_{\rampindic}^{\infty} \times \rampindic \cup \{0\}$ be the range of values of $Y$. 
%We next use the Foster-Lyapunov drift criteria for $\StateofMCExp(\cdot)$ by considering 
\par
The following is an adaptation of a well-known result, e.g., see \citet[Theorem 14.0.1]{meyn2012markov}, for the setting of our paper. 
%%%%%%%%%%%%%%%%%%%%%%%%%%%%%%%%%%%%%%%%%%%%%%%%%%%%%%%%%%%%
%%%%%%%%%%%%%%%%%%%%%%%%%%%%%%%%%%%%%%%%%%%%%%%%%%%%%%%%%%%%
%%%%%%%%%%%%%%%%%%%%%%%% Theorem %%%%%%%%%%%%%%%%%%%%%%%%%%%
%%%%%%%%%%%%%%%%%%%%%%%%%%%%%%%%%%%%%%%%%%%%%%%%%%%%%%%%%%%%
%%%%%%%%%%%%%%%%%%%%%%%%%%%%%%%%%%%%%%%%%%%%%%%%%%%%%%%%%%%%
\begin{theorem} \label{Thm: f-positivity of MC}
(\textbf{Foster-Lyapunov drift criterion}) Let $\{\StateofMCExp(t)\}_{t = 1}^{\infty}$ be an irreducible and aperiodic discrete time Markov chain evolving on a countable state space $\mc \StateofMCExp$. Suppose that there exist $V: \mc \StateofMCExp \ra [0,\infty)$, $f: \mc \StateofMCExp \ra [1,\infty)$, a finite constant $b$, and a finite set $B \in \mc \StateofMCExp$ such that, for all $z \in \mc \StateofMCExp$, 
\begin{equation}\label{Eq: Foster-Lyapunov criteria}
        \E{\Lyap{\StateofMCExp(t+1)} - \Lyap{\StateofMCExp(t)}|\: \StateofMCExp(t) = z} \leq -f(z) + b\mathbbm{1}_{B}(z),
\end{equation} 
where $\mathbbm{1}_{B}(z)$ is the indicator function of the set $B$. Then, $\lim_{t \ra \infty}\E{f(\StateofMCExp(t))}$ exists and is finite. 
\end{theorem}
%%%%%%%%%%%%%%%%%%%%%%%%%%%%%%%%%%%%%%%%%%%%%%%%%%%%%%%%%%%%
%%%%%%%%%%%%%%%%%%%%%%%%%%%%%%%%%%%%%%%%%%%%%%%%%%%%%%%%%%%%
%%%%%%%%%%%%%%%%%%%%%%%% Theorem %%%%%%%%%%%%%%%%%%%%%%%%%%%
%%%%%%%%%%%%%%%%%%%%%%%%%%%%%%%%%%%%%%%%%%%%%%%%%%%%%%%%%%%%
%%%%%%%%%%%%%%%%%%%%%%%%%%%%%%%%%%%%%%%%%%%%%%%%%%%%%%%%%%%%

%%%%%%%%%%%%%%%%%%%%%%%%%%%%%%%%%%%%%%%%%%%%%%%%%%%%%%%
%%%%%%%%%%%%%%%%%%%%% Remark %%%%%%%%%%%%%%%%%%%%%%%%%%
%%%%%%%%%%%%%%%%%%%%%%%%%%%%%%%%%%%%%%%%%%%%%%%%%%%%%%%
%\ksmargin{maybe remove Remark~\ref{Remark: Implications of Foster Lyap thm}?}
%\ksmargin{Milad: define positive recurrence and what it means in traffic context}
\begin{remark}\label{Remark: Implications of Foster Lyap thm}
If the conditions of Theorem~\ref{Thm: f-positivity of MC} hold true, then $V$ is referred to as a \emph{Lyapunov} function. Additionally, if $f(\StateofMCExp(t)) = \|\Queuelength{}(t)\|_{\infty}$, where $\|\Queuelength{}(t)\|_{\infty}$ denotes the sup norm of the vector of queue sizes $\Queuelength{}(t)$, then 
$\lim_{t \ra \infty}\E{\|\Queuelength{}(t)\|_{\infty}} < \infty$, and hence $\limsup_{t \ra \infty}\E{\Queuelength{i}(t)} < \infty$ for all $i \in [\numramps]$. 
%\item Theorem~\ref{Thm: f-positivity of MC} also implies that $\{\StateofMCExp(t)\}_{t = 1}^{\infty}$ is \emph{positive recurrent}.  
\end{remark}
%%%%%%%%%%%%%%%%%%%%%%%%%%%%%%%%%%%%%%%%%%%%%%%%%%%%%%%
%%%%%%%%%%%%%%%%%%%%% Remark %%%%%%%%%%%%%%%%%%%%%%%%%%
%%%%%%%%%%%%%%%%%%%%%%%%%%%%%%%%%%%%%%%%%%%%%%%%%%%%%%%
We end this section by introducing a notation which will be used in the proofs of Theorem~\ref{Prop: stability of QRM policy for low merging speed} and \ref{Prop: stability of general aFCQ-RM} to construct Lyapunov functions and/or prove certain results about them. Given on-ramp $i \in [\numramps]$, we let $\Nodedegree{i}(t)$ be the \emph{degree} of on-ramp $i$ at time $t$, which represents the number of vehicles in the network at time $t$ that need to cross the merging point of on-ramp $i$ in order to reach their destination. 

%%%%%%%%%%%%%%%%%%%%%%%%%%%%%%%%%%%%%%%%%%%%%%%%%%%%%%%%%%%%
%%%%%%%%%%%%%%%%%%%%%%%%%%%%%%%%%%%%%%%%%%%%%%%%%%%%%%%%%%%%
%%%%%%%%%%%%%%%%%%%%%%%% Appendix %%%%%%%%%%%%%%%%%%%%%%%%%%
%%%%%%%%%%%%%%%%%%%%%%%%%%%%%%%%%%%%%%%%%%%%%%%%%%%%%%%%%%%%
%%%%%%%%%%%%%%%%%%%%%%%%%%%%%%%%%%%%%%%%%%%%%%%%%%%%%%%%%%%%
\section{Proofs of the Main Results}

%%%%%%%%%%%%%%%%%%%%%%%%%%%%%%%%%%%%%%%%%%%%%%%%%%%%%%%%%%%%
%%%%%%%%%%%%%%%%%%%%%%%%%%%%%%%%%%%%%%%%%%%%%%%%%%%%%%%%%%%%
%%%%%%%%%%%%%%%%%%%%%%%% Appendix %%%%%%%%%%%%%%%%%%%%%%%%%%
%%%%%%%%%%%%%%%%%%%%%%%%%%%%%%%%%%%%%%%%%%%%%%%%%%%%%%%%%%%%
%%%%%%%%%%%%%%%%%%%%%%%%%%%%%%%%%%%%%%%%%%%%%%%%%%%%%%%%%%%%
\subsection{Proof of Theorem \ref{Prop: stability of QRM policy for low merging speed}}\label{Section: (Appx) Proof of Q-RM prop}
For the sake of readability, we present proofs of intermediate claims at the end. Under the Renewal policy, the first cycle starts when the vehicles reach the free flow state, which is after at most $\Tempty$ seconds by (VC3). The free flow state and (M4) imply that any vehicle that is released thereafter remains in the $\Speedmode$ mode in the future, which implies that the location of the released vehicles coincide with a slot for all times in the future. Without loss of generality, let the start of the first cycle $t_1$ be at time $0$, and let the location of each vehicle coincides with a slot at time $0$. We can adopt the Markov chain setting from Appendix~\ref{sec:performance-analysis} with $\{\StateofMC(t_k)\}_{k \geq 1}$ as the Markov chain, where $t_k$ is the start of the $k$-th cycle. Consider the function $V: \StateSpace \ra [0,\infty)$
\begin{equation*}
    \Lyap{\StateofMC(t_k)} = \Cyclelength^{2}(k),
\end{equation*}
where $\StateSpace$ is the range of values of $\StateofMC$, and $\Cyclelength(k) = t_{k+1} - t_{k}$ is the length of the $k$-th cycle. We let $\Lyap{\StateofMC(t_k)} \equiv \Lyap{t_k}$ for brevity. It is shown at the end of the proof that
\begin{equation}\label{Eq: (Q-RM) cycle length 2nd key inequality}
    \Cyclelength(k+1) \leq \MaxNumArrival{}{}{\Cyclelength(k)} + \numcells + \numaccslots{a},
\end{equation}
where $\MaxNumArrival{}{}{\Cyclelength(k)}$ depends on the number of arrivals in the interval $[t_k+1,t_{k+1}]$ and satisfies the following: 
\begin{equation}\label{Eq: (Q-RM) implication of SLLN}
\lim_{t \ra \infty}\E{\left(\frac{\MaxNumArrival{}{}{\Cyclelength(k)}}{\Cyclelength(k)}\right)^2 \middle|~ \Cyclelength(k) = t} = \left(\max_{i \in [\numramps]}(\frac{\timestep_i}{\timestep}-1)\Avgload{i} - (\frac{\timestep_i}{\timestep} -2)\Arrivalrate{i}\right)^2.
%\lim_{t \ra \infty}\E{\left(\frac{\MaxNumArrival{}{}{\Cyclelength(k)}}{\Cyclelength(k)}\right)^2 \middle|~ \Cyclelength(k) = t} = \left(\max_{i \in [\numramps]}\lceil \frac{\timestep_i}{\timestep} \rceil\Avgload{i} - (\lceil \frac{\timestep_i}{\timestep} \rceil-1)\Arrivalrate{i}\right)^2.
\end{equation}
\par
By the assumption of the Theorem, $(\timestep_i/\timestep-1)\Avgload{i} - (\timestep_i/\timestep -2)\Arrivalrate{i} < 1$ for all $i \in [\numramps]$. Combining this with \eqref{Eq: (Q-RM) cycle length 2nd key inequality} and \eqref{Eq: (Q-RM) implication of SLLN}, it follows that
\begin{equation*}
\begin{aligned}
        \limsup_{t \ra \infty}\E{\left(\frac{\Cyclelength(k+1)}{\Cyclelength(k)}\right)^2 \middle|~ \Cyclelength(k) = t} &\leq \limsup_{t \ra \infty}\E{\left(\frac{\MaxNumArrival{}{}{\Cyclelength(k)}}{\Cyclelength(k)}\right)^2 \middle|~ \Cyclelength(k) = t} \\
        &= \left(\max_{i \in [\numramps]}( \frac{\timestep_i}{\timestep}-1)\Avgload{i} - (\frac{\timestep_i}{\timestep} -2)\Arrivalrate{i}\right)^2 < 1.
%        &= \left(\max_{i \in [\numramps]}\lceil \frac{\timestep_i}{\timestep} \rceil\Avgload{i} - (\lceil \frac{\timestep_i}{\timestep} \rceil-1)\Arrivalrate{i}\right)^2 < 1.
\end{aligned}
\end{equation*}
\par
Therefore, there exists $\delta \in (0,1)$ and $T > 0$ such that for all $t > T$ we have
\begin{equation*}
    \E{\left(\frac{\Cyclelength(k+1)}{\Cyclelength(k)}\right)^2 \middle|~ \Cyclelength(k) = t} < 1 - \delta,
\end{equation*}
which in turn implies that
\begin{equation*}
    \E{\Cyclelength^2(k+1) - \Cyclelength^2(k) \middle|~ \Cyclelength(k) > T} < -\delta \Cyclelength^2(k).
\end{equation*}
\par
Since $\Queuelength{i}(t_k) \leq \Cyclelength(k) \leq \Cyclelength^2(k)$ for all $i \in [\numramps]$, it follows that $\E{\Cyclelength^2(k+1) - \Cyclelength^2(k) \middle|~ \Cyclelength(k) > T} < -\delta\|\Queuelength{}(t_k)\|_{\infty}$, where $\|\Queuelength{}\|_{\infty} = \max_{i}\Queuelength{i}$ as defined in Appendix~\ref{sec:performance-analysis}. 
\par
If $\Cyclelength(k) \leq T$, then $\MaxNumArrival{}{}{\Cyclelength(k)}$ is shown in \eqref{eq:boundedness-of-Atilde} to satisfy $\MaxNumArrival{}{}{\Cyclelength(k)} \leq \numramps T \max_{is \in [\numramps]} (\timestep_i/\timestep-1)$. Combining this with \eqref{Eq: (Q-RM) cycle length 2nd key inequality} gives $\Cyclelength(k+1) \leq \numramps T \max_{i \in [\numramps]} (\timestep_i/\timestep-1)  + \numcells + \numaccslots{a}$. Therefore, 
    $\E{\Cyclelength^2(k+1) \middle|~ \Cyclelength(k) \leq T} \leq
    %&\leq \left(\numramps T \max_{i \in [\numramps]} (\timestep_i/\timestep-1)  + 2\numcells + \numaccslots{a}\right)^2 \\ 
     \left(\numramps T \max_{i \in [\numramps]} (\timestep_i/\timestep-1)  + \numcells + \numaccslots{a}\right)^2 + \Cyclelength^2(k) -\delta\|\Queuelength{}(t_k)\|_{\infty}$,
where we have used $\delta\|\Queuelength{}(t_k)\|_{\infty} \leq \|\Queuelength{}(t_k)\|_{\infty} \leq \Cyclelength^2(k)$.
\par
Combining all the previous steps gives
\begin{equation*}
    \E{V(t_{k+1}) - V(t_k) \middle|~ \StateofMC(t_k)} \leq -\delta \|\Queuelength{}(t_k)\|_{\infty} + \left(\numramps T \max_{i \in [\numramps]}(\frac{\timestep_i}{\timestep}-1) + \numcells + \numaccslots{a}\right)^2\mathbbm{1}_{B},
%    \E{V(t_{k+1}) - V(t_k) \middle|~ V(t_k)} \leq -\delta \|\Queuelength{}(t_k)\|_{\infty} + \left(\numramps T \max_{i \in [\numramps]}\frac{\timestep_i}{\timestep} + 2\numcells + \numaccslots{a}\right)^2\mathbbm{1}_{B},
\end{equation*}
where $B = \{\StateofMC(t_k): V(t_k) \leq T^2\}$ (a finite set). The result then follows from Theorem~\ref{Thm: f-positivity of MC}.
\subsection*{\underline{Proof of \eqref{Eq: (Q-RM) cycle length 2nd key inequality}}}
Consider the $(k+1)$-th cycle. Let $s_{i} \in [t_{k+1}, t_{k+2})$ be the time at which the queue at on-ramp $i \in [\numramps]$ becomes empty of the $(k+1)$-th cycle quota. We claim that
\begin{equation*}
%    s_{i} - t_{k+1} \leq \lceil \frac{\timestep_i}{\timestep} \rceil\Nodedegree{i}(t_{k+1}) - (\lceil \frac{\timestep_i}{\timestep} \rceil - 1)|\Queuelength{i}(t_{k+1})|,
    s_{i} - t_{k+1} \leq (\frac{\timestep_i}{\timestep}-1) \Nodedegree{i}(t_{k+1}) - ( \frac{\timestep_i}{\timestep} - 2)\Queuelength{i}(t_{k+1})+\numcells+\numaccslots{a},
\end{equation*}
where $\Nodedegree{i}$ is defined in Appendix~\ref{sec:performance-analysis}. Suppose not, i.e.,
\begin{equation}\label{eq:(Renewal)contradiction-assumption}
    s_{i} - t_{k+1} - \left((\frac{\timestep_i}{\timestep}-1)( \Nodedegree{i}(t_{k+1}) - \Queuelength{i}(t_{k+1}))+\numcells+\numaccslots{a}\right) > \Queuelength{i}(t_{k+1}).
\end{equation}
\par
We will use \eqref{eq:(Renewal)contradiction-assumption} to reach a contradiction. If at least one upstream vehicle crosses the merging point of on-ramp $i$ in $\timestep_i/\timestep - 1$ time steps, then the last acceleration lane slot of on-ramp $i$ that is not on the mainline must be empty because of the vehicle following safety consideration (M4). Since the total number of vehicles from other on-ramps that need to cross the merging point of on-ramp $i$ is $\Nodedegree{i}(t_{k+1}) - \Queuelength{i}(t_{k+1})$, and the number of vehicles from cycle $k$ that are still on the freeway at the start of the $(k+1)$-th cycle is at most $\numcells+\numaccslots{a}$, the aforementioned acceleration lane slot can be empty of the $(k+1)$-th cycle quota for at most $(\timestep_i/\timestep-1)(\Nodedegree{i}(t_{k+1}) - \Queuelength{i}(t_{k+1}))+\numcells+\numaccslots{a}$ time steps. Therefore, it must be occupied for at least $s_{i}-t_{k+1} - \left((\timestep_i/\timestep-1)(\Nodedegree{i}(t_{k+1}) - \Queuelength{i}(t_{k+1}))+\numcells+\numaccslots{a}\right)$ time steps, which by \eqref{eq:(Renewal)contradiction-assumption} is greater than $\Queuelength{i}(t_{k+1})$. This, however, contradicts the feature of the Renewal policy under which the number of vehicles released by an on-ramp during a cycle does not exceed the queue size at the start of the cycle.
%consider all the mainline slots upstream of the merging point that are at most $\lceil \timestep_{i}/\timestep \rceil$ time steps away. Whenever at least one of these slots is occupied during the time interval $[t_{k+1},s_{i})$, the last acceleration lane slot of on-ramp $i$ that is not on the mainline must be empty (see (M4)). Since the number of vehicles from other on-ramps that need to cross the merging point of on-ramp $i$ is $\Nodedegree{i}(t_{k+1}) - |\Queuelength{i}(t_{k+1})|$, the aforementioned acceleration lane slot is empty for at most $\lceil \timestep_i/\timestep \rceil(\Nodedegree{i}(t_{k+1}) - |\Queuelength{i}(t_{k+1})|)$ time steps. Therefore, it must be occupied for $s_{i}-t_{k+1} - \left(\lceil \timestep_i/\timestep \rceil(\Nodedegree{i}(t_{k+1}) - |\Queuelength{i}(t_{k+1})|)\right)$ time steps, which by assumption is greater than $|\Queuelength{i}(t_{k+1})|$. This however contradicts the feature of the $\QRM$ policy under which the number of vehicles released by an on-ramp during a cycle does not exceed the queue length at the beginning of the cycle.   
\par
%\mpmargin{why $2n_c$?}
Since $\Cyclelength(k+1) \leq \max_{i \in [\numramps]}s_{i} - t_{k+1} + \numcells + \numaccslots{a}$, it follows that
\begin{equation}\label{Eq: (Q-RM) cycle length key inequality}
    \Cyclelength(k+1) \leq \max_{i \in [\numramps]} \left\{(\frac{\timestep_i}{\timestep}-1) \Nodedegree{i}(t_{k+1}) - ( \frac{\timestep_i}{\timestep} - 2)\Queuelength{i}(t_{k+1}) \right\} + \numcells + \numaccslots{a}.
%    \Cyclelength(k+1) \leq \max_{i \in [\numramps]} \left\{\lceil \frac{\timestep_i}{\timestep} \rceil\Nodedegree{i}(t_{k+1}) - (\lceil \frac{\timestep_i}{\timestep} \rceil - 1)|\Queuelength{i}(t_{k+1})| \right\} \mpcomment{+ 2\numcells + \numaccslots{a}}.
\end{equation}
\par
Let $\Numarrivals{i}(s)$ be the number of arrivals to on-ramp $i$ at time $s$, and $\Numarrivals{\ell,i}(s)$ be the number of arrivals to all the on-ramps at time $s$ that need to cross link $i$. We let $\Cumularrivals{i}{t_{k+1} - t_{k}}$ \footnote{In this context, the beginning and end of each interval is specified whenever needed. Thus, the notation $\Cumularrivals{i}{t_{k+1} - t_{k}}$ should not create any confusion.} denote the cumulative number of arrivals to on-ramp $i$ during the interval $[t_{k} + 1,t_{k+1}]$, i.e.,
\begin{equation*}
    \Cumularrivals{i}{t_{k+1} - t_{k}} = \sum_{s = t_k + 1}^{t_{k+1}}\Numarrivals{i}(s).
\end{equation*}
\par
We define $\Cumularrivals{\ell,i}{t_{k+1} - t_{k}}$ similarly. Note that $\Queuelength{i}(t_{k+1})$ is precisely the cumulative number of arrivals to on-ramp $i$ in $[t_k+1, t_{k+1}]$, i.e., $\Queuelength{i}(t_{k+1}) = \Cumularrivals{i}{t_{k+1} - t_k}$. Similarly, $\Nodedegree{i}(t_{k+1}) = \Cumularrivals{\ell,i}{t_{k+1} - t_{k}}$. This and \eqref{Eq: (Q-RM) cycle length key inequality} imply \eqref{Eq: (Q-RM) cycle length 2nd key inequality} with 
\begin{equation*} 
\MaxNumArrival{}{}{\Cyclelength(k)} = \max_{i \in [\numramps]}\left\{(\frac{\timestep_i}{\timestep}-1)\Cumularrivals{\ell,i}{t_{k+1} - t_k} - (\frac{\timestep_i}{\timestep}-2)\Cumularrivals{i}{t_{k+1} - t_k} \right\}.
%    \MaxNumArrival{}{}{\Cyclelength(k)} = \max_{i \in [\numramps]}\left\{\lceil \frac{\timestep_i}{\timestep} \rceil\Cumularrivals{\ell,i}{t_{k+1} - t_k} - (\lceil \frac{\timestep_i}{\timestep} \rceil - 1)\Cumularrivals{i}{t_{k+1} - t_k} \right\}.
\end{equation*}
\par
Note that because the number of arrivals to all the on-ramps is at most $\numramps$ at each time step, we have
\begin{equation}\label{eq:boundedness-of-Atilde}
    \MaxNumArrival{}{}{\Cyclelength(k)} \leq \numramps \Cyclelength(k) \max_{i \in [\numramps]} (\timestep_i/\timestep-1).
\end{equation}
\subsection*{\underline{Proof of \eqref{Eq: (Q-RM) implication of SLLN}}}
Consider the sequences $\{\Numarrivals{i}(s)\}_{s = t_k + 1}^{\infty}$ and $\{\Numarrivals{\ell,i}(s)\}_{s = t_k + 1}^{\infty}$. Each sequence is i.i.d and $\E{\Numarrivals{i}(s)} = \Arrivalrate{i}$, $\E{\Numarrivals{\ell,i}(s)} = \Avgload{i}$. By the strong law of large numbers, with probability one,
\begin{equation*}
     \lim_{\Cyclelength(k) \ra \infty} \frac{( \timestep_i/\timestep-1) \Cumularrivals{\ell,i}{t_{k+1} - t_k} - ( \timestep_i/\timestep - 2)\Cumularrivals{i}{t_{k+1} - t_k}}{\Cyclelength(k)} = (\frac{\timestep_i}{\timestep}-1)\Avgload{i} - (\frac{\timestep_i}{\timestep}-2)\Arrivalrate{i}.
%     \lim_{\Cyclelength(k) \ra \infty} \frac{\lceil \timestep_i/\timestep \rceil\Cumularrivals{\ell,i}{t_{k+1} - t_k} - (\lceil \timestep_i/\timestep \rceil - 1)\Cumularrivals{i}{t_{k+1} - t_k}}{\Cyclelength(k)} = \lceil \frac{\timestep_i}{\timestep} \rceil\Avgload{i} - (\lceil \frac{\timestep_i}{\timestep} \rceil-1)\Arrivalrate{i}.
\end{equation*}
\par
If the sequence $\left\{\left(\frac{\MaxNumArrival{}{}{\Cyclelength(k)}}{\Cyclelength(k)}\right)^2\right\}_{\Cyclelength(k)=1}^{\infty}$ is upper bounded by an integrable function, \mpcomment{then similar to the proof of Lemma 1 in \citet{georgiadis1995scheduling} it follows that}
\begin{equation*}
\lim_{\Cyclelength(k) \ra \infty} \E{\left(\frac{\MaxNumArrival{}{}{\Cyclelength(k)}}{\Cyclelength(k)}\right)^2} = \left(\max_{i \in [\numramps]}( \frac{\timestep_i}{\timestep}-1)\Avgload{i} - (\frac{\timestep_i}{\timestep} -2)\Arrivalrate{i}\right)^2,
\end{equation*}
which in turn gives \eqref{Eq: (Q-RM) implication of SLLN}. 
The upper bound follows from the following fact: the number of arrivals to all the on-ramps is at most $\numramps$ at each time step. Thus,
\begin{equation*}
\begin{aligned}
    \frac{\MaxNumArrival{}{}{\Cyclelength(k)}}{\Cyclelength(k)} &= \frac{1}{\Cyclelength(k)}\max_{i \in [\numramps]}\left((\frac{\timestep_i}{\timestep}-1 )\Cumularrivals{\ell,i}{t_{k+1} - t_k} - ( \frac{\timestep_i}{\timestep} - 2)\Cumularrivals{i}{t_{k+1} - t_k}\right) \\
    &\leq \frac{1}{\Cyclelength(k)} \max_{i \in [\numramps]} (\frac{\timestep_i}{\timestep}-1 )\Cumularrivals{\ell,i}{t_{k+1} - t_k} \leq  \max_{i \in [\numramps]} ( \frac{\timestep_i}{\timestep}-1 ) \numramps,
%    \frac{\MaxNumArrival{}{}{\Cyclelength(k)}}{\Cyclelength(k)} &= \frac{1}{\Cyclelength(k)}\max_{i \in [\numramps]}\left(\lceil \frac{\timestep_i}{\timestep} \rceil\Cumularrivals{\ell,i}{t_{k+1} - t_k} - (\lceil \frac{\timestep_i}{\timestep} \rceil - 1)\Cumularrivals{i}{t_{k+1} - t_k}\right) \\
%    &\leq \frac{1}{\Cyclelength(k)} \max_{i \in [\numramps]} \lceil \frac{\timestep_i}{\timestep} \rceil\Cumularrivals{\ell,i}{t_{k+1} - t_k} \leq  \max_{i \in [\numramps]} \lceil \frac{\timestep_i}{\timestep} \rceil \numramps,
\end{aligned}
\end{equation*}
as desired.
\subsection{Proof of Theorem \ref{Prop: stability of general aFCQ-RM}}\label{Section: (Appx) Proof of general aFCQ-RM prop}
We first show that for any initial condition, there exists $k_0 \in \N$ such that $\StateofDySys_f(k\updateperiod) = 0$ for all $k \geq k_0$. That is, $\StateofDySys_{f_1}(k\updateperiod) = 0$ for all $k \geq k_0$ and $\StateofDySys_{f_2}(t) = 0$ for all $t \geq k_0\updateperiod$. The equality $\StateofDySys_{f_2}(t) = 0$ for all $t \geq k_0\updateperiod$ implies that each vehicle remains at a safe distance with respect to its leading vehicle, and predicts to be at a safe distance with respect to its virtual leading vehicle in merging areas. This and (VC2) imply that if a vehicle is in the $\Speedmode$ at time $k_0\updateperiod$, it remains in this mode in the future. Furthermore, if a vehicle is in the $\Safetymode$ mode at time $k_0\updateperiod$, $\StateofDySys_{f_1}(k_0\updateperiod) = 0$ implies that it moves at the constant speed $\speedlim$. Together with (VC1) and $\StateofDySys_{f_2}(t) = 0$ for all $t \geq k_0\updateperiod$, it follows that the vehicle maintains its speed in the future. All this would imply that the vehicles reach and remain in the free flow state at time $k_0\updateperiod$. Hence, after a finite time after $k_0\updateperiod$, $\Releasetime{}(\cdot)=0$, the location of all the vehicles coincide with a slot, and we can use the Markov chain setting from Appendix~\ref{sec:performance-analysis}.
%%%%%%%%%%%%%%%%% Commented (V2)(10/31/22) %%%%%%%%%%%%%%%%%%%
\mpcommentout{
For any initial condition, we first show that $\|\StateofDySys^G(t_0)\| = 0$ after some finite time $t_0 \geq 0$. That is, at time $t_0$, $\min\{0, y_e(t_0) - S_e(t_0), \inf_{s \in [t_m, t_f]}\hat{y}_e(s) - \hat{S}_e(s)\} = 0$ and every vehicle is either moving at the constant speed $\speedlim$, and/or it has been in the $\Speedmode$ mode since being released. Combining $\min\{0, y_e(t_0) - S_e(t_0), \inf_{s \in [t_m, t_f]}\hat{y}_e(s) - \hat{S}_e(s)\} = 0$ with (VC1)-(VC2) imply that every vehicle either maintains the constant speed $\speedlim$ or remains in the $\Speedmode$ mode in the future, i.e., vehicles reach the free flow state at time $t_0$. Combining this with (M4), it follows that the vehicles being released after $t_0$ will remain in the $\Speedmode$ mode. All this would imply that $\|\StateofDySys^G(t)\| = 0$, for all $t \geq t_0$, i.e., vehicles will remain in the free flow state. Hence, after a finite time after $t_0$, $\Releasetime{}(\cdot)=0$, and the location of all vehicles coincide with a slot and we can use the Markov chain setting from Appendix~\ref{sec:performance-analysis}.
}
%%%%%%%%%%%%%%%%% Commented (V2)(10/31/22) %%%%%%%%%%%%%%%%%%%
\par
Recall from Remark~\ref{remark:X_f-bound} that $v_e$ and $\delta_e + \hat{\delta}_e$ are bounded. Let $\Perimeter_a$ be the total length of all the acceleration lane. Since successive releases from an on-ramp are at least $\timestep$ seconds apart, the number of vehicles that communicate $\delta_e + \hat{\delta}_e$ in $[t-\updateperiod,t]$ is at most $\numramps\updateperiod/\timestep+(P + P_a)/\vehiclelength$. One can then easily show that $X_f(t)$ is upper bounded by
\begin{equation*}
    \bar{X} := \frac{\numramps\updateperiod}{\timestep}+2\frac{P + P_a}{\vehiclelength}. %+ \frac{P + P_a}{\vehiclelength}%(\max\{\overline{V} - \speedlim, \speedlim\} +  \max\{\maxaccel, |\minaccel|\}). 
\end{equation*}
\par
To show $\StateofDySys_f(\cdot)=0$ after a finite time, we use a proof by contradiction. Suppose that $\StateofDySys_f(k\updateperiod)\neq 0$ for infinitely many $k \in \N$. Since $\StateofDySys_f(k\updateperiod) \leq \bar{X}$ for all $k \in \N$, there exists an infinite sequence $\{k_n\}_{n \geq 1}$ such that $\StateofDySys_f(k_n\updateperiod) > \StateofDySys_f((k_n - 1)\updateperiod) - \DesDecreaseconst$ for all $n \geq 1$. This implies that $\ReleasetimeInc{}(k_n \updateperiod) = \ReleasetimeAdj{} \ReleasetimeInc{}((k_n-1) \updateperiod)$ for all $n \geq 1$. Since $\ReleasetimeInc{}(\cdot)$ is non-decreasing and $\ReleasetimeAdj{} > 1$, it follows that $\lim_{t \ra \infty}\ReleasetimeInc{}(t) = \infty$, which in turn implies that $\limsup_{t \ra \infty}\Releasetime{}(t) = \infty$. Let $t_f$ be such that $\Releasetime{}(t_f) > \numramps \Tempty(1 + \ReleasetimeDecConst{2}/\updateperiod)$. Note that for all $t \in [t_f, t_f + \numramps\Tempty]$, $g(t) > \numramps\Tempty$. Thus, each on-ramp releases at most one vehicle during the interval $[t_f, t_f + \numramps\Tempty]$. Hence, there exists a time interval of length at least $\Tempty$ seconds in $[t_f, t_f + \numramps\Tempty]$ during which no on-ramp releases a vehicle. Condition (VC3) then implies that the vehicles will reach the free flow state at the end of such a time interval. This and (M4) imply that any vehicle that is released thereafter will remain in the $\Speedmode$ mode. Thus, $\StateofDySys_f(k\updateperiod) = 0$ for all $k \geq k_0$ for some $k_0 \in \N$; a contradiction to the assumption that $\StateofDySys_f(k\updateperiod)\neq 0$ for infinitely many $k$.
\par
Without loss of generality, we let $\Releasetime{}(0) = 0$, and we assume that the vehicles are initially in the free flow state, and their location coincide with a slot as in Appendix~\ref{sec:performance-analysis}. For the sake of readability, we present proofs of intermediate claims at the end. We adopt the Markov chain setting from Appendix~\ref{sec:performance-analysis} with $\{\StateofMC(t\triangle)\}_{t \geq 0}$ as the Markov chain, where $\triangle$ is an exact multiple of the cycle length $\Cyclelength$ so that $\StateofMC$ satisfies the Markov property, and is to be determined in the rest of the proof. Consider the function $V: \StateSpace \ra [0,\infty)$ given by 
\begin{equation}
\label{eq:lyap-def}
 \Lyap{t\triangle} \equiv \Lyap{\StateofMC(t\triangle)} := \Nodedegree{}^2(t\triangle),
\end{equation}
where $\StateSpace$ is the range of values of $Y$ in Appendix~\ref{sec:performance-analysis}, and $\Nodedegree{}(\cdot) = \max_{i \in [\numramps]}\Nodedegree{i}(\cdot)$, where $\Nodedegree{i}$ is the degree of on-ramp $i$ as defined in Appendix~\ref{sec:performance-analysis}. Note that \eqref{eq:lyap-def} implies $\Lyap{(t+1)\triangle} - \Lyap{t\triangle} = \Nodedegree{}^2((t+1)\triangle) - \Nodedegree{}^2(t\triangle)$. We claim that if $\Lyap{t\triangle}$ is large enough, specifically if $\Lyap{t\triangle} > L := \left(\numramps \triangle^2 + \numcells +\numaccslots{a}\right)^2$, then there exists $i \in [\numramps]$ such that $\Nodedegree{i}$ decreases by at least one in every $ \timestep_i/\timestep -1$ time steps in the interval $[t\triangle, (t+1)\triangle]$, without considering the new arrivals. Note that $\Lyap{t\triangle} > L$ implies $\Nodedegree{}(t\triangle) > \numramps \triangle^2 + \numcells + \numaccslots{a}$, 
%If not, then since the total number of arrivals to the network is bounded by $\numramps$ at each time step, it follows for all $s \in [t - k\Cyclelength + 1, t]$ that
%\begin{equation*}
% N(s) \leq N(t - k\Cyclelength + 1) + \numramps (s - t + k\Cyclelength - 1) 
% \leq \numramps k^2 \Cyclelength + \numcells + \numaccslots{a} + \numramps(s - t + k\Cyclelength - 1).   
%\end{equation*} 
%\par
%This would lead to the contradiction that 
%\begin{equation*}
%    V(t) =  \sum_{s = t - k\Cyclelength + 1}^{t}\Nodedegree{}^2(s) \leq \sum_{s = t - k\Cyclelength + 1}^{t} \left(\numramps k^2 \Cyclelength + \numcells + \numaccslots{a} + \numramps(s - t + k\Cyclelength - 1) \right)^2 = L. 
%\end{equation*}
%\par
which in turn implies that there exists at least one on-ramp, say $q \in [\numramps]$, such that $\Queuelength{q}(t\triangle) > \triangle^2$. If not, then $N(t\triangle) \leq \sum_{i} \Queuelength{i}(t\triangle) + \numcells + \numaccslots{a} \leq \numramps \triangle^2 + \numcells + \numaccslots{a}$, which is a contradiction. Let $\triangle > \triangle_1 := \max_{i \in [\numramps]}\{\timestep_i/\timestep-1\}$. If, in addition, $\triangle \geq 1 + \sqrt{3(\lceil \numaccslots{q}/\Cyclelength\rceil + 1)\Cyclelength}$, then we show at the end of the proof that for all $w \in [t\triangle, (t+1)\triangle - \timestep_q/\timestep + 1]$ we have
\begin{equation}
\label{eq:k-min}
    \Nodedegree{q}(w + \frac{\timestep_q}{\timestep}-1) \leq \Nodedegree{q}(w) - 1 + \sum_{s = w+1}^{w + \timestep_q/\timestep -1} \Numarrivals{\ell, q}(s),
%    \Nodedegree{q}(w + \lceil\frac{\timestep_q}{\timestep}\rceil) \leq \Nodedegree{q}(w) - 1 + \sum_{s = w+1}^{w + \lceil \timestep_q/\timestep \rceil} \Numarrivals{f, q}(s),
\end{equation}
where $\Numarrivals{\ell, q}(s)$ is the number of arrivals to all the on-ramps at time $s$ that need to cross link $q$. By summing up \eqref{eq:k-min} over disjoint sub-intervals of length $\timestep_q/\timestep - 1$ we obtain 
\begin{equation}\label{Eq: Inequality for off-ramp j node degree that decreases at each time}
    \begin{aligned}
        \Nodedegree{q}((t+1)\triangle) &\leq \Nodedegree{q}(t\triangle) - \frac{\triangle}{\timestep_q/\timestep - 1} + 1 + \sum_{s=t\triangle+1}^{(t+1)\triangle} \Numarrivals{\ell, q}(s) \\
         &\leq \Nodedegree{}(t\triangle) - \frac{\triangle}{\timestep_q/\timestep - 1} + 1 + \sum_{s=t\triangle+1}^{(t+1)\triangle} \Numarrivals{\ell, q}(s).
%        \Nodedegree{q}(t + 1) &\leq \Nodedegree{q}(t - k\Cyclelength + 1) - \frac{k\Cyclelength}{\lceil\timestep_q/\timestep\rceil} + 1 + \sum_{s=t-k\Cyclelength+2}^{t+1} \Numarrivals{f, q}(s) \\
%         &\leq \Nodedegree{}(t - k\Cyclelength + 1) - \frac{k\Cyclelength}{\lceil\timestep_q/\timestep\rceil} + 1 + \sum_{s=t-k\Cyclelength+2}^{t+1} \Numarrivals{f, q}(s).
    \end{aligned}
\end{equation}
\par
We show at the end of the proof that a similar form of \eqref{Eq: Inequality for off-ramp j node degree that decreases at each time} holds for all the on-ramps. Letting $\triangle_2 := \max_{i \in [\numramps]}\{1 + \sqrt{3(\lceil \numaccslots{i}/\Cyclelength\rceil + 1)\Cyclelength}\}$, we then show that for all $\triangle \geq \max\{\triangle_1, \triangle_2\}$,
\begin{equation}\label{Eq: (Prop 4) Inequality for off-ramp p node degree that resembeles that of off-ramp j}
    \begin{aligned}
    \Nodedegree{}((t+1)\triangle) \leq \Nodedegree{}(t\triangle) - \delta \triangle + C + \MaxNumArrival{}{}{\triangle},
    \end{aligned}
\end{equation}
where $\delta \in (0,1)$, $C := (\Cyclelength + \numcells + 1)\numramps + 2\numaccslots{a}$, and $\MaxNumArrival{}{}{\triangle}$ depends on the number of arrivals in $[t\triangle+1,(t+1)\triangle]$ (thus, it is independent of $\Nodedegree{}(t\triangle)$) and is shown to have the following properties \footnote{The expression for $\MaxNumArrival{}{}{\triangle}$ in this proof is different than $\MaxNumArrival{}{}{\Cyclelength(k)}$ in the proof of Theorem~\ref{Prop: stability of QRM policy for low merging speed}. However, they play similar roles in the proofs which justifies the same notation.}: there exist $\eps, \triangle_3 > 0$ such that $\eps < \delta$ and for all $\triangle \geq \triangle_3$ we have
\begin{equation}
\label{eq:Kolmogorov-implication}
\E{\MaxNumArrival{}{}{\triangle}} < \eps \triangle, \quad \E{\MaxNumArrival{}{2}{\triangle}} < 2\eps(\triangle)^2.
\end{equation}  
\par
The inequality \eqref{Eq: (Prop 4) Inequality for off-ramp p node degree that resembeles that of off-ramp j} implies that
\begin{equation}\label{Eq: (Prop 4) Key inequality for maximum node degree squared}
\begin{aligned}
        \Nodedegree{}^2((t+1)\triangle) \leq \left(\Nodedegree{}(t\triangle) - \delta \triangle + C\right)^2 + \MaxNumArrival{}{2}{\triangle} 
         + 2\left(\Nodedegree{}(t\triangle) - \delta \triangle + C\right)\MaxNumArrival{}{}{\triangle}.
\end{aligned}
\end{equation}
\par
By choosing $\triangle \geq \triangle_3$ and taking conditional expectation from both sides of \eqref{Eq: (Prop 4) Key inequality for maximum node degree squared}, we obtain
\begin{equation}
\label{eq:lyap-diff}
\begin{aligned}
        \E{\Nodedegree{}^2((t+1)\triangle) - \Nodedegree{}^2(t\triangle)\:|\: \Lyap{t\triangle} > L} \leq -2\Nodedegree{}(t\triangle)((\delta - \eps) 
        \triangle- C) + (\delta \triangle - C)^2
        + 2\eps \triangle(\triangle - \delta \triangle + C).
\end{aligned}
\end{equation}
\par
By definition, the degree of each on-ramp is at least as large as its queue length. Hence, 
%the number of on-ramp arrivals at each time step is at most one, we have
$\|\Queuelength{}(t\triangle)\|_{\infty} = \max_{i}\Queuelength{i}(t\triangle) \leq \Nodedegree{}(t\triangle)$. Plugging into \eqref{eq:lyap-diff} gives for all $\triangle \geq \max\{\triangle_1 , \triangle_2, \triangle_3, \frac{C}{\delta -\eps}\}$,
\begin{equation*}
\begin{aligned}
\mathbb{E}\big[\Nodedegree{}^2((t+1)\triangle) &- \Nodedegree{}^2(t\triangle)\:|\: \Lyap{t\triangle} > L \big]\\
&\leq -2 \|\Queuelength{}(t\triangle)\|_{\infty}\left((\delta - \eps) \triangle - C\right) + (\delta \triangle - C)^2 
        + 2\eps \triangle(\triangle - \delta \triangle + C)
%\\
%&\leq -2 \max_{s \in [t - k\Cyclelength + 1, t]} \|\Queuelength{}(s)\|_{\infty}  \left((\delta - \eps) k \Cyclelength - C\right) + (\delta k \Cyclelength - C)^2 + 2k \Cyclelength (1 - \eps)(\delta k \Cyclelength - C)
\\
&= - \|\Queuelength{}(t\triangle)\|_{\infty} - 2\|\Queuelength{}(t\triangle)\|_{\infty} \left((\delta - \eps) \triangle - C - \frac{1}{2}\right) + (\delta \triangle - C)^2 
        + 2\eps \triangle(\triangle - \delta \triangle + C).
\end{aligned}
\end{equation*}
\par
Moreover, since $\|\Queuelength{}(t\triangle)\|_{\infty} \geq \Queuelength{q}(t\triangle) \geq \triangle^2$, for all $\triangle \geq \max\{\triangle_1, \triangle_2 ,\triangle_3, \triangle_4:= \frac{C+1/2}{\delta-\eps}\}$ we have
\begin{equation*}
\begin{split}
\E{\Nodedegree{}^2((t+1)\triangle) - \Nodedegree{}^2(t\triangle)|\: \Lyap{t\triangle} > L} & \leq  - \|\Queuelength{}(t\triangle)\|_{\infty}\\
& - 2(\delta - \eps) \triangle^3 + \left( \delta^2 + 2\eps(1-\delta) + 2C + 1 \right) \triangle^2 - 2(\delta - \eps)C \triangle+ C^2.
\end{split}
\end{equation*}
\par
In addition to the previously stated lower bound $\triangle \geq \max\{\triangle_1, \triangle_2 ,\triangle_3, \triangle_4\}$, if we also choose $\triangle$ such that 
\begin{equation}\label{Eq: Inequalities that k must satisfy}
- 2(\delta - \eps) \triangle^3 + \left( \delta^2 + 2\eps(1-\delta) + 2C + 1 \right) \triangle^2 - 2(\delta - \eps)C \triangle+ C^2 < 0,
\end{equation}
then  
$\E{\Lyap{(t+1)\triangle} - \Lyap{t\triangle}|\: \Lyap{t\triangle} > L} \leq  - \|\Queuelength{}(t\triangle)\|_{\infty}$. 
Such a $\triangle$ always exists because the $-2(\delta-\eps) \triangle^3$ term in \eqref{Eq: Inequalities that k must satisfy} dominates for sufficiently large $\triangle$.

Finally, if $\Lyap{t\triangle} = \Nodedegree{}^2(t\triangle) \leq L$, then %$\Nodedegree{}^2(s) \leq L$ for every $s \in [t - k\Cyclelength + 1 , t]$. Therefore,
%\begin{equation*}
    $\|\Queuelength{}(t\triangle)\|_{\infty} \leq \sqrt{L}$
%\end{equation*}
because $\|\Queuelength{}(t\triangle)\|_{\infty} \leq \Nodedegree{}(t\triangle)$. Also, $\Lyap{(t+1)\triangle} \leq (\numramps \triangle+\Nodedegree{}(t\triangle))^2 \leq (m\triangle+\sqrt{L})^2$ because the number of arrivals to all the on-ramps at each time step is at most $\numramps$. Combining this with the previously considered case of $\Lyap{t\triangle}>L$, we get
\begin{equation*}
        \E{\Lyap{(t+1)\triangle} - \Lyap{t\triangle}|\: \StateofMC(t\triangle)} \leq -\|\Queuelength{}(t\triangle)\|_{\infty} + \left((m\triangle+\sqrt{L})^2 + (1 + \sqrt{L}) \right)\mathbbm{1}_{B},
\end{equation*} 
where $B = \{\StateofMC(t\triangle): \Lyap{t\triangle} \leq L\}$ (a finite set). The result follows from Theorem~\ref{Thm: f-positivity of MC}.

\subsection*{\underline{Proof of \eqref{eq:k-min}}}
Let $w \in [t\triangle, (t+1)\triangle - \timestep_q/\timestep + 1]$. If at least one upstream vehicle crosses on-ramp $q$ in $[w, w+\timestep_q/\timestep-1)$, then \eqref{eq:k-min} obviously follows. If not, then the last acceleration lane slot of on-ramp $q$ that is not on the mainline satisfies (M1)-(M5) at some time in $[w, w+\timestep_q/\timestep-1)$. So, it must be occupied if on-ramp $q$ had non-zero quotas at the time this slot was the first acceleration lane slot. This would again give \eqref{eq:k-min}. 
%and consider all the mainline slots upstream of the merging point of on-ramp $q$ that are at most $\lceil \timestep_q/\timestep \rceil$ time steps away at time $w$. If at least one of these slots is occupied, then \eqref{eq:k-min} obviously follows. If not, consider the last acceleration lane slot of on-ramp $q$ that is not on the mainline at time $w$. This slot satisfies (M1)-(M5), and so it must be occupied at time $w$ if on-ramp $q$ had non-zero quotas at the time this slot was the first acceleration lane slot, which is at most $\numaccslots{q}$ time steps before $w$. This would again give \eqref{eq:k-min}. 
\par
We now show that on-ramp $q$ indeed had non-zero quotas at the time the aforementioned slot was the first acceleration lane slot, which is at most $\numaccslots{q}$ time steps before $w$. Let $t'$ be the start of the most recent cycle at least $\numaccslots{q}$ time steps before $t\triangle$. Considering at most one arrival per time step gives $\Queuelength{q}(t') \geq \Queuelength{q}(t\triangle) - (t\triangle-t') \geq \triangle^2 - (\numaccslots{q} + \Cyclelength)$, which is at least $\triangle + 2(\lceil\numaccslots{q}/\Cyclelength \rceil
+ 1)\Cyclelength$ if 
\begin{equation*}
    \triangle \geq 1 + \sqrt{3(\lceil\frac{\numaccslots{q}}{\Cyclelength}\rceil + 1)\Cyclelength}.
\end{equation*}
\par
Now, let $t'' \geq t'$ be the start of the most recent cycle at least $\numaccslots{q}$ time steps before $w$. Since at most one vehicle is released from on-ramp $q$ per time step, it follows that $\Queuelength{q}(t'') \geq \Queuelength{q}(t') - (t'' - t') > \numaccslots{q} + \Cyclelength$, where the last inequality follows from $t'' - t' \leq \triangle + (\lceil\numaccslots{q}/\Cyclelength \rceil+ 1)\Cyclelength$. This gives us the non-zero quota property of on-ramp $q$.  

\subsection*{\underline{Proof of \eqref{Eq: (Prop 4) Inequality for off-ramp p node degree that resembeles that of off-ramp j}}}
Consider on-ramp $i \in [\numramps]$. If during every $\timestep_i/\timestep - 1$ time steps in the interval $[t\triangle, (t+1)\triangle]$ at least one vehicle crosses on-ramp $i$, then for all $w \in [t\triangle, (t+1)\triangle-\timestep_i/\timestep+1]$ we have 
\begin{equation*}
   \Nodedegree{i}(w + \timestep_i/\timestep - 1) \leq \Nodedegree{i}(w) - 1 + \sum_{s = w+1}^{w +  \timestep_i/\timestep -1} \Numarrivals{\ell, i}(s).
%   \Nodedegree{i}(w + \lceil\timestep_i/\timestep\rceil) \leq \Nodedegree{i}(w) - 1 + \sum_{s = w+1}^{w + \lceil \timestep_i/\timestep \rceil} \Numarrivals{\ell, i}(s).
\end{equation*}
\par
Hence,
\begin{equation*}
\begin{aligned}
        \Nodedegree{i}((t+1)\triangle) &\leq \Nodedegree{i}(t\triangle) - \frac{\triangle }{\timestep_i/\timestep - 1} + 1 + \sum_{s=t\triangle+1}^{(t+1)\triangle} \Numarrivals{\ell, i}(s) \\
        &\leq \Nodedegree{}(t\triangle) - \frac{\triangle}{\timestep_i/\timestep - 1} + 1 + \sum_{s=t\triangle+1}^{(t+1)\triangle} \Numarrivals{\ell, i}(s). 
%        \Nodedegree{i}(t+1) &\leq \Nodedegree{i}(t - k\Cyclelength + 1) - \frac{k\Cyclelength}{\lceil\timestep_i/\timestep\rceil} + 1 + \sum_{s=t-k \Cyclelength+2}^{t+1} \Numarrivals{\ell, i}(s) \\
%        &\leq \Nodedegree{}(t - k\Cyclelength + 1) - \frac{k\Cyclelength}{\lceil\timestep_i/\timestep\rceil} + 1 + \sum_{s=t-k \Cyclelength+2}^{t+1} \Numarrivals{\ell, i}(s),    
\end{aligned}
\end{equation*}
\par
If not, let $s_i \in [t\triangle, (t+1)\triangle-\timestep_i/\timestep+1] $ be the last time at which a vehicle does not cross on-ramp $i$ in $[s_i, s_i + \timestep_i/\timestep-1)$, i.e., $\Nodedegree{i}(s_i + \timestep_i/\timestep-1) = \Nodedegree{i}(s_i) + \sum_{s = s_i + 1}^{s_i + \timestep_i/\timestep -1} \Numarrivals{\ell, i}(s)$. Hence, for all $w \in [s_i+1, (t+1)\triangle-\timestep_i/\timestep+1]$,  $\Nodedegree{i}(w+\timestep_i/\timestep-1) \leq \Nodedegree{i}(w) - 1 + \sum_{s = w+1}^{w +  \timestep_i/\timestep -1} \Numarrivals{\ell, i}(s)$. This further gives 
\begin{equation}\label{Eq: (Prop 4) Equality for off-ramp p node degree at time t+1}
    \Nodedegree{i}((t+1)\triangle) \leq \Nodedegree{i}(s_i + 1) - \frac{(t+1)\triangle - s_i-1}{\timestep_i/\timestep-1} + 1 + \sum_{s=s_i+2}^{(t+1)\triangle} \Numarrivals{\ell, i}(s).
%    \Nodedegree{i}(t+1) \leq \Nodedegree{i}(s_i + 1) - \frac{t - s_i}{\lceil\timestep_i/\timestep\rceil} + 1 + \sum_{s=s_i+2}^{t+1} \Numarrivals{\ell, i}(s).
\end{equation}
\par
Furthermore, we claim that the queue size at on-ramp $i$ at time $s_i + 1$ cannot exceed $\Cyclelength + \numaccslots{i}$, i.e., $\Queuelength{i}(s_i + 1) \leq \Cyclelength + \numaccslots{i}$. Note that since no upstream vehicles crosses on-ramp $i$ in $[s_i, s_i + \timestep_i/\timestep-1)$, it must be that: (i) all the mainline slots upstream of the merging point that are at most $\timestep_i/\timestep-1$ time steps away are empty at time $s_i$, (ii) the last acceleration lane slot of on-ramp $i$ that is not on the mainline at time $s_i$ is empty. Moreover, (i) implies that the aforementioned last acceleration lane slot satisfies (M1)-(M5) at time $s_i$. Therefore, it must be that on-ramp $i$ had zero quotas at the time this slot was the first acceleration lane slot, which is at most $\numaccslots{i}$ time steps before $s_i$. Since the number of on-ramp arrivals is at most one per time step, it follows that $\Queuelength{i}(s_i + 1) \leq \Cyclelength + \numaccslots{i}$. Hence,
\begin{equation*}
    \begin{aligned}
        \Nodedegree{i}(s_i + 1) &\leq \Nodedegree{i-1}(s_i + 1) + \Queuelength{i}(s_i + 1) + \numcells + \numaccslots{i} \\
        &\leq \Nodedegree{i-1}(s_i + 1) + \Cyclelength + \numcells + 2\numaccslots{i}.
    \end{aligned}
\end{equation*} 
\par
Combining this with \eqref{Eq: (Prop 4) Equality for off-ramp p node degree at time t+1} gives
\begin{equation}\label{Eq: (Prop 4) Starting inequality for off-ramp p node degree in terms of off-ramp p-1}
    \Nodedegree{i}((t+1)\triangle) \leq \Nodedegree{i-1}(s_i + 1) - \frac{(t+1)\triangle - s_i-1}{\timestep_i/\timestep-1} + 1 + \sum_{s=s_i+2}^{(t+1)\triangle} \Numarrivals{\ell, i}(s) + \Cyclelength + \numcells + 2\numaccslots{i}.
%    \Nodedegree{i}(t+1) \leq \Nodedegree{i-1}(s_i + 1) - \frac{t - s_i}{\lceil\timestep_i/\timestep\rceil} + 1 + \sum_{s=s_i+2}^{t+1} \Numarrivals{\ell, i}(s) + \Cyclelength + \numcells + 2\numaccslots{i}.
\end{equation}
\par
Similarly, if during every $\timestep_{i-1}/\timestep - 1$ time steps in the interval $[t\triangle, s_i+1]$ at least one vehicle crosses on-ramp $i-1$, then 
\begin{equation*}
        \Nodedegree{i-1}(s_i+1) \leq \Nodedegree{}(t\triangle) - \frac{s_i+1-t\triangle}{\timestep_{i-1}/\timestep - 1} + 1 + \sum_{s=t\triangle+1}^{s_i+1} \Numarrivals{\ell, i-1}(s). 
\end{equation*}
\par
Otherwise, there exists $s_{i-1} \in [t\triangle, s_i - \timestep_{i-1}/\timestep + 2]$ such that
\begin{equation*} 
    \Nodedegree{i-1}(s_i + 1) \leq \Nodedegree{i-2}(s_{i-1} + 1) - \frac{s_i - s_{i-1}}{\timestep_{i-1}/\timestep-1} + 1 + \sum_{s=s_{i-1} + 2}^{s_i+1} \Numarrivals{\ell, i-1}(s) + \Cyclelength + \numcells + 2\numaccslots{i-1}, 
%    \Nodedegree{i-1}(s_i + 1) \leq \Nodedegree{i-2}(s_{i-1} + 1) - \frac{s_i - s_{i-1}}{\lceil\timestep_{i-1}/\timestep\rceil} + 1 + \sum_{s=s_{i-1} + 1}^{s_i} \Numarrivals{\ell, i-1}(s+1) + \Cyclelength + \numcells + 2\numaccslots{i-1}, 
\end{equation*}
\par
This process can be repeated until we find an on-ramp, indexed by $i - \numramps_i$ for some $\numramps_i \in \{0\} \cup [\numramps-1]$, such that during every $\timestep_{i-\numramps_i}/\timestep - 1$ time steps in the interval $[t\triangle, s_{i-\numramps_i+1}+1]$ at least one vehicle crosses it.
%for all $w \in [t\triangle, s_{i - \numramps_i + 1} - \timestep_{i-\numramps_i}/\timestep + 2]$, $\Nodedegree{i - \numramps_i}(w + \timestep_{i-\numramps_i}/\timestep-1) \leq \Nodedegree{i - \numramps_i}(w) - 1 + \sum_{s = w+1}^{w + \timestep_{i - \numramps_i}/\timestep-1} \Numarrivals{\ell, i - \numramps_i}(s)$.
Indeed, one such on-ramp is always $q$ \footnote{This is the only place in the proof where we use the ring geometry of the road. For the straight road configuration, the process is repeated until we either find such an on-ramp, or we stop at the upstream entry point at time $s_1 > t\triangle$. Then, since $s_{1} + 1 - t\triangle \leq \triangle$ and $\Nodedegree{}(t\triangle) > \numramps \triangle^2 + \numcells + \numaccslots{a}$, we will have $\Nodedegree{1}(s_1+1) \leq \Cyclelength + 2\numaccslots{1} \leq \Nodedegree{}(t\triangle) - \frac{s_{1} + 1 - t\triangle}{\timestep_{1}/\timestep-1} +1 + \sum_{s=t\triangle+1}^{s_{1}+1} \Numarrivals{\ell, 1}(s) + \Cyclelength + 2\numaccslots{1}$. The rest of the proof will be the same.}; see the argument around \eqref{eq:k-min}. Therefore, 
\begin{equation*}
    \Nodedegree{i - \numramps_i}(s_{i - \numramps_i + 1} + 1) \leq \Nodedegree{}(t\triangle) - \frac{s_{i - \numramps_i + 1} + 1 - t\triangle}{\timestep_{i-\numramps_i}/\timestep-1} +1 + \sum_{s=t\triangle+1}^{s_{i-\numramps_i+1}+1} \Numarrivals{\ell, i - \numramps_i}(s).
%    \Nodedegree{i - \numramps_i}(s_{i - \numramps_i + 1} + 1) \leq \Nodedegree{}(t - k\Cyclelength + 1) - \frac{s_{i - \numramps_i + 1} - t + k\Cyclelength}{\lceil\timestep_{i-\numramps_i}/\timestep\rceil} +1 + \sum_{s=t - k\Cyclelength + 2}^{s_{i-\numramps_i+1}+1} \Numarrivals{\ell, i - \numramps_i}(s).
\end{equation*}
\par
By following all the inequalities involved in this process, starting from $\Nodedegree{i}((t+1)\triangle)$ on the left-hand side and ending in $\Nodedegree{}(t\triangle)$ on the right-hand side, we obtain
\begin{equation*}
        \Nodedegree{i}((t+1)\triangle) \leq \Nodedegree{}(t\triangle) + \sum_{p = 0}^{\numramps_{i}}\left[\Cumularrivals{\ell,i-p}{s_{i-p+1} - s_{i-p}} - \frac{s_{i-p+1} - s_{i-p}}{\timestep_{i-p}/\timestep -1} \right] + (\Cyclelength + \numcells + 1)m + 2\numaccslots{a},
%        \Nodedegree{i}(t+1) \leq \Nodedegree{}(t - k \Cyclelength + 1) + \sum_{p = 0}^{\numramps_{i}}\left[\Cumularrivals{\ell,i-p}{s_{i-p+1} - s_{i-p}} - \frac{s_{i-p+1} - s_{i-p}}{\lceil \timestep_{i-p}/\timestep \rceil} \right] + (\Cyclelength + \numcells + 1)m + 2\numaccslots{a},
\end{equation*}
where $s_{i+1} = (t+1)\triangle-1$, $s_{i - \numramps_i} = t\triangle-1$, and $\Cumularrivals{\ell,i}{s_{i-p+1} - s_{i-p}}$ is the cumulative number of arrivals in the interval $[s_{i-p}+2,s_{i-p+1}+1]$ that need to cross link $i$. By the assumption of the theorem, $(\timestep_{i}/\timestep-1)\Avgload{i} < 1$ for all $i \in [\numramps]$. So, there exists $\delta \in (0,1)$ such that for all $i \in [\numramps]$,
\begin{equation*}
    \Avgload{i} + \delta <  \frac{1}{ \timestep_{i}/\timestep-1}.
\end{equation*}
\par
Thus, for all $i \in [\numramps]$ we have
\begin{equation*}
        \sum_{p = 0}^{\numramps_{i}}\bigg[\Cumularrivals{\ell,i-p}{s_{i-p+1} - s_{i-p}}  - \frac{s_{i-p+1} - s_{i-p}}{ \timestep_{i-p}/\timestep-1} \bigg] \\
%        \sum_{p = 0}^{\numramps_{i}}\bigg[\Cumularrivals{\ell,i-p}{s_{i-p+1} - s_{i-p}}  - \frac{s_{i-p+1} - s_{i-p}}{\lceil \timestep_{i-p}/\timestep \rceil} \bigg] \\
        < -\delta \triangle  + \sum_{p = 0}^{\numramps_{i}}\left[\Cumularrivals{\ell,i-p}{s_{i-p+1} - s_{i-p}} - \Avgload{i - p}(s_{i-p+1} - s_{i-p}) \right].
\end{equation*}
\par
Hence, \eqref{Eq: (Prop 4) Inequality for off-ramp p node degree that resembeles that of off-ramp j} follows with
\begin{equation}\label{eq:Af-tilde-def}
    \MaxNumArrival{}{}{\triangle} = \max_{i \in [\numramps]} \sum_{p = 0}^{\numramps_{i}}\left[\Cumularrivals{\ell,i-p}{s_{i-p+1} - s_{i-p}} - \Avgload{i - p}(s_{i-p+1} - s_{i-p}) \right].
\end{equation}

\subsection*{\underline{Proof of \eqref{eq:Kolmogorov-implication}}} 
\par
Consider the sequence $\{\Numarrivals{\ell, i_s}(s)\}_{s = t\triangle+1}^{\infty}$, where the indices $i_s \in [\numramps]$ are allowed to depend on time $s$. For a given $s \in [t\triangle+1, \infty)$, the term $\Numarrivals{\ell,i_s}(s)$ is independent of the other terms in the sequence, $\Avgload{i_s} = \E{\Numarrivals{\ell,i_s}(s)}$ is bounded, and $\sigma^2_{i_s} := \E{\Numarrivals{\ell,i_s}^2(s) - \Avgload{i_s}^2}$ is (uniformly) bounded for all $i_s \in [\numramps]$. As a result, $\lim_{\triangle \ra \infty}\sum_{s = t\triangle+1}^{(t+1)\triangle}(s - t\triangle)^{-2}\sigma^2_{i_s}$ is also bounded. From Kolmogorov's strong law of large numbers (\citet[Theorem 10.12]{folland1999real}), we have, with probability $1$,
\begin{equation*}
    \lim_{\triangle \ra \infty}\frac{1}{\triangle}\left(\sum_{s = t\triangle+1}^{(t+1)\triangle}\Numarrivals{\ell,i_s}(s) - \sum_{s = t\triangle+1}^{(t+1)\triangle}\Avgload{i_s}\right) = 0.
\end{equation*}
\par
By following similar steps to the proof of \eqref{Eq: (Q-RM) implication of SLLN} in Theorem \ref{Prop: stability of QRM policy for low merging speed}, it follows for $n = 1, 2$ that
\begin{equation*}
\lim_{\triangle \ra \infty} \E{\left(\frac{\MaxNumArrival{}{}{\triangle}}{\triangle}\right)^n} = 0,
\end{equation*}
which in turn gives \eqref{eq:Kolmogorov-implication}.
\subsection{Proof of Proposition \ref{Prop: stability of distributed aFCQ-RM}}\label{Section: (Appx) Proof of distributed aFCQ-RM prop}
We show that $\StateofDySys_f(k\updateperiod) = 0$ after a finite time. This would then imply that $\Releasetime{i}(\cdot) = 0$ for all $i \in [\numramps]$ after a finite time. Thereafter, the rest of the proof follows along the lines of the proof of Theorem \ref{Prop: stability of general aFCQ-RM}. 
\par
Suppose not; then, there exists $q \in [\numramps]$ and an infinite sequence $\{k_n\}_{n \geq 1}$ such that $\StateofDySys_f^q(k_n\updateperiod) \neq 0$ for all $n \geq 1$. We prove that $\limsup_{k \ra \infty}\Releasetime{q}(k \updateperiod) = \infty$ by considering the following two cases:
\begin{enumerate}
\item[(i)] if 
\begin{equation}\label{Eq: Inequality to increase g in the dist aFCQ-RM}
    \StateofDySys_{f}^q(k\updateperiod) \leq \max\{\StateofDySys_{f}^q((k-1)\updateperiod) - \DesDecreaseconst, 0\}
\end{equation}
holds for finitely many $k$'s, then Algorithm \ref{Alg: Distributed aFCQ-RM policy} implies $\limsup_{k \ra \infty}\Releasetime{q}(k \updateperiod) = \infty$.
\item[(ii)] if \eqref{Eq: Inequality to increase g in the dist aFCQ-RM} holds for an infinite sequence, let $\{k'_n\}_{n \geq 1}$ be the sequence of all $k$'s for which \eqref{Eq: Inequality to increase g in the dist aFCQ-RM} holds. Then, there exists an infinite subsequence $\{k'_{n_\ell}\}_{\ell \geq 1}$ of $\{k'_n\}_{n \geq 1}$ for which \eqref{Eq: Inequality to increase g in the dist aFCQ-RM} does not hold at $k = k'_{n_\ell} + 1$ for all $\ell \geq 1$. If not, then there exists $M \in \N$ such that \eqref{Eq: Inequality to increase g in the dist aFCQ-RM} holds for all $k \geq k_{M}$. Since $\StateofDySys_f^q((k_{M}-1) \updateperiod)$ is bounded, this implies that $\StateofDySys_f^q(k \updateperiod) = 0$ for all $k$ sufficiently greater than $k_{M}$ -- a contradiction to $\StateofDySys_f^q$ being non-zero for the infinite sequence $\{k_n\}_{n \geq 1}$. With respect to the subsequence $\{k'_{n_\ell}\}_{\ell \geq 1}$, Algorithm \ref{Alg: Distributed aFCQ-RM policy} implies that $\ReleasetimeInc{q}((k'_{n_\ell}+1) \updateperiod) = \ReleasetimeAdj{} \ReleasetimeInc{q}(k'_{n_\ell} \updateperiod)$ for all $\ell \geq 1$. Since $\ReleasetimeInc{q}(\cdot)$ is non-decreasing and $\ReleasetimeAdj{} > 1$, this implies $\lim_{\ell \to \infty} \Releasetime{q}((k'_{n_\ell}+1) \updateperiod) \geq \lim_{\ell \to \infty} \ReleasetimeInc{q}((k'_{n_\ell}+1) \updateperiod) = \infty$. That is, $\limsup_{k \ra \infty}\Releasetime{q}(k \updateperiod) = \infty$.
\end{enumerate}

%\begin{equation}\label{Eq: Ineqality to increase theta_2 in the dist aFCQ-RM}
%        \|\StateofDySys_{j}((k + 1)\updateperiod)\| > \max\{\|\StateofDySys_{j}(k \updateperiod)\| - \DesDecreaseconst{}, 0\}
%\end{equation}
%holds for $k = k_{\ell_n}$. Suppose not; let. This and the definition of the sequence $\{k_n\}_{n \geq 1}$  imply that $\{k_{N} + 1, k_{N} + 2, \cdots\} \subset \{k_n\}_{n \geq 1}$, i.e., for all $k \geq k_N$, \eqref{Eq: Inequality to increase g in the dist aFCQ-RM} holds.   for infinitely many $k$'s. 
\par
%Let $K_q(k\updateperiod) := \max\{\Releasetime{q}(k\updateperiod) - \Tmax, 0\}/\ReleasetimeDecConst{2}$ and 
Recall from the description of the $\DisDRRRM$ policy that $\sum_{j > q-1}\StateofDySys_{f}^j(\cdot) = \StateofDySys_f(\cdot)$ for the ring road configuration. Also, recall from the proof of Theorem~\ref{Prop: stability of general aFCQ-RM} that $\StateofDySys_f(\cdot)$ is bounded by $\bar{\StateofDySys}$. Thus, if $\sum_{j > q-1}\StateofDySys_{f}^j(k\updateperiod) \leq \max\{\sum_{j > q-1}\StateofDySys_{f}^j((k-1)\updateperiod) - \DesDecreaseconst, 0\}$ for $\lceil \bar{\StateofDySys}/\DesDecreaseconst \rceil$ consecutive periods, then $\StateofDySys_f(\cdot)=0$ at the end of it. In addition, if $\StateofDySys_f(\cdot)=0$ for a large enough interval of time (say of length at least $\bar{T}$), then by a similar argument to the proof of Theorem~\ref{Prop: stability of general aFCQ-RM}, one can show that the vehicles will reach the free flow state, which would imply that $\StateofDySys_f(\cdot)=0$ thereafter -- a contradiction to $\StateofDySys_f$ being non-zero for the infinite sequence $\{k_n\}_{n \geq 1}$. Hence, in an interval of length at least $\lceil \bar{\StateofDySys}/\DesDecreaseconst \rceil\updateperiod + \bar{T}$, we must have 
\begin{equation}\label{eq:Xf-increases-once}
\sum_{j > q-1}\StateofDySys_{f}^j(k\updateperiod) > \max\{\sum_{j > q-1}\StateofDySys_{f}^j((k-1)\updateperiod) - \DesDecreaseconst, 0\}
\end{equation}
for at least one $k$. From $\limsup_{k \ra \infty}\Releasetime{q}(k \updateperiod) = \infty$ and the fact that $\Releasetime{q}(\cdot)$ decreases by at most $\ReleasetimeDecConst{2}$ in every update period, it follows that there exists time intervals of arbitrarily large length during which $\Releasetime{q}(\cdot) > \Tmax$. Combining this with the argument around \eqref{eq:Xf-increases-once}, it follows that \eqref{eq:Xf-increases-once} holds for infinitely many $k$'s while $\Releasetime{q}(\cdot) > \Tmax$. Algorithm \ref{Alg: Distributed aFCQ-RM policy} implies that for any such $k$, $\ReleasetimeInc{q-1}(k\updateperiod) = \ReleasetimeAdj{}\ReleasetimeInc{q-1}((k-1)\updateperiod)$. Therefore, $\lim_{k \ra \infty}\ReleasetimeInc{q-1}(k\updateperiod) = \infty$, which then implies that $\limsup_{k \ra \infty}\Releasetime{q-1}(k\updateperiod) = \infty$. By repeating the above arguments for other on-ramps, we can conclude that $\lim_{k \ra \infty}\ReleasetimeInc{i}(k\updateperiod) = \infty$ for all $i \in [\numramps]$.
\par
We now show that there is a time interval during which all the $\Releasetime{i}$'s are simultaneously large enough so that no on-ramp releases a vehicle for at least $\Tempty$ seconds. Let $k_0$ be such that $\Releasetime{q}(k_0\updateperiod)$ and $\ReleasetimeInc{i}(k_0\updateperiod)$ for $i \in [\numramps]$ are all greater than $\max\{\Tmax, \numramps\Tempty(1 + \ReleasetimeDecConst{2}/\updateperiod)\} + \numramps\ReleasetimeDecConst{2}(\lceil\bar{X}/\DesDecreaseconst\rceil + \lceil\bar{T}/\updateperiod\rceil)$. In the  $\lceil\bar{X}/\DesDecreaseconst\rceil + \lceil\bar{T}/\updateperiod\rceil$ periods after the time $k_0\updateperiod$, $\Releasetime{q}(\cdot)$ decreases by at most $\ReleasetimeDecConst{2}(\lceil\bar{X}/\DesDecreaseconst\rceil + \lceil\bar{T}/\updateperiod\rceil)$. During this time, $\Releasetime{q}(\cdot) > \Tmax$  and \eqref{eq:Xf-increases-once} holds at least once, which, by Algorithm~\ref{Alg: Distributed aFCQ-RM policy}, would increase $\Releasetime{q-1}(\cdot)$ by at least $\ReleasetimeAdj{}\ReleasetimeInc{q-1}(k_0\updateperiod) > \ReleasetimeInc{q-1}(k_0\updateperiod)$. Similarly, in the $\lceil\bar{X}/\DesDecreaseconst\rceil + \lceil\bar{T}/\updateperiod\rceil$ periods after $\Releasetime{q-1}(\cdot)$ is increased, $\Releasetime{q-2}(\cdot)$ increases at least once, and so on. Therefore, after at most $\numramps(\lceil\bar{X}/\DesDecreaseconst\rceil + \lceil\bar{T}/\updateperiod\rceil)$ periods after the time $k_0\updateperiod$, we have $\Releasetime{i}(k_f\updateperiod) \geq \max\{T_{\max}, \numramps\Tempty(1 + \ReleasetimeDecConst{2}/\updateperiod)\}$ for all $i \in [\numramps]$, where $k_f = k_0 + \numramps(\lceil\bar{X}/\DesDecreaseconst\rceil+ \lceil\bar{T}/\updateperiod\rceil)$. Note that for all $t$ in the interval $[k_f\updateperiod, k_f\updateperiod + \numramps \Tempty]$, we have $\Releasetime{i}(t) \geq \numramps \Tempty$, $i \in [\numramps]$. Thus, each on-ramp releases at most one vehicle during this interval, which implies that there exists a subinterval of length at least $\Tempty$ seconds during which no on-ramp releases a vehicle. Condition (VC3) implies that the vehicle will reach the free flow state at the end of such a subinterval. This and (M4) again imply that $\StateofDySys_f(\cdot)=0$ after a finite time -- a contradiction to $\StateofDySys_f$ being non-zero for the infinite sequence $\{k_n\}_{n \geq 1}$. 
\subsection{Proof of Theorem \ref{Prop: stability of rFCQ-RM}}\label{Section: (Appx) Proof of rFCQ-RM prop}
\mpcommentout{
\mpmargin{Assumed $\numcells$ is an integer.}
\mpmargin{We need some continuity assumptions on the speed, acceleration.}
}
We show that if $f_i(\cdot)$, $i=1,2,3$, is large enough, then: (i) no vehicle initially present on the freeway changes mode because of a vehicle that is released thereafter; (ii) no vehicle released after the time $t=0$ ever changes mode to the $\Safetymode$ mode. (i) ensures that all the initial vehicles reach the free flow state after at most $\Tempty$ seconds. Combining with (ii), it follows that all the vehicles reach and remain in the free flow state after at most $\Tempty$ seconds. This would imply that $\StateofDySys_g(\cdot)=0$ after a finite time, which sets the additional gaps $f_i$ to zero. The rest of the proof follows along the lines of the proof of Theorem~\ref{Prop: stability of general aFCQ-RM}.
\par
We provide a proof for (ii); the proof for (i) is similar. Suppose that (ii) does not hold, and let $t \in [0, \Tempty)$ be the first time that an ego vehicle released at time $t_0 \geq 0$ changes mode to the $\Safetymode$ mode. Without loss of generality, let $t_0 = 0$. To avoid tedious algebra and without loss of generality, we assume that the ego vehicle has merged into the mainline at some time $t_1$ at the speed $\speedlim$ and moved at this speed up to time $t$. Also, it has changed mode at time $t$ because of a leading vehicle $l$ that had been in the $\Safetymode$ mode and located on the mainline at time $t_0$. Thus, it satisfies $|v_l(\eta) - \speedlim| \leq \StateofDySys_g(\eta)$ for all $\eta \in [0,t]$. Pointers for the proof of the general case is presented at the end. 
\par 
We have 
\begin{equation}\label{Eq: Rel dist ineqeuality in the cFCQ-RM policy}
    \begin{aligned}
        y_e(t) &= |y_e(t_1) + \int_{t_1}^{t}v_l(\eta) - v_e(\eta)d\eta| 
        \geq y_e(t_1) - \int_{t_1}^{t}|v_l(\eta) - \speedlim|d\eta \\
        &\geq f_{2}(\StateofDySys(0)) - \int_{0}^{t_1}|v_l(\eta)-v_l(0)|d\eta - \int_{t_1}^{t}|v_l(\eta) - \speedlim|d\eta \\
        &\geq f_{2}(\StateofDySys(0)) - |v_l(0)-\speedlim|t_1 - \int_{0}^{t}|v_l(\eta) - \speedlim|d\eta.
    \end{aligned}
\end{equation}
\par
We show that, for sufficiently large $f_2(\cdot)$, 
\begin{equation*}
f_2(\StateofDySys(0)) - |v_l(0)-\speedlim|t_1 - \int_{0}^{t}|v_l(\eta) - \speedlim|d\eta \geq S_e(t), 
\end{equation*}
which combined with \eqref{Eq: Rel dist ineqeuality in the cFCQ-RM policy} implies that the ego vehicle does not satisfy the (VC2) criterion for changing mode to the $\Safetymode$ mode -- a contradiction. 
\par
Let $0 \leq \xi_1 \leq \ldots \leq \xi_{\ell} \leq t$ be the ``jump" time instants. That is, when a vehicle that was initially present on the freeway either: (I) leaves the freeway, or (II) changes mode. In between these jump events, e.g., $\eta \in [\xi_{j},\xi_{j+1}]$, $j \in [\ell-1]$, (VC5) implies that $\StateofDySys_g(\eta) \leq ce^{-r(\eta - \xi_{j})}\StateofDySys_g(\xi_{j})$ for some $c, r > 0$. Moreover, for all $j \in [\ell]$, in jump event (I) we have $\StateofDySys_g(\xi_j) \leq \StateofDySys_g(\xi^{-}_{j})$, and in jump event (II), if vehicle $i$ changes mode, then 
%if vehicle $i$ changes mode and $I_i(\xi^{-}_{j}) = 0$, then $\|\StateofDySys^G(\xi_j)\| = \|\StateofDySys^G(\xi^{-}_{j})\|$. Otherwise,
\begin{equation*}
%\begin{aligned}
    \StateofDySys_g(\xi_j) \leq \StateofDySys_g(\xi^{-}_{j}) + |v_i(\xi^{-}_{j}) - \speedlim| + |a_i(\xi^{-}_{j})|,
%    &= \|\StateofDySys(\xi^{-}_{j})\| + \max\left\{0, \frac{v^{2}_i(\xi^{-}_{j}) - v^2_{i+1}(\xi^{-}_{j})}{2|\minaccel|}\right\} 
%    \leq \left(1 + \frac{\overline{V}}{|\minaccel|}\right) \|\StateofDySys(\xi^{-}_{j})\|,
%\end{aligned}
\end{equation*}
where the equality holds if vehicle $i$ changes mode outside the acceleration lane of the on-ramp it has merged from. To avoid tedious algebra and without loss of generality, we assume that this is the case for all the mode changes in $[0,t]$.
%where the equality holds if vehicle $i$ is outside an acceleration lane at time $\xi^{-}_{j}$, which in turn implies that $|v_i(\xi^{-}_{j}) - \speedlim| + |a_i(\xi^{-}_{j})| = 0$. Since the merging occurs instantaneously and the merging speed is $\speedlim$, we have $\|\StateofDySys^G(\xi_j)\| = \|\StateofDySys^G(\xi^{-}_{j})\|$ in jump event (II). 
We now bound $|v_l(\eta)-\speedlim|$ in terms of $\StateofDySys_g(0)$ as follows:  for all $j \in [\ell - 1]$ and $\eta \in [\xi_{j},\xi_{j+1}]$, we have $|v_l(\eta) - \speedlim| \leq \StateofDySys_g(\eta)\leq ce^{-r(\eta - \xi_{j})}\StateofDySys_g(\xi_{j}) \leq \cdots \leq c^{j+1} e^{-r\eta}\StateofDySys_g(0)$. 
%, where $K := c^{-1}\left(c^2(1 + \frac{\overline{V}}{|\minaccel|})\right)^{k_{n(0)}}$.
Hence,
\begin{equation*}
    \begin{aligned}
        \int_{0}^{t}|v_l(\eta) - \speedlim|d\eta +  \frac{\speedlim^2 - v_l^2(t)}{2|\minaccel|} &\leq \left(\frac{1}{r}(1 - e^{-rt}) + \frac{\speedlim + \overline{V}}{2|\minaccel|} e^{-rt}\right)c^{\ell + 1}\StateofDySys_g(0) \\
        &\leq \left(\frac{1}{r} + \frac{\speedlim + \overline{V}}{2|\minaccel|} \right)c^{3n(0)}\StateofDySys_g(0),
    \end{aligned}
\end{equation*}
where $n(0)$ is the initial number of vehicles on the freeway, and in the second inequality, we have used (VC4) to bound $\ell + 1$ by $3n(0)$. Thus, by choosing
\begin{equation*}
    f_2(\StateofDySys(0)) = \left(\left(\frac{1}{r} + \frac{\speedlim + \overline{V}}{2|\minaccel|} \right)c^{3n(0)}+t_1\right)\StateofDySys_g(0),
\end{equation*}
\eqref{Eq: Rel dist ineqeuality in the cFCQ-RM policy} follows. Note that $t_1$ is the constant time between release and exiting the acceleration lane.
%Therefore, picking $a = \left(\frac{1 - e^{-r \Tempty}}{r} + \frac{\speedlim + \overline{V}}{2|\minaccel|} \right)$, $b = c^3$, $f(\StateofDySys) = \frac{v_p^2-\speedlim^2}{2|\minaccel|}$, and by setting $\tilde{f}(\StateofDySys) = ab^{n}\|\StateofDySys^G\| + f(\StateofDySys)$ gives \eqref{Eq: Rel dist ineqeuality in the cFCQ-RM policy}. 
\par
When the simplifying assumptions do not hold, additional terms are needed in \eqref{Eq: Rel dist ineqeuality in the cFCQ-RM policy} to ensure that the ego vehicle does not change mode between release and merging. Moreover, additional terms must be added to $f_2$ to account for the cases where the leading vehicle is initially not in the $\Safetymode$ mode or on the mainline. Similarly, additional terms in $f_2$ are needed to account for mode changes inside the acceleration lanes in jump event (II). All of this would result in a different expression for $f_2$ that still satisfies $f_2(\cdot)=0$ if $\StateofDySys_g(\cdot)=0$.
\subsection{Proof of Theorem \ref{Prop: Necessary condition for stability}}\label{Section: (Appx) Proof of necessary condition for regularity}
Let $t = 0, 1, \ldots$ with time steps of size $\timestep$. Without loss of generality, let the point with the long-run crossing rate no more than one be the merging point of on-ramp $i$ for all $i \in [\numramps]$. We have $\Nodedegree{i}(t) = \Nodedegree{i}(0) + \sum_{s = 1}^{t}\Numarrivals{\ell,i}(s) - \Cumuldepartures{i}(t)$, where we have dropped the dependence on the RM policy for brevity. Note that this equation holds for any point on the $i$-th link, which justifies the previous no loss in generality. 
\par
Since the arrival processes are i.i.d. across the on-ramps, the strong law of large numbers implies that for all $i \in [\numramps]$, with probability one, 
\begin{equation*}
    \liminf_{t \ra \infty} \frac{\sum_{s = 1}^{t}\Numarrivals{\ell,i}(s)}{t} = \Avgload{i},  
\end{equation*}
and hence
\begin{equation*}\label{Eq: SLLN in necessary cond for stability}
      \liminf_{t \ra \infty} \frac{\Nodedegree{i}(t)}{t} = \Avgload{i} - \limsup_{t \ra \infty} \frac{\Cumuldepartures{i}(t)}{t}.
\end{equation*}
\par
If $\Avgload{i} > 1$ for some $i \in [\numramps]$, then, with probability one, $\liminf_{t \ra \infty} \Nodedegree{i}(t)/t$ is bounded away from zero, and hence $\liminf_{t \ra \infty}\Nodedegree{i}(t) = \infty$. Thus, $\liminf_{t \ra \infty}\Queuelength{j}(t) = \infty$ for some on-ramp $j \in [\numramps]$. Combining this with Fatou's lemma imply that the average queue size grows unbounded at on-ramp $j$. This contradicts the freeway being under-saturated. 
\bibliographystyle{apalike}
\bibliography{References_main}

\end{document}